\documentclass{article}
\usepackage{graphicx,geometry,amsmath,amsthm,amssymb}
\usepackage[hypertexnames=false]{hyperref}
\usepackage{array,float}

\newtheorem{theorem}{Theorem}[section]
\newtheorem{lemma}[theorem]{Lemma}

\newtheorem{remark}[theorem]{Remark}
\newtheorem{definition}{Definition}[section]

\numberwithin{equation}{section}

\title{Extension Problems for the Vladimirov--Taibleson Operator and the Hierarchical Laplacian}
\author{Yaojia Sun}
\date{}

\begin{document}

\maketitle

\begingroup
\renewcommand{\thefootnote}{}

\begin{NoHyper}
\footnotetext{%
\par\noindent\hspace*{1.8em}%
}
\end{NoHyper}

\endgroup

\begin{abstract}
    We establish a non-Archimedean Caffarelli-Silvestre extension theory for the Vladimirov-Taibleson operator $D^s$ for $s>0$ and, more generally, for hierarchical Laplacians $L_C$. For every $f\in\mathcal{S}(\mathbb{Q}_p^n)$, we equip the Bruhat-tits tree $\mathcal{T}_{p^n}$ with the weight $w_{v_{k-1}v_k}=p^{k(s-n)}$ under the horocyclic coordinates. We prove that the resulting Dirichlet problem admits a unique bounded weighted-harmonic solution that is continuous on the end compactification. Its boundary traces converge to $f$ uniformly and in $L^2$, while the associated deformed normal derivative converges pointwise and in $L^2$ to $D^s f$. We derive explicit Fourier and Poisson representations of the extension and establish an energy identity between its weighted tree energy and the quadratic form of $D^s$. Motivated by $p$-adic AdS/CFT, we also give an equivalent massive formulation on the unweighted tree, based on a renormalized trace and a scale-corrected normal derivative. Finally, under natural local-finiteness and scaling assumptions, we extend the construction to hierarchical Laplacians $L_C$ on ultrametric spaces. We characterize the canonical edge conductances by requiring the cancellation Green operator on the ultrametric tree $\mathcal{T}_X$ to coincide with $L_C^{-1}$ on $\mathcal{S}_0(X)$. In the corresponding canonical flux class, the extension is unique, its Dirichlet-to-Neumann map is $L_C$, and it satisfies the associated energy identity. The Vladimirov-Taibleson construction is recovered as the homogeneous special case.
\end{abstract}
\section{Introduction}
In 2007, Caffarelli and Silvestre introduced an extension formulation for the fractional Laplacian \cite{MR2354493}. They showed that, for $0<s<1$ and suitable smooth functions $f:\mathbb{R}^n\to\mathbb{R}$, there exists a function $u(x,y)$ on $\mathbb{R}_+^{n+1}$ satisfying
\begin{align*}
    \begin{cases}
        \nabla\cdot(y^{1-2s}\nabla u)=0,&(x,y)\in\mathbb{R}_+^{n+1};\\
        u(x,0)=f(x),&x\in\mathbb{R}^n.
    \end{cases}
\end{align*}
The nonlocal operator $(-\Delta)^s$ is then realized as the Dirichlet-to-Neumann map for this local equation at the boundary $y=0$:
\begin{align*}
    (-\Delta)^sf(x)=-C_{n,s}\lim_{y\to0^+}y^{1-2s}u_y(x,y)=-2sC_{n,s}\lim_{y\to0^+}\frac{u(x,y)-f(x)}{y^{2s}}.
\end{align*}
Thus, the fractional Laplacian $(-\Delta)^s$ on $\mathbb{R}^n$ is recovered as a weighted normal derivative of an elliptic extension to the upper half-space $\mathbb{R}_+^{n+1}$.

This extension theory has proved particularly useful in the study of free boundary problems, including the nonlocal obstacle problem. In the same year, Silvestre obtained nearly optimal regularity for solutions of this problem \cite{MR2270163}. The following year, Caffarelli, Salsa, and Silvestre recast the $n$-dimensional fractional obstacle problem as a thin obstacle problem (the Signorini problem) in one higher dimension \cite{MR2367025}. Using a generalized Almgren-type frequency formula, they established optimal $C^{1,s}$ regularity and gave a precise account of the dichotomy between regular and singular free-boundary points. Fernández-Real and Ros-Oton later studied the critical-drift problem and showed that the homogeneity at a regular free-boundary point depends on the normal component of the drift \cite{MR3831283}.

The theory was subsequently extended to broader classes of operators and to bounded domains. In 2010, Cabré and Tan formulated an extension problem on a semi-infinite cylinder for the square root of the spectral Laplacian on a bounded domain. They also adapted the moving-plane method to the nonlocal setting and proved symmetry results for positive solutions \cite{MR2646117}. In the same year, Stinga and Torrea extended the Caffarelli--Silvestre construction to fractional powers of nonnegative self-adjoint operators, including the harmonic oscillator, using heat semigroups and integral formulas; they also obtained the corresponding Harnack inequalities \cite{MR2754080}. In 2014, Ros-Oton and Serra established a Pohozaev identity for the fractional Laplacian in which the classical normal derivative is replaced in the boundary term by $\left.\frac{u}{\delta^s}\right|_{\partial\Omega}$ \cite{MR3211861}.

Extension methods have also found geometric applications. In 2011, Chang and González placed the Caffarelli--Silvestre extension in the setting of Poincaré--Einstein asymptotically hyperbolic manifolds, thereby providing a PDE approach to the fractional Yamabe problem \cite{MR2737789}. Frank et al. later formulated an extension problem for the CR fractional Laplacian on the Heisenberg group. Their construction is related to the Siegel realization of complex hyperbolic space and preserves the conformal covariance of the boundary operator \cite{MR3286532}.

Subsequent work addressed equations with rough coefficients and operators that are nonlocal in both space and time. In 2016, Caffarelli and Stinga obtained Caccioppoli-type inequalities and Schauder estimates for fractional powers of uniformly elliptic divergence-form operators with bounded measurable coefficients \cite{MR3489634}. The following year, Stinga and Torrea constructed a degenerate parabolic extension equation involving the time derivative and proved the corresponding space-time Harnack estimates \cite{MR3709888}. More recently, Otárola and Salgado approximated the higher-order powers $(-\Delta)^s$, $s\in(1,2)$, of the spectral fractional Laplacian using a polyharmonic extension and a mixed finite-element discretization \cite{MR5081849}.

Nearly forty years before the PDE formulation of Caffarelli and Silvestre, Molchanov and Ostrovskii had studied the localization of nonlocal operators from a probabilistic perspective. They realized symmetric stable processes as boundary traces of higher-dimensional degenerate diffusion processes \cite{MR247668}. Their work provides probabilistic motivation for later tree-based extensions on totally disconnected spaces.

Because $\mathbb{Q}_p^n$ is totally disconnected, the classical differential operators of Euclidean analysis have no direct counterpart in this setting. In 1990, Vladimirov studied the Fourier-defined pseudodifferential operator $D^s$, which serves as the $p$-adic analog of the fractional Laplacian \cite{MR1092528}.

For $\mathbb{Q}_p$, the Bruhat--Tits tree $\mathcal{T}_p$ serves as a discrete hyperbolic space---the $p$-adic upper half-plane---in arithmetic geometry and $p$-adic holography. It is therefore natural to seek a Caffarelli--Silvestre-type extension of the Vladimirov derivative on $\mathcal{T}_p$. In his 1989 work on $p$-adic string theory, Zabrodin realized the operator $D$ via a discrete harmonic extension problem on $\mathcal{T}_p$. This construction became a prototype for the subsequent $p$-adic AdS/CFT correspondence \cite{MR1003429}. More precisely, he considered the Dirichlet problem
\begin{align*}
    \begin{cases}
        \Delta\varphi(v)=0,&v\in\mathcal{T}_p;\\
        \varphi|_{\partial\mathcal{T}_p}(x)=f(x),&x\in\partial\mathcal{T}_p\cong\mathbb{P}^1(\mathbb{Q}_p).
    \end{cases}
\end{align*}
He then defined the boundary normal derivative by
\begin{align*}
    \partial_n^{(p)}\varphi(x)=\lim_{v\to x}(f(x)-\varphi(v))p^{d(C,v)}.
\end{align*}
For $s=1$, he proved that the boundary response satisfies
\begin{align*}
    \partial_n^{(p)}\varphi(x)=p^{-1}\max\{1,|x|_p^2\}Df(x),
\end{align*}
together with the energy identity
\begin{align*}
    \frac{1}{2}\sum_{\substack{u,v\in\mathcal{T}_p\\u\sim v}}(\varphi(u)-\varphi(v))^2=(1-p^{-1})\int_{\mathbb{Q}_p}f(x)Df(x)dx.
\end{align*}
More recently, Huang and Sun extended this construction to arbitrary $s>0$ by considering the weighted graph Laplacian problem \cite{huang2026energyrelationsgeneralisedvladimirov}:
\begin{align*}
    \begin{cases}
        \Delta_w\varphi(v)=0,&v\in\mathcal{T}_p;\\
        \varphi|_{\partial\mathcal{T}_p}(x)=f(x),&x\in\partial\mathcal{T}_p\cong\mathbb{P}^1(\mathbb{Q}_p).
    \end{cases}
\end{align*}
They defined the deformed boundary normal derivative by
\begin{align*}
    \partial_{n,s}^{(p)}\varphi(x)=\lim_{v\to x}(f(x)-\varphi(v))p^{d(C,v)s},
\end{align*}
where
\begin{align*}
    w_{uv}=\frac{1}{2}\left(p^{d(C,u)(s-1)}+p^{d(C,v)(s-1)}\right),
\end{align*}
for adjacent vertices $u\sim v$, with $C$ denoting the base point in $\mathcal{T}_p$. They proved that
\begin{align*}
    \partial_{n,s}^{(p)}\varphi(x)=\int_{\partial\mathcal{T}_p}\left(\frac{C_1}{|x,y|_p^{1+s}}+C_2\right)(f(x)-f(y))d\mu_0(y),
\end{align*}
where $C_1=\frac{p^s-1}{p^s-p^{-1}}$ and $C_2=\frac{p-p^s}{(p+1)(p^s-p^{-1})}$, and established the energy identity
\begin{align*}
    \frac{1}{2}\sum_{\substack{u,v\in\mathcal{T}_p\\u\sim v}}(\varphi(u)-\varphi(v))^2w_{uv}=&\frac{(1-p^{-s})(1+p^{1-s})}{2}\int_{\mathbb{Q}_p}a(x)f(x)D^s\left(a(x)f(x)\right)dx\\
    &+\frac{(1-p^{-s})(p^{s-1}-p^{1-s})}{2(1+p^{-1})(1-p^{-1-s})}
    \left(\int_{\partial\mathcal{T}_p}f(x)d\mu_0(x)\right)^2,
\end{align*}
where 
\begin{align*}
    a(x)=\begin{cases}
        1,&x\in\mathbb{Z}_p;\\
        |x|_p^{s-1},&x\in\mathbb{Q}_p\setminus\mathbb{Z}_p.
        \end{cases}.
\end{align*}
Here $\mu_0$ is a $GL(2,\mathbb{Z}_p)$-invariant measure on the boundary $\partial\mathcal{T}_p$. For
\begin{align*}
    B_v=\{u\in\mathcal{T}_p\mid\text{there is a path }C\to v\to u\},
\end{align*}
the boundary shadow $\partial B_v$ is normalized by
\begin{align*}
    \mu_0(\partial B_v)=p^{-d(C,v)}.
\end{align*}
The corresponding visual ultrametric is
\begin{align*}
    |x,y|_p=p^{-\delta(C\to x,C\to y)},
\end{align*}
where $\delta(C\to x,C\to y)$ denotes the length of the common segment of the paths from $C$ to $x$ and from $C$ to $y$.
\begin{remark}
    Zabrodin's treatment of the case $s=1$ uses radial coordinates based at a fixed vertex $C\in\mathcal{T}_p$. The general-$s$ construction above uses the same radial normalization, which is not adapted to additive translations on the affine boundary $\mathbb{Q}_p$ and therefore makes Fourier analysis less direct. In affine coordinates, its boundary response is not simply the translation-invariant operator $D^s$; moreover, for general $s$, the energy contains the weight $a(x)$ and the additional mean term $\left(\int_{\partial\mathcal{T}_p}f\,d\mu_0\right)^2$. The horocyclic normalization adopted below preserves affine translation symmetry.
\end{remark}

In higher dimensions, $PGL(n+1,\mathbb{Q}_p)$ acts on an $n$-dimensional affine Bruhat--Tits building. For the isotropic operator considered here, however, only the maximum norm $\|\boldsymbol{x}\|_p$ is relevant, and the corresponding nested maximum-norm balls form a regular $p^n$-ary tree.

In 2017, Gubser, Knaute, Parikh, Samberg, and Witaszczyk constructed a $p$-adic AdS/CFT model in which a $(p^n+1)$-regular Bruhat--Tits tree replaces the Euclidean AdS bulk and has finite boundary $\mathbb{Q}_{p^n}$, the unramified degree-$n$ extension of $\mathbb{Q}_p$ \cite{MR3631398}. The resulting boundary correlation functions resemble their Archimedean AdS/CFT counterparts and exhibit the scaling behavior of Vladimirov--Taibleson-type operators. After fixing the norm- and measure-compatible identification described below, we regard the finite boundary $\mathbb{Q}_{p^n}$ as $\mathbb{Q}_p^n$. This motivates our use of the regular tree $\mathcal{T}_{p^n}$, rather than the full affine building, as the extension space for the isotropic maximum-norm operator. Pierce, Rajkumar, Stine, Weisbart, and Yassine later realized local-field Brownian motion as a scaling limit of discrete random walks using martingale and stochastic-process methods \cite{MR4808793}.

Using the horocyclic coordinates of \cite{MR3631398}, we follow the variational approach of \cite{MR2354493}: we define the energy functional and derive its Euler--Lagrange equation. We then establish the Caffarelli--Silvestre extension for the Vladimirov--Taibleson operator by two complementary methods, one based on the Fourier transform and the other on the Poisson integral formula. The first main result is the following.

\begin{theorem}[Caffarelli--Silvestre extension for the Vladimirov--Taibleson operator]
\label{thm:Extension problems related to the Vladimirov-Taibleson operator}
Let $f\in \mathcal{S}(\mathbb{Q}_p^n)$ and $s>0$. Consider the following extension problem
\begin{align*}
    \begin{cases}
    \Delta_w\varphi(v)=0,&v\in \mathcal{T}_{p^n};\\
    \displaystyle\lim_{k\to+\infty}\varphi_k(\boldsymbol{x})=f(\boldsymbol{x}),&\boldsymbol{x}\in\mathbb{Q}_p^n;\\
    \displaystyle\lim_{k\to-\infty}\varphi_k=0.&
\end{cases}
\end{align*}
The last boundary condition is understood as an end limit: $\varphi(v)\to0$ along every sequence of vertices $v\to\infty$. Here $w_{v_{k-1},v_k}=p^{k(s-n)}$ and, in horocyclic coordinates, $v_k=(\boldsymbol{x},k)$ with $k\in\mathbb{Z}$ and $\boldsymbol{x}\in\mathbb{Q}_p^n/p^k\mathbb{Z}_p^n$. This problem has a unique bounded weighted-harmonic solution that is continuous on the end compactification. Moreover, $\varphi_k\to f$ uniformly and in $L^2(\mathbb{Q}_p^n)$ as $k\to+\infty$. Define the deformed boundary normal derivative by
\begin{align*}
        \partial_{n,s}^{(p)}\varphi(\boldsymbol{x})=\lim_{k\to+\infty}p^{(k+1)s}\left(f(\boldsymbol{x})-\varphi_k(\boldsymbol{x})\right),
    \end{align*}
where the limit exists pointwise and in $L^2(\mathbb{Q}_p^n)$. The resulting Dirichlet-to-Neumann map is the Vladimirov--Taibleson operator:
    \begin{align*}
    \partial_{n,s}^{(p)}\varphi(\boldsymbol{x})=D^sf(\boldsymbol{x}).
\end{align*}
The solution also satisfies the energy identity
\begin{align*}
    \mathcal{E}[\varphi]=&\sum_{k\in\mathbb{Z}}p^{ks}\int_{\mathbb{Q}_p^n}|\varphi_k(\boldsymbol{x})-\varphi_{k-1}(\boldsymbol{x})|^2d\boldsymbol{x}\\
    =&\frac{1}{2}\sum_{u\in \mathcal{T}_{p^n}}\sum_{v\sim u}w_{uv}|\varphi(u)-\varphi(v)|^2\\
    =&(1-p^{-s})\int_{\mathbb{Q}_p^n}f(\boldsymbol{x})D^sf(\boldsymbol{x})d\boldsymbol{x}.
\end{align*}
\end{theorem}

The holographic construction in \cite{MR3631398} suggests the following alternative extension.

\begin{theorem}[Alternative extension for the Vladimirov--Taibleson operator]
    \label{thm:another_Extension_method_related_to_the_Vladimirov-Taibleson_operator}
    Let $f\in \mathcal{S}(\mathbb{Q}_p^n)$ and $\alpha>\frac{n}{2}$. Consider the following extension problem
    \begin{align*}
    \begin{cases}
        (\Delta-m_p^2)\varphi(v)=0,&v\in\mathcal{T}_{p^n};\\
        \displaystyle\lim_{k\to+\infty}p^{(n-\alpha)k}\varphi_k=f,&\text{in }L^2(\mathbb{Q}_p^n).
    \end{cases}
\end{align*}
Here $m_p^2=p^{\alpha}+p^{n-\alpha}-p^n-1$ is the mass-squared spectral parameter; despite the notation, it need not be nonnegative. This problem has a unique solution within the class of layer functions $\varphi_k\in L^2(\mathbb{Q}_p^n)$ that are constant on the cosets of $p^k\mathbb{Z}_p^n$. For Bruhat--Schwartz data, the canonical locally constant representatives of the renormalized layers converge uniformly. If $v_R(\boldsymbol{x})$ denotes the unique height-$R$ vertex on the ray ending at $\boldsymbol{x}$, then
\begin{align*}
    f(\boldsymbol{x})=\displaystyle\lim_{R\to+\infty}
    p^{R(n-\alpha)}\varphi(v_R(\boldsymbol{x})).
\end{align*}
Define the scale-corrected boundary normal derivative by
\begin{align*}
    \partial_{n,\alpha}\varphi(\boldsymbol{x})=\lim_{k\to+\infty}\left(f(\boldsymbol{x})-p^{(n-\alpha)k}\varphi_k(\boldsymbol{x})\right)p^{(k+1)(2\alpha-n)},
\end{align*}
where the limit is taken in $L^2(\mathbb{Q}_p^n)$. This boundary operator again recovers the Vladimirov--Taibleson operator:
\begin{align*}
    \partial_{n,\alpha}\varphi(\boldsymbol{x})=D^{2\alpha-n}f(\boldsymbol{x}).
\end{align*}
Here the equality holds in $L^2(\mathbb{Q}_p^n)$. To define the boundary term at a cutoff $R$, truncate the solution at level $R$ and assign the value $p^{\alpha-n}\varphi(v_R)$ to each of the $p^n$ ghost children of $v_R$. This artificial continuation is used solely to define the cutoff residual
\begin{align*}
    \mathcal{R}_R\varphi(v_R):=\varphi(v_{R-1})-p^{n-\alpha}\varphi(v_R).
\end{align*}
With this convention, the following identity holds:
    \begin{align*}
        \lim_{R\to+\infty}\sum_{v_R\in\partial\mathcal{T}_{p^n}^{R}}\varphi(v_R)\mathcal{R}_R\varphi(v_R)=p^{n-2\alpha}(p^{n-\alpha}-p^{\alpha})\int_{\mathbb{Q}_p^n}f(\boldsymbol{x})D^{2\alpha-n}f(\boldsymbol{x})d\boldsymbol{x}.
    \end{align*}
\end{theorem}

We now turn to hierarchical Laplacians. Throughout this part, let $(X,d)$ be a noncompact, locally compact, separable, and perfect ultrametric space with associated ultrametric tree $\mathcal{T}_X$. Let $\mathcal{B}$ be a countable basis of nonempty compact open balls, and let $m$ be a non-atomic Borel Radon measure of full support such that $m(X)=+\infty$ and $0<m(B)<+\infty$ for every $B\in\mathcal{B}$. We assume that every $B\in\mathcal{B}$ has a unique immediate super-ball $B'$, that every ball is the disjoint union of finitely many maximal proper sub-balls, that any two balls are contained in a common ball, and that every $x\in X$ admits a bi-infinite chain
\begin{align*}
    \cdots\supseteq B_{k-1}(x)\supseteq B_k(x)\supseteq B_{k+1}(x)\supseteq\cdots\ni x,
\end{align*}
that shrinks to $\{x\}$ as $k\to+\infty$ and exhausts $X$ as $k\to-\infty$. The hierarchical Laplacian is defined by
    \begin{align*}
        L_Cf(x)=\sum_{\substack{B\in\mathcal{B}\\B\ni x}}C(B)\left(f(x)-\frac{1}{m(B)}\int_{B}f(y)dm(y)\right),
    \end{align*}
where the choice function $C:\mathcal{B}\to(0,+\infty)$ satisfies
\begin{align*}
    \lambda(B):=\sum_{\substack{D\in\mathcal{B}\\D\supseteq B}}C(D)<+\infty,
    \qquad
    \lambda(B_k(x))\longrightarrow+\infty\quad(k\to+\infty).
\end{align*}
Here the first condition holds for every $B\in\mathcal{B}$, and the second holds for every $x\in X$ \cite{MR3269722}. The function $\lambda$ gives the eigenvalue scale; under the convention adopted below, $\lambda(B')$ is the eigenvalue of $L_C$ on the local Haar space $\operatorname{span}W_{B'}$. We determine the extension weights by matching the inverse $L_C^{-1}$ on
\begin{align*}
    \mathcal{S}_0(X)=\left\{f\in \mathcal{S}(X)\bigg|\ \int_Xf(x)dm(x)=0\right\}
\end{align*}
with the cancellation Green operator induced by the graph weight. The third main result is the following.

\begin{theorem}[Caffarelli--Silvestre extension for the hierarchical Laplacian]
\label{thm:Extension_problems_related_to_the_hierarchical_Laplacian}
Let $f\in \mathcal{S}_0(X)$ and assume the standing hypotheses above.
Suppose that
\begin{align*}
    \kappa:=\lim_{k\to+\infty}\frac{\lambda(B_k(x))}{C(B_k(x))}
\end{align*}
exists, belongs to $(0,+\infty)$, and is independent of $x$, with convergence uniform on compact subsets of $X$. Consider the following extension problem
    \begin{align*}
        \begin{cases}
        \Delta_W\varphi(v_B)=0,&v_B\in\mathcal{T}_X;\\
        \varphi|_{\partial\mathcal{T}_X}(x)=f(x),&x\in X;\\
        \varphi(\infty)=0.
    \end{cases}
    \end{align*}
    Here $W(B,B')=\frac{1}{\kappa}\frac{m(B)\lambda(B)\lambda(B')}{C(B)}$, where $B'$ is the unique immediate super-ball of $B$. The problem has a unique solution in the canonical class of harmonic functions with an $L^2(X,m)$ trace and a flux
    \begin{align*}
        \Phi(B)=W(B,B')\left(\varphi(v_B)-\varphi(v_{B'})\right)
    \end{align*}
    of the form
    \begin{align*}
        \Phi(B)=\int_Bg_\varphi(y)dm(y),
        \qquad g_\varphi\in \mathcal{S}_0(X).
    \end{align*}
    For this solution, $g_\varphi=\kappa^{-1}L_Cf$. Given a chain $\cdots\supseteq B_{k-1}(x)\supseteq B_k(x)\supseteq B_{k+1}(x)\supseteq\cdots\ni x$, define the deformed normal derivative at $x\in\partial\mathcal{T}_X\setminus\{\infty\}$ by
    \begin{align*}
        \partial_n^{(W)}\varphi(x)=\lim_{k\to+\infty}\frac{W(B_k,B_{k-1})}{m(B_k)}(f(x)-\varphi(v_{B_{k-1}})),
    \end{align*}
    where the limit exists in $L^2(X,m)$. The corresponding Dirichlet-to-Neumann map is the hierarchical Laplacian:
    \begin{align*}
        \partial_n^{(W)}\varphi(x)=L_Cf(x).
    \end{align*}
    Here the equality holds in $L^2(X,m)$.
    The solution also satisfies the energy identity
\begin{align*}
    \frac{1}{2}\sum_{\substack{v_A,v_B\in \mathcal{T}_{X}\\v_A\sim v_B}}|\varphi(v_A)-\varphi(v_B)|^2W(A,B)
    =&\frac{1}{\kappa}\int_Xf(x)L_Cf(x)dm(x).
\end{align*}
\end{theorem}
\section{Preliminaries}
\subsection{\texorpdfstring{The $p$-adic field and the Fourier transform}{The p-adic field and the Fourier transform}}
Let $p$ be a prime. Every nonzero $x\in\mathbb{Q}$ can be written as $x=p^{v}\frac{a}{b}$, where $v\in\mathbb{Z}$ and $a,b\in\mathbb{Z}$ satisfy $p\nmid ab$. Define $|x|_p=p^{-v}$ and set $|0|_p=0$. The completion of $\mathbb{Q}$ with respect to this absolute value is the $p$-adic field $\mathbb{Q}_p$. The absolute value extends uniquely to $\mathbb{Q}_p$ and satisfies the strong triangle inequality
\begin{align*}
    |x+y|_p\leq\max\{|x|_p,|y|_p\}.
\end{align*}
Let $v_p(x)=-\log_p|x|_{p}$ for $x\neq0$ and define $v_p(0)=+\infty$. Then each nonzero $x\in\mathbb{Q}_p$ has a unique $p$-adic expansion
\begin{align*}
    x=\sum_{k=v_p(x)}^{+\infty}a_kp^k,
\end{align*}
where $a_{v_p(x)}\in\mathbb{F}_p^{\times}$ and $a_k\in\mathbb{F}_p$ for $k>v_p(x)$; here we use the usual representatives of $\mathbb{F}_p$ in $\{0,\ldots,p-1\}$.

The ring of $p$-adic integers $\mathbb{Z}_p$ is defined by
\begin{align*}
    \mathbb{Z}_p
    :=\{x\in\mathbb{Q}_p\mid v_p(x)\geq0\}
    =\{x\in\mathbb{Q}_p\mid |x|_p\leq1\}.
\end{align*}
The group of units of $\mathbb{Z}_p$ is denoted by
\begin{align*}
    \mathbb{Z}_p^*
    =\{x\in\mathbb{Q}_p\mid |x|_p=1\}
    =\mathbb{Z}_p\setminus p\mathbb{Z}_p.
\end{align*}
Let $\chi_p(x)=e^{2\pi i\{x\}_p}$ be the standard additive character on $\mathbb{Q}_p$, where $\{\cdot\}_p$ is the $p$-adic fractional part defined by
\begin{align*}
    \{x\}_p=
    \begin{cases}
        0,&v_p(x)\geq0,\\
        \displaystyle\sum_{k=v_p(x)}^{-1}a_kp^k,&v_p(x)<0.
    \end{cases}
\end{align*}
For any $f\in L^1(\mathbb{Q}_p)$, the Fourier transform is defined by
\begin{align*}
    \hat{f}(\xi)=\int_{\mathbb{Q}_p}f(x)\chi_p(x\xi)dx,
\end{align*}
where $dx$ is the Haar measure on $\mathbb{Q}_p$ normalized by
$|\mathbb{Z}_p|=1$.

The $n$-dimensional vector space $\mathbb{Q}_p^n$ is equipped with the supremum norm, defined for
$\boldsymbol{x}=(x_1,\dots,x_n)\in\mathbb{Q}_p^n$ by
\begin{align*}
    \|\boldsymbol{x}\|_p=\max_{1\leq i\leq n}\{|x_i|_p\},
\end{align*}
which also satisfies the strong triangle inequality
\begin{align*}
    \|\boldsymbol{x}+\boldsymbol{y}\|_p
    \leq\max\{\|\boldsymbol{x}\|_p,\|\boldsymbol{y}\|_p\}.
\end{align*}
The space $\mathbb{Q}_p^n$ is complete with respect to the norm $\|\cdot\|_p$. For later use, we set
\begin{align*}
    v_p(\boldsymbol{x})=\min_{1\leq i\leq n}v_p(x_i),
    \qquad v_p(\boldsymbol{0})=+\infty.
\end{align*}
Thus $\|\boldsymbol{x}\|_p=p^{-v_p(\boldsymbol{x})}$ for $\boldsymbol{x}\neq\boldsymbol{0}$.

Let $\boldsymbol{x}\cdot\boldsymbol{\xi} =\displaystyle\sum_{i=1}^{n}x_i\xi_i$. We use the product Haar measure $d\boldsymbol{x}$ on $\mathbb{Q}_p^n$, normalized by $|\mathbb{Z}_p^n|=1$. For every $f\in L^1(\mathbb{Q}_p^n)$, its Fourier transform is defined by
\begin{align*}
    \hat{f}(\boldsymbol{\xi})
    =\int_{\mathbb{Q}_p^n}f(\boldsymbol{x})
      \chi_p(\boldsymbol{x}\cdot\boldsymbol{\xi})d\boldsymbol{x}.
\end{align*}
With this convention, the inverse Fourier transform is
\begin{align*}
    f(\boldsymbol{x})
    =\int_{\mathbb{Q}_p^n}\hat{f}(\boldsymbol{\xi})
      \chi_p(-\boldsymbol{x}\cdot\boldsymbol{\xi})
      d\boldsymbol{\xi},
\end{align*}
whenever the classical inversion formula applies. The Bruhat--Schwartz space $\mathcal{S}(\mathbb{Q}_p^n)$ consists of locally constant, compactly supported complex-valued functions. The Fourier transform maps $\mathcal{S}(\mathbb{Q}_p^n)$ onto itself and extends to a unitary operator on $L^2(\mathbb{Q}_p^n)$. In particular, Plancherel's identity takes the form
\begin{align*}
    \int_{\mathbb{Q}_p^n}f(\boldsymbol{x})
      \overline{g(\boldsymbol{x})}d\boldsymbol{x}
    =\int_{\mathbb{Q}_p^n}\hat f(\boldsymbol{\xi})
      \overline{\hat g(\boldsymbol{\xi})}d\boldsymbol{\xi}.
\end{align*}

\subsection{The Vladimirov--Taibleson operator and its spectrum}
\begin{definition}
    Let $\alpha>0$. The Vladimirov--Taibleson operator $D^\alpha$ is the nonnegative self-adjoint operator on $L^2(\mathbb{Q}_p^n)$ defined by
    \begin{align*}
        \widehat{D^\alpha f}(\boldsymbol{\xi})
        =\|\boldsymbol{\xi}\|_p^{\alpha}\hat{f}(\boldsymbol{\xi}),
    \end{align*}
    with operator domain
    \begin{align*}
        \mathcal{D}(D^\alpha)
        =\left\{f\in L^2(\mathbb{Q}_p^n)\ \middle|\
        \|\boldsymbol{\xi}\|_p^\alpha\hat f(\boldsymbol{\xi})
        \in L^2(\mathbb{Q}_p^n)\right\}.
    \end{align*}
\end{definition}
\begin{remark}
    For $f\in \mathcal{S}(\mathbb{Q}_p^n)$, the preceding Fourier-multiplier definition is equivalently given by the absolutely convergent singular integral
    \begin{align*}
        D^\alpha f(\boldsymbol{x})
        =\frac{1-p^{\alpha}}{1-p^{-\alpha-n}}
        \int_{\mathbb{Q}_p^n}
        \frac{f(\boldsymbol{y})-f(\boldsymbol{x})}
        {\|\boldsymbol{y}-\boldsymbol{x}\|_p^{\alpha+n}}
        d\boldsymbol{y}.
    \end{align*}
\end{remark}
\begin{remark}
    The Vladimirov--Taibleson operator may be regarded as the non-Archimedean counterpart of the fractional Laplacian $(-\Delta)^{\frac{\alpha}{2}}$ in $\mathbb{R}^n$, whose Fourier symbol is $\|\boldsymbol{\xi}\|^\alpha$.
\end{remark}
\begin{theorem}
    The functions
    \begin{align*}
        \Psi_{\gamma,\boldsymbol{a},\boldsymbol{j}}(\boldsymbol{x})
        =p^{\frac{n(\gamma-1)}{2}}
        \chi_p\!\left(p^{-\gamma}\boldsymbol{j}\cdot
        (\boldsymbol{x}-\boldsymbol{a})\right)
        \boldsymbol{1}_{\boldsymbol{a}
        +p^{\gamma-1}\mathbb{Z}_p^n}(\boldsymbol{x}),
    \end{align*}
    where $\gamma\in\mathbb{Z}$, $\boldsymbol{a}$ ranges over a fixed set of representatives of $\mathbb{Q}_p^n/p^{\gamma-1}\mathbb{Z}_p^n$, and $\boldsymbol{j}\in\mathbb{F}_p^n\setminus\{\boldsymbol{0}\}$, form a complete orthonormal basis of $L^2(\mathbb{Q}_p^n)$ and satisfy
    \begin{align*}
        D^\alpha\Psi_{\gamma,\boldsymbol{a},\boldsymbol{j}}
        =p^{\gamma\alpha}
        \Psi_{\gamma,\boldsymbol{a},\boldsymbol{j}}.
    \end{align*}
    Consequently,
    \begin{align*}
        \operatorname{Spec}(D^\alpha)
        =\{0\}\cup\{p^{\gamma\alpha}\mid\gamma\in\mathbb{Z}\}.
    \end{align*}
    Every positive eigenvalue has infinite multiplicity. The point $0$ belongs to the spectrum as an accumulation point, but it is not an $L^2$-eigenvalue.
\end{theorem}
\begin{proof}
    The support of each wavelet has measure $p^{-n(\gamma-1)}$, so the factor $p^{\frac{n(\gamma-1)}{2}}$ gives unit norm. On each fixed ball $\boldsymbol{a}+p^{\gamma-1}\mathbb{Z}_p^n$, the characters indexed by $\boldsymbol{j}\in\mathbb{F}_p^n$ form the finite Fourier basis on the set of $p^n$ maximal sub-balls. Omitting $\boldsymbol{j}=\boldsymbol{0}$ leaves the mean-zero subspace. Wavelets supported on disjoint balls are orthogonal, while wavelets at different scales are orthogonal because a coarser wavelet is constant on the support of a finer mean-zero wavelet. The nested partitions of $\mathbb{Q}_p^n$ into balls therefore show that these wavelets form a complete orthonormal basis of $L^2(\mathbb{Q}_p^n)$.

    With the Fourier convention fixed above, a direct change of variables gives
    \begin{align*}
        \operatorname{supp}
        \widehat{\Psi}_{\gamma,\boldsymbol{a},\boldsymbol{j}}
        =-p^{-\gamma}\boldsymbol{j}
        +p^{1-\gamma}\mathbb{Z}_p^n.
    \end{align*}
    Since $\boldsymbol{j}\neq\boldsymbol{0}$ in $\mathbb{F}_p^n$, one has $\|\boldsymbol{\xi}\|_p=p^\gamma$ throughout this support. The Fourier-multiplier definition of $D^\alpha$ now yields the stated eigenvalue.

    Finally, the spectrum of a multiplication operator is the essential range of its multiplier. In the present case, this range is $\{0\}\cup\{p^{\gamma\alpha}\mid\gamma\in\mathbb{Z}\}$. The singleton $\{\boldsymbol{0}\}$ has Haar measure zero, so the multiplier has trivial kernel in $L^2$. In particular, the constant function $1$ is not an $L^2$-eigenfunction because $|\mathbb{Q}_p^n|=+\infty$.
\end{proof}

\subsection{Bruhat--Tits trees and horocyclic coordinates}
The Bruhat--Tits tree
$\mathcal{T}_p=PGL(2,\mathbb{Q}_p)/PGL(2,\mathbb{Z}_p)$ is a $(p+1)$-regular tree whose boundary satisfies
\begin{align*}
    \partial \mathcal{T}_p
    \cong\mathbb{P}^1(\mathbb{Q}_p)
    =\mathbb{Q}_p\cup\{\infty\}.
\end{align*}
Fixing the point at infinity
$\infty\in\partial\mathcal{T}_p$ identifies
$\partial\mathcal{T}_p\setminus\{\infty\}$ with the $p$-adic field.
The tree $\mathcal{T}_p$ admits a natural stratification by its
horocycles, which are the level sets of a Busemann function.
We use the following horocyclic coordinates.
\begin{definition}[Horocyclic coordinates on $\mathcal{T}_p$]
    Each vertex $v\in \mathcal{T}_p$ is parameterized by a pair
    $(x,k)$, where
    \begin{enumerate}
        \item The scale height is $k\in\mathbb{Z}$. Along a ray to a
        point of $\mathbb{Q}_p$, the height tends to $+\infty$, whereas
        along the distinguished ray to $\infty$, it tends to
        $-\infty$.
        \item The horizontal position is
        $x\in\mathbb{Q}_p/p^k\mathbb{Z}_p$. For a fixed height $k$, the
        set of vertices
        $\{v=(x,k)\in\mathcal{T}_p\mid
        x\in\mathbb{Q}_p/p^k\mathbb{Z}_p\}$ forms a horizontal slice of
        the tree, and
        \begin{align*}
            \{y\in\partial \mathcal{T}_p\mid
            \text{there is a path }\infty\to(x,k)\to y\}
            =x+p^k\mathbb{Z}_p.
        \end{align*}
    \end{enumerate}
\end{definition}
In these coordinates, each vertex
$v=(x,k)\in\mathcal{T}_p$ has a unique upper neighbor $(x,k-1)$,
closer to $\infty$, and $p$ lower neighbors
$(x+ap^k,k+1)$ with $a\in\mathbb{F}_p$, in the direction of
$\mathbb{Q}_p$.

Let $\mathbb{Q}_{p^n}$ be the unramified extension of
$\mathbb{Q}_p$ of degree $n$, let $\mathbb{Z}_{p^n}$ be its ring of
integers, and normalize its absolute value by $|p|_{p^n}=p^{-1}$.
Consider the $(p^n+1)$-regular Bruhat--Tits tree
\begin{align*}
    \mathcal{T}_{p^n}
    =PGL(2,\mathbb{Q}_{p^n})/PGL(2,\mathbb{Z}_{p^n}),
\end{align*}
whose boundary satisfies
\begin{align*}
    \partial \mathcal{T}_{p^n}
    \cong\mathbb{P}^1(\mathbb{Q}_{p^n})
    =\mathbb{Q}_{p^n}\cup\{\infty\}.
\end{align*}
Fix a $\mathbb{Z}_p$-basis $e_1,\ldots,e_n$ of
$\mathbb{Z}_{p^n}$. The map
\begin{align*}
    (x_1,\ldots,x_n)\longmapsto\sum_{i=1}^n x_i e_i
\end{align*}
is a $\mathbb{Q}_p$-linear isometric isomorphism from
$(\mathbb{Q}_p^n,\|\cdot\|_p)$ onto $\mathbb{Q}_{p^n}$ endowed with
its normalized absolute value. It maps $\mathbb{Z}_p^n$ onto
$\mathbb{Z}_{p^n}$ and preserves normalized Haar measure. Indeed, for
a nonzero vector $(x_1,\ldots,x_n)$, after factoring out
$p^{\min_i v_p(x_i)}$, at least one coefficient is a unit, and the
reductions of $e_1,\ldots,e_n$ form an $\mathbb{F}_p$-basis of
$\mathbb{Z}_{p^n}/p\mathbb{Z}_{p^n}$; hence the remaining sum is a
unit. We use this identification throughout the paper. In particular,
the finite part of the boundary is identified with $\mathbb{Q}_p^n$:
\begin{align*}
    \partial\mathcal{T}_{p^n}\setminus\{\infty\}
    \cong\mathbb{Q}_{p^n}\cong\mathbb{Q}_p^n.
\end{align*}
The corresponding horocyclic coordinates on $\mathcal{T}_{p^n}$ are
defined as follows.
\begin{definition}[Horocyclic coordinates on $\mathcal{T}_{p^n}$]
    Each vertex $v\in \mathcal{T}_{p^n}$ is parameterized by a pair
    $(\boldsymbol{x},k)$, where
    \begin{enumerate}
        \item The scale height is $k\in\mathbb{Z}$. Along a ray to a
        point of $\mathbb{Q}_p^n$, the height tends to $+\infty$,
        whereas along the distinguished ray to $\infty$, it tends to
        $-\infty$.
        \item The horizontal position is
        $\boldsymbol{x}\in
        \mathbb{Q}_p^n/p^k\mathbb{Z}_p^n$. For a fixed height $k$, the
        set of vertices
        \begin{align*}
            \{v=(\boldsymbol{x},k)\in\mathcal{T}_{p^n}\mid
            \boldsymbol{x}\in
            \mathbb{Q}_p^n/p^k\mathbb{Z}_p^n\}
        \end{align*}
        forms a horizontal slice of the tree, and
        \begin{align*}
            \{\boldsymbol{y}\in
            \partial\mathcal{T}_{p^n}\setminus\{\infty\}\mid
            \text{there is a path }
            \infty\to(\boldsymbol{x},k)\to\boldsymbol{y}\}
            =\boldsymbol{x}+p^k\mathbb{Z}_p^n.
        \end{align*}
    \end{enumerate}
\end{definition}
\begin{figure}[H]
    \centering
    \includegraphics[width=1\textwidth]{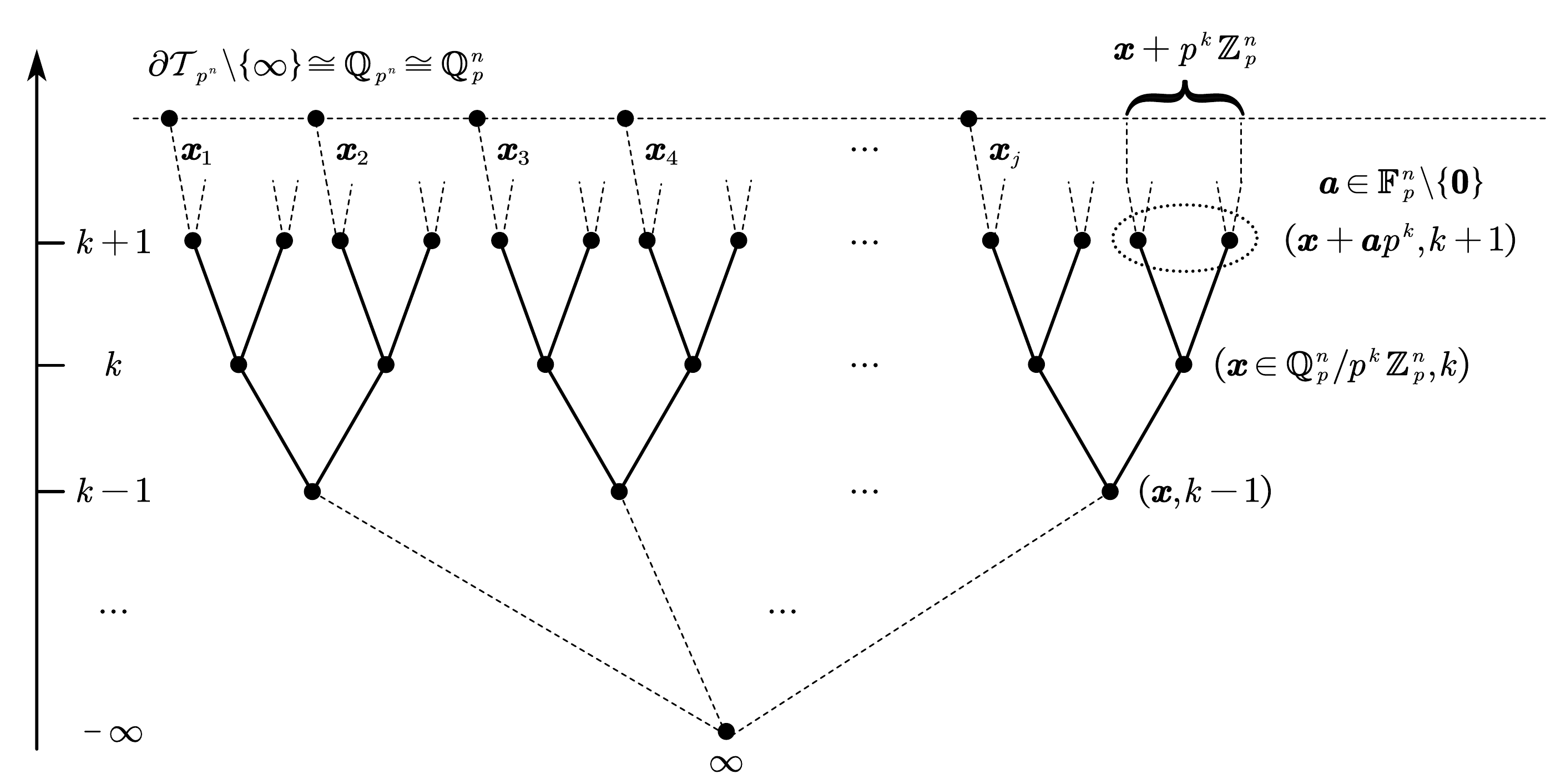}
    \caption{Horocyclic coordinates on $\mathcal{T}_{p^n}$}
    \label{fig:horocyclic coordinate}
\end{figure}
In these coordinates, each
$v=(\boldsymbol{x},k)\in\mathcal{T}_{p^n}$ has a unique upper neighbor
$(\boldsymbol{x},k-1)$, closer to $\infty$, and $p^n$ lower neighbors
$(\boldsymbol{x}+\boldsymbol{a}p^k,k+1)$ with
$\boldsymbol{a}\in\mathbb{F}_p^n$, in the direction of
$\mathbb{Q}_p^n$.

Equivalently, each ball
$\boldsymbol{x}+p^k\mathbb{Z}_p^n\subseteq\mathbb{Q}_p^n$ has
\begin{align*}
    [p^k\mathbb{Z}_p^n:p^{k+1}\mathbb{Z}_p^n]
    =|\mathbb{Z}_p^n/p\mathbb{Z}_p^n|
    =|\mathbb{F}_p^n|=p^n
\end{align*}
maximal proper sub-balls of radius $p^{-(k+1)}$. Since two balls in
$\mathbb{Q}_p^n$ are either disjoint or nested, each ball is uniquely
contained in a super-ball of radius $p^{-(k-1)}$. We therefore regard
each ball $\boldsymbol{x}+p^k\mathbb{Z}_p^n$ as a vertex
$v=(\boldsymbol{x},k)$ connected to $p^n+1$ vertices, corresponding to
its $p^n$ maximal proper sub-balls and its unique immediate super-ball.

\subsection{Ultrametric spaces and the associated ultrametric tree}
\begin{definition}[Ultrametric space]
    Let $(X,d)$ be a metric space. If the strong triangle inequality
    \begin{align*}
        d(x,y)\leq\max\{d(x,z),d(y,z)\}
    \end{align*}
    holds for all $x,y,z\in X$, then $(X,d)$ is called an ultrametric
    space.
\end{definition}

Throughout the hierarchical part of this paper, we assume that $(X,d)$
is a noncompact, locally compact, separable, perfect ultrametric space.
Any two balls in an ultrametric space are either disjoint or one is
contained in the other. We further assume that $\mathcal{B}$ is a
countable basis consisting of nonempty compact open balls.

\begin{definition}[Radon measures and balls]
    Let $m$ be a non-atomic Borel Radon measure of full support on
    $(X,d)$. We assume that
    \begin{align*}
        m(X)=+\infty,\qquad
        0<m(B)<+\infty\quad\text{for every }B\in\mathcal{B}.
    \end{align*}
\end{definition}

\begin{definition}[Sub-balls and super-balls]
    We assume that every ball $B\in\mathcal{B}$ has a unique immediate
    strict super-ball $B'\in\mathcal{B}$. We call $B'$ the immediate
    super-ball of $B$ and call $B$ a sub-ball of $B'$; we write
    $B\in\text{sub}(B')$. We also assume that every
    $B'\in\mathcal{B}$ is the disjoint union of finitely many
    (at least two) maximal proper sub-balls, whose collection is denoted
    by $\text{sub}(B')$.
\end{definition}

We impose the following additional standing assumptions on the ball
structure. Any two balls in $\mathcal{B}$ are contained in a common ball
of $\mathcal{B}$, which can be reached from either one by finitely many
iterations of $B\mapsto B'$. For every $B\in\mathcal{B}$, the union of
its successive super-balls is $X$. Along every infinite strictly
descending chain formed by successively choosing sub-balls, the
diameters tend to zero.
Consequently, for each $x\in X$, the chain of balls containing $x$ can
be indexed as
\begin{align*}
    \cdots\supseteq B_{k-1}\supseteq B_k
    \supseteq B_{k+1}\supseteq\cdots\ni x,
    \qquad B_k'=B_{k-1},
\end{align*}
with
\begin{align*}
    \bigcap_{k\in\mathbb{Z}}B_k=\{x\},\qquad
    \bigcup_{k\in\mathbb{Z}}B_k=X,\qquad
    \lim_{k\to+\infty}\operatorname{diam}(B_k)=0.
\end{align*}

\begin{definition}[Associated ultrametric tree]
    The vertex set is $\mathcal{B}$, and each
    $B\in\mathcal{B}$ is connected to its immediate super-ball $B'$ by an
    undirected edge. The resulting connected, locally finite, acyclic
    graph is called the ultrametric tree associated with $X$, denoted by
    $\mathcal{T}_X$. Its descending ends are identified with the points
    of $X$, while its common ascending end is denoted by $\infty$.
    Accordingly,
    \begin{align*}
        \partial\mathcal{T}_X\cong X\cup\{\infty\}.
    \end{align*}
    Changing the reference vertex changes the Busemann height only by an
    additive constant. We write $h(B)=h(v_B)\in\mathbb{Z}$ for the
    Busemann height relative to $\infty$, normalized by
    \begin{align*}
        h(B)-h(B')=1.
    \end{align*}
    All truncations below are defined in terms of this Busemann height. We
    choose the indexing of each boundary chain so that $h(B_k(x))=k$. For each
    $k\in\mathbb Z$, the balls of height $k$ are pairwise disjoint and
    cover $X$.

    For each boundary point $x\in X$, the geodesic joining the ends
    $\infty$ and $x$ corresponds to the bi-infinite chain
    \begin{align*}
        \cdots\supseteq B_{k-1}\supseteq B_k
        \supseteq B_{k+1}\supseteq\cdots\ni x.
    \end{align*}
    Since $B_k\downarrow\{x\}$ as $k\to+\infty$, non-atomicity and
    continuity of $m$ from above give
    \begin{align*}
        \lim_{k\to+\infty}m(B_k)=0.
    \end{align*}
\end{definition}

\subsection{Hierarchical Laplacians on ultrametric spaces}
\begin{definition}[Bruhat--Schwartz space]
    The Bruhat--Schwartz space $\mathcal{S}(X)$ is the set of locally constant
    functions with compact support. Its zero-mean subspace is defined by
    \begin{align*}
        \mathcal{S}_0(X)=\left\{f\in \mathcal{S}(X)\ \middle|\
        \int_Xf(x)dm(x)=0\right\}.
    \end{align*}
\end{definition}
\begin{remark}
    Fix a ball $B\in\mathcal{B}$. Then we have the algebraic direct sum
    \begin{align*}
        \mathcal{S}(X)=\mathcal{S}_0(X)\oplus\boldsymbol{1}_B\mathbb{R}.
    \end{align*}
    Indeed, for every $f\in \mathcal{S}(X)$,
    \begin{align*}
        f(x)
        =\left(f(x)-\frac{\int_Xf(y)dm(y)}{m(B)}\boldsymbol{1}_B(x)\right)+\frac{\int_Xf(y)dm(y)}{m(B)}\boldsymbol{1}_B(x),
    \end{align*}
    where the first term on the right-hand side belongs to $\mathcal{S}_0(X)$
    and the second belongs to $\boldsymbol{1}_B\mathbb{R}$. The
    intersection of these two subspaces is $\{0\}$ because $m(B)>0$.
\end{remark}

\begin{definition}[Hierarchical Laplacian]
    Let $C:\mathcal{B}\to(0,+\infty)$ be a strictly positive choice function
    such that, for every $B\in\mathcal{B}$,
    \begin{align*}
        \sum_{\substack{D\in\mathcal{B}\\D\supseteq B}}C(D)<+\infty,
    \end{align*}
    and, for every $x\in X$, along the descending chain
    $B_k(x)\ni x$ described above,
    \begin{align*}
        \sum_{\substack{D\in\mathcal{B}\\D\supseteq B_k}}C(D)
        \longrightarrow+\infty
        \qquad(k\to+\infty).
    \end{align*}

    For $f\in \mathcal{S}(X)$, the hierarchical Laplacian $L_Cf(x)$ is
    defined pointwise by the series
    \begin{align*}
        L_Cf(x)
        =\sum_{\substack{B\in\mathcal{B}\\B\ni x}}
        C(B)\left(f(x)-\frac{1}{m(B)}
        \int_Bf(y)dm(y)\right).
    \end{align*}
\end{definition}
The series is well-defined on $\mathcal{S}(X)$. Indeed, the terms corresponding
to sufficiently small balls vanish because $f$ is locally constant,
while the remaining ancestor tail converges because
$\sum\limits_{D\supseteq B}C(D)<+\infty$.

\begin{remark}
    For $x\ne y$, let $B_{x,y}$ be the smallest ball containing $x$ and
    $y$. Then
    \begin{align*}
        L_Cf(x)
        =&\sum_{\substack{B\in\mathcal{B}\\B\ni x}}
        \frac{C(B)}{m(B)}
        \int_B(f(x)-f(y))dm(y)\\
        =&\int_X\left(
        \sum_{\substack{B\in\mathcal{B}\\B\ni x}}
        \frac{C(B)}{m(B)}\boldsymbol{1}_B(y)\right)
        (f(x)-f(y))dm(y)\\
        =&\int_X
        \sum_{\substack{B\in\mathcal{B}\\B\supseteq B_{x,y}}}
        \frac{C(B)}{m(B)}(f(x)-f(y))dm(y).
    \end{align*}
    Thus,
    \begin{align*}
        K(x,y)
        =\sum_{\substack{B\in\mathcal{B}\\B\supseteq B_{x,y}}}
        \frac{C(B)}{m(B)}
    \end{align*}
    depends on the smallest common ball $B_{x,y}$. The diagonal value
    $K(x,x)$ may be assigned arbitrarily because the kernel is multiplied
    by $f(x)-f(y)$. In general, $K(x,y)$ is not determined by
    $d(x,y)$ alone. It is determined by $d(x,y)$ under an additional
    homogeneity assumption; for example,
    this holds if a measure-preserving isometry group acts transitively
    on pairs of points at each fixed distance and satisfies
    $C(\tau(B))=C(B)$ for every ball $B$ and every isometry $\tau$ in
    the group.
\end{remark}
\begin{remark}
    The operator $L_C$ is a symmetric, nonnegative jump-type operator
    whose jump kernel is determined by the hierarchical structure. Under
    the additional homogeneity assumption above, it is isotropic with
    respect to $d$. For $f\in \mathcal{S}(X)$, its quadratic form is
    \begin{align*}
        \langle f,L_Cf\rangle_{L^2(X,m)}
        =\frac{1}{2}\int_X\int_XK(x,y)
        |f(x)-f(y)|^2dm(x)dm(y)\geq0.
    \end{align*}
\end{remark}
\begin{remark}
    Let $X=\mathbb{Q}_p^n$, equip $X$ with the metric
    \begin{align*}
        d(\boldsymbol{x},\boldsymbol{y})
        =\|\boldsymbol{x}-\boldsymbol{y}\|_p,
    \end{align*}
    and let $m$ be the Haar measure $d\boldsymbol{x}$. Define
    \begin{align*}
        C(\boldsymbol{x}+p^{k}\mathbb{Z}_p^n)
        =p^{(k+1)s}(1-p^{-s}),
        \qquad s>0.
    \end{align*}
    If
    $\|\boldsymbol{x}-\boldsymbol{y}\|_p=p^{-k}$, then
    $B_{\boldsymbol{x},\boldsymbol{y}}
    =\boldsymbol{x}+p^k\mathbb{Z}_p^n$, and hence
    \begin{align*}
        K(\boldsymbol{x},\boldsymbol{y})
        =&\sum_{j=-\infty}^{k}
        \frac{p^{(j+1)s}(1-p^{-s})}{p^{-nj}}\\
        =&\frac{p^s-1}{1-p^{-s-n}}
        \|\boldsymbol{x}-\boldsymbol{y}\|_p^{-s-n}.
    \end{align*}
    Comparing this expression with the singular integral for $D^s$
    shows that $L_C=D^s$ on $\mathcal{S}(\mathbb{Q}_p^n)$.
\end{remark}
\begin{definition}[Eigenvalue scale function]
    For $B\in\mathcal{B}$, define the scale function
    $\lambda(B)$ by
    \begin{align*}
        \lambda(B)
        =\sum_{\substack{D\in\mathcal{B}\\D\supseteq B}}C(D).
    \end{align*}
    By the preceding assumptions, $0<\lambda(B)<+\infty$ and
    $\lambda(B_k)\to+\infty$ as $k\to+\infty$ along every boundary
    chain.
\end{definition}
Since $B'$ is the immediate super-ball of $B$, we have
\begin{align*}
    C(B)=\lambda(B)-\lambda(B').
\end{align*}
\begin{lemma}[Eigenfunctions of $L_C$]
\label{lem:Eigenfunctions_of_LC}
    For any ball $B\in\mathcal{B}$, define the piecewise constant
    function
    \begin{align*}
        f_B(x)=
        \begin{cases}
            \dfrac{1}{m(B)}-\dfrac{1}{m(B')},&x\in B,\\
            -\dfrac{1}{m(B')},&x\in B'\setminus B,\\
            0,&x\notin B'.
        \end{cases}
    \end{align*}
    Then
    \begin{align*}
        L_Cf_B(x)=\lambda(B')f_B(x),
        \qquad x\in X.
    \end{align*}
\end{lemma}
\begin{proof}
    Direct computation gives
    \begin{align*}
        \int_Xf_B(x)dm(x)=0.
    \end{align*}
    We distinguish three cases according to the location of $x$.
    \begin{enumerate}
        \item[]\textbf{Case 1.} Suppose $x\notin B'$. Then
        $f_B(x)=0$. If a ball $D\ni x$ is disjoint from $B'$, then
        $f_B|_D=0$. If $D$ intersects $B'$, the ultrametric nesting
        property and $x\notin B'$ imply $D\supseteq B'$, and hence
        \begin{align*}
            \int_Df_B(y)dm(y)=\int_{B'}f_B(y)dm(y)=0.
        \end{align*}
        Thus every term in $L_Cf_B(x)$ vanishes.

        \item[]\textbf{Case 2.} Suppose $x\in B$. If $D\ni x$ and
        $D\subseteq B$, then $f_B$ is constant on $D$, so
        \begin{align*}
            f_B(x)-\frac{1}{m(D)}\int_Df_B(y)dm(y)=0.
        \end{align*}
        Since $B'$ is the immediate super-ball of $B$, every remaining
        ball $D\ni x$ contains $B'$. For such $D$,
        $\int_Df_Bdm=0$, so the corresponding summand is $C(D)f_B(x)$.
        Therefore,
        \begin{align*}
            L_Cf_B(x)
            =\sum_{\substack{D\in\mathcal{B}\\D\supseteq B'}}
            C(D)f_B(x)
            =\lambda(B')f_B(x).
        \end{align*}

        \item[]\textbf{Case 3.} Suppose $x\in B'\setminus B$. Then
        $x$ belongs to a maximal sub-ball of $B'$ different from $B$.
        On every ball $D$ contained in that sub-ball, $f_B$ is
        constant and the corresponding term vanishes. Every remaining
        ball $D\ni x$ contains $B'$, so $\int_Df_Bdm=0$. Hence
        \begin{align*}
            L_Cf_B(x)
            =\sum_{\substack{D\in\mathcal{B}\\D\supseteq B'}}
            C(D)f_B(x)
            =\lambda(B')f_B(x).
        \end{align*}
    \end{enumerate}
\end{proof}
\begin{remark}
    For each ball $B'\in\mathcal{B}$, let
    \begin{align*}
        W_{B'}=\{f_B\mid B\in\text{sub}(B')\}.
    \end{align*}
    This is a linearly dependent family of eigenfunctions, all with
    eigenvalue $\lambda(B')$. It satisfies
    \begin{align*}
        \sum_{B\in\text{sub}(B')}m(B)f_B=0.
    \end{align*}
    This is its only linear relation up to a scalar multiple. Indeed, if
    $\sum\limits_{B\in\text{sub}(B')}c_Bf_B=0$ and
    $S=\sum\limits_{B\in\text{sub}(B')}c_B$, then restricting the identity to a fixed sub-ball
    $A\in\text{sub}(B')$ gives
    \begin{align*}
        \frac{c_A}{m(A)}-\frac{S}{m(B')}=0,
    \end{align*}
    so $c_A$ is proportional to $m(A)$. Consequently,
    \begin{align*}
        \dim\operatorname{span}W_{B'}
        =\#W_{B'}-1.
    \end{align*}

    The spaces $\operatorname{span}W_{B'}$ corresponding to distinct
    balls $B'$ are mutually orthogonal. Indeed, for functions belonging
    to two such spaces, the corresponding supporting balls are either
    disjoint or nested. In the latter case, the function supported on the
    larger ball is constant on the maximal sub-ball containing the smaller
    ball, whereas the function supported on the smaller ball has integral
    zero.

    Finally, these spaces algebraically span $\mathcal{S}_0(X)$. To see this, let
    $f\in \mathcal{S}_0(X)$ and choose a ball containing
    $\operatorname{supp}f$. By compactness of the support and local
    constancy of $f$, this ball admits a finite refinement into sub-balls
    such that $f$ is constant on each terminal sub-ball. Subtracting
    successive averages as one passes from the chosen ball to the terminal
    sub-balls expresses $f$ as a finite sum of elements of
    $\operatorname{span}W_{B'}$. Thus, after deleting one element from
    each $W_{B'}$, the remaining functions form an algebraic
    eigenbasis of $\mathcal{S}_0(X)$. We fix such a choice whenever an expansion
    in terms of the functions $f_B$ is used below; the full family
    $\{f_B\}_{B\in\mathcal{B}}$ is not an orthonormal basis.
\end{remark}
\section{The extension method and the deformed normal derivative}

\subsection{The energy functional and its Euler--Lagrange equation}
Let $\varphi$ be a function on $\mathcal{T}_{p^n}$. In horocyclic coordinates, write $\varphi_k(\boldsymbol{x})=\varphi(\boldsymbol{x},k)$. We may then regard $\varphi$ as a sequence $\{\varphi_k\}_{k\in\mathbb{Z}}$ of locally constant functions on $\mathbb{Q}_p^n$, with $\varphi_k$ constant on every ball $\boldsymbol{x}_0+p^k\mathbb{Z}_p^n$. If $d\delta$ denotes counting measure on $\mathbb{Q}_p^n/p^k\mathbb{Z}_p^n$, then
\begin{equation}
\label{eq:Qpn_to_Qpn/pkZpn}
 \begin{aligned}
    \int_{\mathbb{Q}_p^n}|\varphi_k(\boldsymbol{x})-\varphi_{k-1}(\boldsymbol{x})|^2d\boldsymbol{x}
    =&|p^k\mathbb{Z}_p^n|\int_{\mathbb{Q}_p^n/p^k\mathbb{Z}_p^n}|\varphi_k(\boldsymbol{x})-\varphi_{k-1}(\boldsymbol{x})|^2d\delta\\
    =&p^{-nk}\int_{\mathbb{Q}_p^n/p^k\mathbb{Z}_p^n}|\varphi_k(\boldsymbol{x})-\varphi_{k-1}(\boldsymbol{x})|^2d\delta.
\end{aligned}
\end{equation}
We assign to the edge joining $(\boldsymbol{x},k-1)$ and $(\boldsymbol{x},k)$ the weight
\begin{align*}
w_{(\boldsymbol{x},k-1),(\boldsymbol{x},k)}=p^{k(s-n)}.
\end{align*}
By \eqref{eq:Qpn_to_Qpn/pkZpn}, the global Dirichlet energy functional on $\mathcal{T}_{p^n}$ can be written as
\begin{equation}
\label{eq:energy functional}
\begin{aligned}
\mathcal{E}[\varphi]
=&\frac{1}{2}\sum_{u\in \mathcal{T}_{p^n}}\sum_{v\sim u}w_{uv}|\varphi(u)-\varphi(v)|^2\\
=&\sum_{k\in\mathbb{Z}}p^{-k(n-s)}\int_{\mathbb{Q}_p^n/p^k\mathbb{Z}_p^n}|\varphi_k(\boldsymbol{x})-\varphi_{k-1}(\boldsymbol{x})|^2d\delta\\
=&\sum_{k\in\mathbb{Z}}p^{ks}\int_{\mathbb{Q}_p^n}|\varphi_k(\boldsymbol{x})-\varphi_{k-1}(\boldsymbol{x})|^2d\boldsymbol{x}.
\end{aligned}
\end{equation}
The factor $1/2$ appears only in the double sum because each unoriented edge is counted twice. In each of the last two sums, every edge is indexed exactly once by its endpoint of greater height.

For the variational calculation, first assume that $\varphi$ is real-valued and let
\begin{align*}
    \mathcal{G}=\left\{g\in \mathbb{R}^{\mathcal{T}_{p^n}}\ \middle|\
    \begin{array}{l}
    g_k(\boldsymbol{x})\text{ is locally constant on every ball }
    \boldsymbol{x}_0+p^k\mathbb{Z}_p^n,\\
    g\text{ is supported on finitely many vertices}
    \end{array}\right\}.
\end{align*}
For $g\in\mathcal{G}$, write
$g_k(\boldsymbol{x}_0)=g_k(\boldsymbol{x}_0+p^k\mathbb{Z}_p^n)$. If $\varphi$ minimizes the energy functional \eqref{eq:energy functional}, then
\begin{align*}
&\frac{d\mathcal{E}[\varphi+\varepsilon g]}{d\varepsilon}\Big|_{\varepsilon=0}\\
=&2\sum_{k\in\mathbb{Z}}p^{ks}\int_{\mathbb{Q}_p^n}
(\varphi_k(\boldsymbol{x})-\varphi_{k-1}(\boldsymbol{x}))
(g_k(\boldsymbol{x})-g_{k-1}(\boldsymbol{x}))d\boldsymbol{x}\\
=&2\sum_{k\in\mathbb{Z}}\int_{\mathbb{Q}_p^n}
\left(p^{ks}(\varphi_k(\boldsymbol{x})-\varphi_{k-1}(\boldsymbol{x}))
-p^{(k+1)s}(\varphi_{k+1}(\boldsymbol{x})-\varphi_k(\boldsymbol{x}))\right)
g_k(\boldsymbol{x})d\boldsymbol{x}\\
=&2\sum_{k\in\mathbb{Z}}
\sum_{\boldsymbol{x}_0\in\mathbb{Q}_p^n/p^k\mathbb{Z}_p^n}
g_k(\boldsymbol{x}_0)\Bigg(
p^{ks}\int_{\boldsymbol{x}_0+p^k\mathbb{Z}_p^n}
(\varphi_k(\boldsymbol{x})-\varphi_{k-1}(\boldsymbol{x}))d\boldsymbol{x}-p^{(k+1)s}\int_{\boldsymbol{x}_0+p^k\mathbb{Z}_p^n}
(\varphi_{k+1}(\boldsymbol{x})-\varphi_k(\boldsymbol{x}))d\boldsymbol{x}\Bigg)\\
=&0.
\end{align*}
Because $g_k(\boldsymbol{x}_0)$ may be chosen independently at each vertex, its coefficient in the last expression must vanish. Decomposing the ball $\boldsymbol{x}_0+p^k\mathbb{Z}_p^n$ into its $p^n$ children yields
\begin{align*}
|\boldsymbol{x}_0+p^k\mathbb{Z}_p^n|p^{ks}
(\varphi_k(\boldsymbol{x}_0)-\varphi_{k-1}(\boldsymbol{x}_0))-p^{(k+1)s}\sum_{\boldsymbol{a}\in\mathbb{F}_p^n}
|\boldsymbol{x}_0+\boldsymbol{a}p^k+p^{k+1}\mathbb{Z}_p^n|
(\varphi_{k+1}(\boldsymbol{x}_0+\boldsymbol{a}p^k)-\varphi_k(\boldsymbol{x}_0))=0.
\end{align*}
Using
$|\boldsymbol{x}_0+p^k\mathbb{Z}_p^n|=p^{-nk}$ and
$|\boldsymbol{x}_0+\boldsymbol{a}p^k+p^{k+1}\mathbb{Z}_p^n|=p^{-n(k+1)}$, this equation becomes
\begin{align*}
p^{k(s-n)}(\varphi_k(\boldsymbol{x}_0)-\varphi_{k-1}(\boldsymbol{x}_0))-p^{(k+1)(s-n)}\sum_{\boldsymbol{a}\in\mathbb{F}_p^n}
(\varphi_{k+1}(\boldsymbol{x}_0+\boldsymbol{a}p^k)-\varphi_k(\boldsymbol{x}_0))=0.
\end{align*}
Multiplying by $-1$ yields the Euler--Lagrange equation associated with \eqref{eq:energy functional}, namely the weighted graph Laplacian equation
\begin{align}
\label{eq:Euler-Lagrange equation}
    \Delta_w\varphi(\boldsymbol{x},k)
    =p^{k(s-n)}(\varphi_{k-1}(\boldsymbol{x})-\varphi_k(\boldsymbol{x}))+p^{(k+1)(s-n)}\sum_{\boldsymbol{a}\in\mathbb{F}_p^n}
    (\varphi_{k+1}(\boldsymbol{x}+\boldsymbol{a}p^k)-\varphi_k(\boldsymbol{x}))=0,
\end{align}
for every $k\in\mathbb{Z}$ and every
$\boldsymbol{x}\in\mathbb{Q}_p^n/p^k\mathbb{Z}_p^n$. Since $\Delta_w$ has real coefficients, the same equation holds componentwise for complex-valued functions.

\subsection{Green's identities and the deformed normal derivative}
For the one-sided truncation, consider
\begin{align*}
    \mathcal{T}_{p^n}^{R}
    =\{v=(\boldsymbol{x},k)\in\mathcal{T}_{p^n}\mid k\leq R\}
\end{align*}
and define its upper boundary by
\begin{align*}
    \partial\mathcal{T}_{p^n}^{R}
    =\{v=(\boldsymbol{x},k)\in\mathcal{T}_{p^n}\mid k=R\}.
\end{align*}
Although $\mathcal{T}_{p^n}^{R}$ is infinite, the finite-energy and $L^2$
assumptions in the following lemma ensure that every sum in the proof
converges absolutely.

\begin{lemma}[Green's first identity for the truncated tree $\mathcal{T}_{p^n}^{R}$]
\label{lem:Green's first identity for the truncated tree}
Let $g,h\in\mathbb{R}^{\mathcal{T}_{p^n}}$ satisfy
$\mathcal{E}[g]<+\infty$, $\mathcal{E}[h]<+\infty$, and
$g_R\in L^2(\mathbb{Q}_p^n)$. Then
\begin{align*}
\frac{1}{2}\sum_{\substack{u,v\in\mathcal{T}_{p^n}^{R}\\u\sim v}}
(g(u)-g(v))(h(u)-h(v))w_{uv}=-\sum_{u\in\mathcal{T}_{p^n}^{R}}g(u)\Delta_wh(u)
-\sum_{u\in\partial\mathcal{T}_{p^n}^{R}}g(u)
\sum_{\substack{v\notin\mathcal{T}_{p^n}^{R}\\v\sim u}}
(h(u)-h(v))w_{uv}.
\end{align*}
\end{lemma}
\begin{proof}
We first justify the rearrangements of the infinite sums. Writing
$g_k=g_R-\sum\limits_{j=k+1}^{R}(g_j-g_{j-1})$ for $k\leq R$ and applying
Young's inequality for sequences to the summable kernel
$\{p^{-ms/2}\}_{m\geq1}$ gives the weighted discrete Hardy estimate
\begin{align*}
\left(\sum_{k\leq R}p^{ks}\|g_k\|_2^2\right)^{1/2}\leq\frac{p^{Rs/2}}{\sqrt{1-p^{-s}}}\|g_R\|_2+\frac{p^{-s/2}}{1-p^{-s/2}}\left(\sum_{k\leq R}p^{ks}\|g_k-g_{k-1}\|_2^2\right)^{1/2}<+\infty.
\end{align*}
It follows from the Cauchy--Schwarz inequality and
$\mathcal{E}[h]<+\infty$ that
\begin{align*}
&\sum_{u\in\mathcal{T}_{p^n}^{R}}|g(u)|
\sum_{v\sim u}|h(u)-h(v)|w_{uv}\\
=&\sum_{k\leq R}p^{ks}\int_{\mathbb{Q}_p^n}
(|g_k|+|g_{k-1}|)|h_k-h_{k-1}|d\boldsymbol{x}+p^{(R+1)s}\int_{\mathbb{Q}_p^n}
|g_R||h_{R+1}-h_R|d\boldsymbol{x}<+\infty.
\end{align*}
Thus the vertex sum, the paired interior-edge sum, and the upper boundary
sum below are all absolutely convergent.

By the definition of $\Delta_w$,
\begin{align*}
-\sum_{u\in\mathcal{T}_{p^n}^{R}}g(u)\Delta_wh(u)
=\sum_{u\in\mathcal{T}_{p^n}^{R}}g(u)
\sum_{v\sim u}(h(u)-h(v))w_{uv}.
\end{align*}
We partition the edges in this sum into interior edges and edges crossing
the upper boundary.
\begin{enumerate}
    \item[]\textbf{Case 1.} Suppose that $v\in\mathcal{T}_{p^n}^{R}$, so that $\{u,v\}$ is an interior edge. The two orientations contribute
    \begin{align*}
    g(u)(h(u)-h(v))w_{uv}
    +g(v)(h(v)-h(u))w_{uv}=(g(u)-g(v))(h(u)-h(v))w_{uv}.
    \end{align*}
    Thus the total contribution of the interior edges is
    \begin{align*}
    \frac{1}{2}\sum_{\substack{u,v\in\mathcal{T}_{p^n}^{R}\\u\sim v}}
    (g(u)-g(v))(h(u)-h(v))w_{uv}.
    \end{align*}
    \item[]\textbf{Case 2.} Suppose that $v\notin\mathcal{T}_{p^n}^{R}$. Then
    $u\in\partial\mathcal{T}_{p^n}^{R}$ and $v$ is a child of $u$ at
    height $R+1$. This edge occurs only once in the vertex sum and contributes
    \begin{align*}
    \sum_{u\in\partial\mathcal{T}_{p^n}^{R}}g(u)
    \sum_{\substack{v\notin\mathcal{T}_{p^n}^{R}\\v\sim u}}
    (h(u)-h(v))w_{uv}.
    \end{align*}
\end{enumerate}
Moving the contribution from Case 2 to the other side gives the claimed identity.
\end{proof}

\begin{definition}[Deformed normal derivative on $\partial \mathcal{T}_{p^n}\setminus\{\infty\}$]
\label{def:The deformed normal derivative}
Let $\varphi\in\mathbb{R}^{\mathcal{T}_{p^n}}$ and suppose that
$\lim\limits_{k\to+\infty}\varphi_k(\boldsymbol{x})=f(\boldsymbol{x})$ on
$\partial\mathcal{T}_{p^n}\setminus\{\infty\}\cong\mathbb{Q}_p^n$.
At each point where the following limit exists, we define
\begin{align*}
    \partial_{n,s}^{(p)}\varphi(\boldsymbol{x})
    =\lim_{k\to+\infty}p^{(k+1)s}
    \left(f(\boldsymbol{x})-\varphi_k(\boldsymbol{x})\right).
\end{align*}
If $p^{(k+1)s}(f-\varphi_k)$ converges in $L^2(\mathbb{Q}_p^n)$, we use
the same notation for its $L^2$ limit and specify the mode of convergence
explicitly.
\end{definition}

\begin{theorem}[Green's first identity for $\mathcal{T}_{p^n}$]
\label{thm:Green's first identity for Tpn}
Let $g,h\in\mathbb{R}^{\mathcal{T}_{p^n}}$ satisfy the following conditions:
\begin{enumerate}
    \item $\mathcal{E}[g]<+\infty$ and $\mathcal{E}[h]<+\infty$;
    \item there exist $g_{\partial},h_{\partial}\in L^2(\mathbb{Q}_p^n)$ such that
    $g_R\to g_{\partial}$ and $h_R\to h_{\partial}$ in $L^2(\mathbb{Q}_p^n)$;
    \item $p^{(R+1)s}(h_{\partial}-h_R)
    \longrightarrow\partial_{n,s}^{(p)}h
    \quad\text{in }L^2(\mathbb{Q}_p^n)$;
    \item $\displaystyle\sum_{u\in\mathcal{T}_{p^n}}|g(u)\Delta_wh(u)|<+\infty$;
\end{enumerate}
Then
\begin{align*}
\frac{1}{2}\sum_{\substack{u,v\in\mathcal{T}_{p^n}\\u\sim v}}
(g(u)-g(v))(h(u)-h(v))w_{uv}=-\sum_{u\in\mathcal{T}_{p^n}}g(u)\Delta_wh(u)
+(1-p^{-s})\int_{\mathbb{Q}_p^n}
g_{\partial}(\boldsymbol{x})\partial_{n,s}^{(p)}h(\boldsymbol{x})d\boldsymbol{x}.
\end{align*}
\end{theorem}
\begin{proof}
We apply Lemma \ref{lem:Green's first identity for the truncated tree} to
$\mathcal{T}_{p^n}^{R}$. For all sufficiently large $R$, Condition (2)
ensures that $g_R\in L^2(\mathbb{Q}_p^n)$, and hence the lemma applies. We
begin with the upper boundary term.
Let $\mathbb{E}_R$ denote averaging over the ball
$\boldsymbol{x}+p^R\mathbb{Z}_p^n$; thus
\begin{align*}
\mathbb{E}_RF(\boldsymbol{x})
=p^{nR}\int_{\boldsymbol{x}+p^R\mathbb{Z}_p^n}F(\boldsymbol{y})d\boldsymbol{y}.
\end{align*}
Choose one representative $\boldsymbol{x}_0$ from each coset in
$\mathbb{Q}_p^n/p^R\mathbb{Z}_p^n$. The height-$R$ ball $\boldsymbol{x}_0+p^R\mathbb{Z}_p^n$ has the $p^n$ children
\begin{align*}
    \boldsymbol{x}_0+\boldsymbol a p^R
    +p^{R+1}\mathbb{Z}_p^n,
    \qquad \boldsymbol a\in\mathbb F_p^n,
\end{align*}
and every edge from height $R$ to height $R+1$ has weight
$p^{(R+1)(s-n)}$. Since $h_{R+1}$ is constant on these children,
\begin{align*}
    \sum_{\boldsymbol a\in\mathbb F_p^n}
    h_{R+1}(\boldsymbol{x}_0+\boldsymbol a p^R)
    =p^n\mathbb E_Rh_{R+1}(\boldsymbol{x}_0).
\end{align*}
The upper boundary term in the one-sided Green identity is therefore
given by the following discrete-to-integral computation:
\begin{align*}
-\sum_{u\in\partial\mathcal{T}_{p^n}^{R}}g(u)
\sum_{\substack{v\notin\mathcal{T}_{p^n}^{R}\\v\sim u}}
(h(u)-h(v))w_{uv}&=-\sum_{\boldsymbol{x}_0\in
\mathbb Q_p^n/p^R\mathbb Z_p^n}
g_R(\boldsymbol{x}_0)p^{(R+1)(s-n)}
\sum_{\boldsymbol a\in\mathbb F_p^n}
\bigl(h_R(\boldsymbol{x}_0)
-h_{R+1}(\boldsymbol{x}_0+\boldsymbol a p^R)\bigr)\\
&=\sum_{\boldsymbol{x}_0\in
\mathbb Q_p^n/p^R\mathbb Z_p^n}
g_R(\boldsymbol{x}_0)p^{(R+1)s-Rn}
\bigl(\mathbb E_Rh_{R+1}(\boldsymbol{x}_0)
-h_R(\boldsymbol{x}_0)\bigr)\\
&=\int_{\mathbb{Q}_p^n}g_R(\boldsymbol{x})p^{(R+1)s}
(\mathbb{E}_Rh_{R+1}(\boldsymbol{x})-h_R(\boldsymbol{x}))d\boldsymbol{x}.
\end{align*}
In the last equality, we used
$m(\boldsymbol{x}_0+p^R\mathbb{Z}_p^n)=p^{-Rn}$.
Moreover, the following identity is exact:
\begin{align*}
p^{(R+1)s}(\mathbb{E}_Rh_{R+1}-h_R)=\mathbb{E}_R\left[p^{(R+1)s}(h_{\partial}-h_R)\right]
-p^{-s}\mathbb{E}_R\left[p^{(R+2)s}(h_{\partial}-h_{R+1})\right].
\end{align*}
The averaging operators are contractions on $L^2$, while
$\mathbb{E}_RF\to F$ in $L^2$ for every $F\in L^2(\mathbb{Q}_p^n)$.
Condition (3) therefore implies
\begin{align*}
p^{(R+1)s}(\mathbb{E}_Rh_{R+1}-h_R)
\longrightarrow(1-p^{-s})\partial_{n,s}^{(p)}h
\quad\text{in }L^2(\mathbb{Q}_p^n).
\end{align*}
Combining this convergence with $g_R\to g_{\partial}$ in $L^2$ and
applying the Cauchy--Schwarz inequality, we obtain
\begin{align*}
\lim_{R\to+\infty}
\int_{\mathbb{Q}_p^n}g_R(\boldsymbol{x})p^{(R+1)s}
(\mathbb{E}_Rh_{R+1}(\boldsymbol{x})-h_R(\boldsymbol{x}))d\boldsymbol{x}=(1-p^{-s})\int_{\mathbb{Q}_p^n}
g_{\partial}(\boldsymbol{x})\partial_{n,s}^{(p)}h(\boldsymbol{x})d\boldsymbol{x}.
\end{align*}
The finite-energy assumptions and the Cauchy--Schwarz inequality give
\begin{align*}
\frac{1}{2}\sum_{\substack{u,v\in\mathcal{T}_{p^n}\\u\sim v}}
|g(u)-g(v)|\,|h(u)-h(v)|w_{uv}<+\infty,
\end{align*}
so the interior-edge sums over $\mathcal{T}_{p^n}^{R}$ converge to the
full edge pairing. Condition (4) justifies passage to the limit in the
vertex sum. Letting $R\to+\infty$ in the one-sided Green identity proves
the result.
\end{proof}
\section{Extension problems for the Vladimirov--Taibleson operator}
Throughout this section, fix $s>0$ and $f\in \mathcal{S}(\mathbb{Q}_p^n)$, and set
$f(\infty)=0$. We seek a bounded solution that extends continuously to
the end compactification of $\mathcal{T}_{p^n}$.

Consider the Dirichlet problem
\begin{align}
\label{eq:the_Dirichlet_problem_extension}
\begin{cases}
    \Delta_w\varphi(v)=0,&v\in \mathcal{T}_{p^n};\\
    \varphi|_{\partial \mathcal{T}_{p^n}}(\boldsymbol{x})=f(\boldsymbol{x}),&\boldsymbol{x}\in \partial \mathcal{T}_{p^n}.
\end{cases}
\end{align}
\begin{remark}
Equivalently, the boundary condition may be written as
$\displaystyle\lim_{k\to+\infty}\varphi_k(\boldsymbol{x})=f(\boldsymbol{x})$
for $\boldsymbol{x}\in\mathbb{Q}_p^n$, and
$\displaystyle\lim_{k\to-\infty}\varphi_k=0$ at the distinguished end
$\infty$. Both limits are understood as end limits; in particular, the
condition at $\infty$ is required along every sequence converging to that
end.
\end{remark}

We prove Theorem
\ref{thm:Extension problems related to the Vladimirov-Taibleson operator}
by two complementary methods: Fourier analysis and the Poisson integral
formula.

\subsection{The Fourier-transform method}
\begin{lemma}[Character orthogonality for $n=1$]
\label{lem:The orthogonal relation of the characters for n=1}
    \begin{align*}
         \int_{\mathbb{Z}_p^*}\chi_p(p^kv)dv=\begin{cases}
             1-p^{-1},&k\geq0;\\
             -p^{-1},&k=-1;\\
             0,&k\leq-2.
         \end{cases},
    \end{align*}
    and
    \begin{align*}
        \int_{\mathbb{Z}_p}\chi_p(p^kv)dv=\begin{cases}
            1,&k\geq0;\\
            0,&k<0.
        \end{cases}.
    \end{align*}
\end{lemma}
\begin{proof}
We begin with
    \begin{align*}
        \int_{\mathbb{Z}_p^*}\chi_p(p^kv)dv
        =\int_{\mathbb{Z}_p^*}e^{2\pi i\{p^kv\}_p}dv.
    \end{align*}
    \begin{enumerate}
        \item[]\textbf{Case 1.} If $k\geq0$, then
        $\int_{\mathbb{Z}_p^*}\chi_p(p^kv)dv=|\mathbb{Z}_p^*|=1-p^{-1}$.
        \item[]\textbf{Case 2.} Suppose that $k<0$, and write $k=-l$ with
        $l>0$. Write $v=\sum\limits_{i=0}^{\infty}a_ip^i$, where
        $a_0\in\mathbb{F}_p^{\times}$ and $a_i\in\mathbb{F}_p$ for
        $i\geq1$. Then
        \begin{align*}
            p^kv=p^{-l}v
            =a_0p^{-l}+a_1p^{-l+1}+\dots+a_{l-1}p^{-1}
            +a_l+a_{l+1}p+\cdots,
        \end{align*}
        hence
        \begin{align*}
            \{p^kv\}_p
            =a_0p^{-l}+a_1p^{-l+1}+\dots+a_{l-1}p^{-1}.
        \end{align*}
        It follows that
        \begin{align*}
            \int_{\mathbb{Z}_p^*}\chi_p(p^kv)dv
            =&|1+p^l\mathbb{Z}_p|
            \int_{\mathbb{Z}_p^*/(1+p^l\mathbb{Z}_p)}
            e^{2\pi i(a_0p^{-l}+a_1p^{-l+1}+\dots+a_{l-1}p^{-1})}
            d\delta\\
            =&p^{-l}\sum_{a_0=1}^{p-1}e^{2\pi ia_0p^{-l}}
            \prod_{j=1}^{l-1}\sum_{a_j=0}^{p-1}
            e^{2\pi ia_jp^{-l+j}}.
        \end{align*}
        \begin{enumerate}
            \item If $l=1$, then
            \begin{align*}
                \int_{\mathbb{Z}_p^*}\chi_p(p^kv)dv
                =&p^{-1}\sum_{a_0=1}^{p-1}e^{2\pi ia_0p^{-1}}\\
                =&p^{-1}\left(\sum_{a_0=0}^{p-1}
                e^{2\pi ia_0p^{-1}}-1\right)=-p^{-1}.
            \end{align*}
            \item If $l>1$, then
            \begin{align*}
                \sum_{a_{l-1}=0}^{p-1}e^{2\pi ia_{l-1}p^{-1}}=0,
            \end{align*}
            and hence
            \begin{align*}
                \int_{\mathbb{Z}_p^*}\chi_p(p^kv)dv=0.
            \end{align*}
        \end{enumerate}
    \end{enumerate}
    Combining the two cases gives
    \begin{align*}
         \int_{\mathbb{Z}_p^*}\chi_p(p^kv)dv=\begin{cases}
             1-p^{-1},&k\geq0;\\
             -p^{-1},&k=-1;\\
             0,&k\leq-2.
         \end{cases}
    \end{align*}
    We next compute
    \begin{align*}
        \int_{\mathbb{Z}_p}\chi_p(p^kv)dv
        =&\int_{\mathbb{Z}_p}e^{2\pi i\{p^kv\}_p}dv\\
        =&\sum_{j=0}^{\infty}\int_{p^j\mathbb{Z}_p^*}
        e^{2\pi i\{p^kv\}_p}dv\\
        =&\sum_{j=0}^{\infty}p^{-j}\int_{\mathbb{Z}_p^*}
        e^{2\pi i\{p^{k+j}u\}_p}du.
    \end{align*}
    \begin{enumerate}
        \item[]\textbf{Case 1.} If $k\geq0$, then
        \begin{align*}
             \int_{\mathbb{Z}_p}\chi_p(p^kv)dv
             =\sum_{j=0}^{\infty}p^{-j}(1-p^{-1})=1.
        \end{align*}
        \item[]\textbf{Case 2.} If $k<0$, then
        \begin{align*}
            \int_{\mathbb{Z}_p}\chi_p(p^kv)dv
            =p^{-(-k-1)}(-p^{-1})
            +\sum_{j=-k}^{\infty}p^{-j}(1-p^{-1})
            =-p^k+p^k=0.
        \end{align*}
    \end{enumerate}
    Thus,
    \begin{align*}
        \int_{\mathbb{Z}_p}\chi_p(p^kv)dv=\begin{cases}
            1,&k\geq0;\\
            0,&k<0.
        \end{cases}
    \end{align*}
\end{proof}

\begin{lemma}[Character orthogonality]
\label{lem:The orthogonal relation of the characters}
    \begin{align*}
        \int_{\mathbb{Z}_p^n}
        \chi_p(p^k\boldsymbol{x}\cdot\boldsymbol{\xi})d\boldsymbol{x}
        =\begin{cases}
            1,&k\geq-v_p(\boldsymbol{\xi});\\
            0,&k<-v_p(\boldsymbol{\xi}).
        \end{cases}
    \end{align*}
    and
    \begin{align*}
        \int_{\mathbb{Z}_p^n\setminus p\mathbb{Z}_p^n}
        \chi_p(p^k\boldsymbol{x}\cdot\boldsymbol{\xi})d\boldsymbol{x}
        =\begin{cases}
            1-p^{-n},&k\geq-v_p(\boldsymbol{\xi});\\
            -p^{-n},&k=-v_p(\boldsymbol{\xi})-1;\\
             0,&k\leq-v_p(\boldsymbol{\xi})-2.
        \end{cases}
    \end{align*}
\end{lemma}
\begin{proof}
For each nonzero coordinate, write
$\xi_i=p^{v_p(\xi_i)}\eta_i$ with $\eta_i\in\mathbb{Z}_p^*$; each zero
coordinate contributes a factor of $1$ to the product below. Lemma
\ref{lem:The orthogonal relation of the characters for n=1} yields
    \begin{align*}
        \int_{\mathbb{Z}_p^n}
        \chi_p(p^k\boldsymbol{x}\cdot\boldsymbol{\xi})d\boldsymbol{x}
        =&\prod_{i=1}^{n}\int_{\mathbb{Z}_p}
        \chi_p(p^kx_i\xi_i)dx_i\\
        =&\prod_{\xi_i\neq0}\int_{\mathbb{Z}_p}
        \chi_p(p^{k+v_p(\xi_i)}x_i\eta_i)dx_i\\
        =&\prod_{\xi_i\neq0}\int_{\mathbb{Z}_p}
        \chi_p(p^{k+v_p(\xi_i)}y_i)dy_i\\
        =&\begin{cases}
            1,&k\geq\max\limits_{1\leq i\leq n}\{-v_p(\xi_i)\};\\
            0,&k<\max\limits_{1\leq i\leq n}\{-v_p(\xi_i)\}
        \end{cases}\\
        =&\begin{cases}
            1,&k\geq-v_p(\boldsymbol{\xi});\\
            0,&k<-v_p(\boldsymbol{\xi}).
        \end{cases}
    \end{align*}
    Similarly,
    \begin{align*}
        &\int_{\mathbb{Z}_p^n\setminus p\mathbb{Z}_p^n}
        \chi_p(p^k\boldsymbol{x}\cdot\boldsymbol{\xi})d\boldsymbol{x}\\
        =&\int_{\mathbb{Z}_p^n}
        \chi_p(p^k\boldsymbol{x}\cdot\boldsymbol{\xi})d\boldsymbol{x}
        -p^{-n}\int_{\mathbb{Z}_p^n}
        \chi_p(p^{k+1}\boldsymbol{y}\cdot\boldsymbol{\xi})d\boldsymbol{y}\\
        =&\begin{cases}
            1,&k\geq-v_p(\boldsymbol{\xi});\\
            0,&k<-v_p(\boldsymbol{\xi})
        \end{cases}
        -\begin{cases}
            p^{-n},&k\geq-v_p(\boldsymbol{\xi})-1;\\
            0,&k<-v_p(\boldsymbol{\xi})-1
        \end{cases}\\
        =&\begin{cases}
            1-p^{-n},&k\geq-v_p(\boldsymbol{\xi});\\
            -p^{-n},&k=-v_p(\boldsymbol{\xi})-1;\\
             0,&k\leq-v_p(\boldsymbol{\xi})-2.
        \end{cases}
    \end{align*}
\end{proof}

\begin{lemma}
\label{lem:hat_varphi=0_k<-vp_xi}
Assume that $\varphi_k\in L^1(\mathbb{Q}_p^n)$ and that $\varphi_k$ is
constant on every ball $\boldsymbol{x}_0+p^k\mathbb{Z}_p^n$. Then
$\hat{\varphi}_k(\boldsymbol{\xi})=0$ for
$k<-v_p(\boldsymbol{\xi})$.
\end{lemma}
\begin{proof}
    By the assumed local constancy of $\varphi_k$ on each ball
    $\boldsymbol{x}_0+p^k\mathbb{Z}_p^n$,
    \begin{align*}
        \hat{\varphi}_k(\boldsymbol{\xi})
        =&\int_{\mathbb{Q}_p^n}\varphi_k(\boldsymbol{x})
        \chi_p(\boldsymbol{x}\cdot\boldsymbol{\xi})d\boldsymbol{x}\\
        =&\sum_{\boldsymbol{x}_0\in
        \mathbb{Q}_p^n/p^k\mathbb{Z}_p^n}
        \varphi_k(\boldsymbol{x}_0)
        \int_{\boldsymbol{x}_0+p^k\mathbb{Z}_p^n}
        \chi_p(\boldsymbol{x}\cdot\boldsymbol{\xi})d\boldsymbol{x}.
    \end{align*}
    If $k<-v_p(\boldsymbol{\xi})$, Lemma
    \ref{lem:The orthogonal relation of the characters} implies
    \begin{align*}
        &\int_{\boldsymbol{x}_0+p^k\mathbb{Z}_p^n}
        \chi_p(\boldsymbol{x}\cdot\boldsymbol{\xi})d\boldsymbol{x}\\
        =&p^{-nk}\chi_p(\boldsymbol{x}_0\cdot\boldsymbol{\xi})
        \int_{\mathbb{Z}_p^n}
        \chi_p(p^k\boldsymbol{y}\cdot\boldsymbol{\xi})d\boldsymbol{y}=0.
    \end{align*}
    Hence $\hat{\varphi}_k(\boldsymbol{\xi})=0$ for
    $k<-v_p(\boldsymbol{\xi})$.
\end{proof}

\begin{theorem}
\label{thm:Fourier transform of varphi_k}
For $f\in \mathcal{S}(\mathbb{Q}_p^n)$, the solution of
\eqref{eq:the_Dirichlet_problem_extension} has Fourier transform
    \begin{align*}
    \hat{\varphi}_k(\boldsymbol{\xi})=\begin{cases}
        \hat{f}(\boldsymbol{\xi})
        (1-p^{-s(k+1)}\|\boldsymbol{\xi}\|_p^s),
        &k\geq-v_p(\boldsymbol{\xi});\\
        0,&k<-v_p(\boldsymbol{\xi}).
    \end{cases}
\end{align*}
At $\boldsymbol{\xi}=\boldsymbol{0}$, the first line is understood to be
$\hat\varphi_k(\boldsymbol{0})=\hat f(\boldsymbol{0})$.
\end{theorem}
\begin{proof}
We first carry out the Fourier calculation for locally constant integrable
layers. The Poisson representation derived later confirms that the formula
obtained below indeed defines such layers. Taking the Fourier transform of
the harmonic equation in \eqref{eq:the_Dirichlet_problem_extension}, we obtain
\begin{align*}
    p^{k(s-n)}
    (\hat{\varphi}_{k-1}(\boldsymbol{\xi})
    -\hat{\varphi}_{k}(\boldsymbol{\xi}))+p^{(k+1)(s-n)}
    \left(\hat{\varphi}_{k+1}(\boldsymbol{\xi})
    \sum_{\boldsymbol{a}\in\mathbb{F}_p^n}
    \chi_p(-p^k\boldsymbol{a}\cdot\boldsymbol{\xi})
    -p^n\hat{\varphi}_{k}(\boldsymbol{\xi})\right)=0,
\end{align*}
together with the Fourier boundary condition
$\hat\varphi_k(\boldsymbol\xi)\to\hat f(\boldsymbol\xi)$ as
$k\to+\infty$.

Fix $\boldsymbol{\xi}\neq\boldsymbol{0}$. Only the following two values of
the finite character sum are needed:
\begin{align*}
    \sum_{\boldsymbol{a}\in\mathbb{F}_p^n}
    \chi_p(-p^k\boldsymbol{a}\cdot\boldsymbol{\xi})
    =\begin{cases}
        p^n,&k\geq-v_p(\boldsymbol{\xi});\\
        0,&k=-v_p(\boldsymbol{\xi})-1.
    \end{cases}
\end{align*}
If $k\geq-v_p(\boldsymbol{\xi})$, every character is trivial. If
$k=-v_p(\boldsymbol{\xi})-1$, at least one coordinate yields a factor equal
to the sum of all $p$-th roots of unity, which is zero. We do not need a uniform formula
when $k\leq-v_p(\boldsymbol{\xi})-2$.

We distinguish three cases. If $k\leq-v_p(\boldsymbol{\xi})-2$, Lemma
\ref{lem:hat_varphi=0_k<-vp_xi} gives
$\hat\varphi_{k-1}=\hat\varphi_k=\hat\varphi_{k+1}=0$, so the transformed
equation is automatic. If $k=-v_p(\boldsymbol{\xi})-1$, the first two
Fourier transforms vanish and the finite character sum is zero, so the
equation again holds. If $k\geq-v_p(\boldsymbol{\xi})$, the finite character
sum is $p^n$, and we obtain
\begin{align*}
    p^{k(s-n)}
    (\hat{\varphi}_{k-1}(\boldsymbol{\xi})
    -\hat{\varphi}_{k}(\boldsymbol{\xi}))+p^{(k+1)(s-n)}p^n
    (\hat{\varphi}_{k+1}(\boldsymbol{\xi})
    -\hat{\varphi}_{k}(\boldsymbol{\xi}))=0,
\end{align*}
which simplifies to
\begin{align*}
    \hat{\varphi}_{k+1}(\boldsymbol{\xi})
    -\hat{\varphi}_{k}(\boldsymbol{\xi})
    =p^{-s}(\hat{\varphi}_{k}(\boldsymbol{\xi})
    -\hat{\varphi}_{k-1}(\boldsymbol{\xi})).
\end{align*}
By Lemma \ref{lem:hat_varphi=0_k<-vp_xi},
\begin{equation}
\label{eq:hatvarphi_k-hatvarphi_k-1}
\begin{aligned}
    \hat{\varphi}_k(\boldsymbol{\xi})
    -\hat{\varphi}_{k-1}(\boldsymbol{\xi})
    =&\bigl(\hat{\varphi}_{-v_p(\boldsymbol{\xi})}
    (\boldsymbol{\xi})
    -\hat{\varphi}_{-v_p(\boldsymbol{\xi})-1}
    (\boldsymbol{\xi})\bigr)
    p^{-s(k+v_p(\boldsymbol{\xi}))}\\
    =&\hat{\varphi}_{-v_p(\boldsymbol{\xi})}(\boldsymbol{\xi})
    p^{-s(k+v_p(\boldsymbol{\xi}))},
\end{aligned}
\end{equation}
for every $k\geq-v_p(\boldsymbol{\xi})$. Therefore,
\begin{align*}
    \hat{f}(\boldsymbol{\xi})
    =&\sum_{k=-v_p(\boldsymbol{\xi})}^{+\infty}
    (\hat{\varphi}_k(\boldsymbol{\xi})
    -\hat{\varphi}_{k-1}(\boldsymbol{\xi}))\\
    =&\hat{\varphi}_{-v_p(\boldsymbol{\xi})}(\boldsymbol{\xi})
    \sum_{k=-v_p(\boldsymbol{\xi})}^{+\infty}
    p^{-s(k+v_p(\boldsymbol{\xi}))}\\
    =&\frac{1}{1-p^{-s}}
    \hat{\varphi}_{-v_p(\boldsymbol{\xi})}(\boldsymbol{\xi}).
\end{align*}
Thus, for $k\geq-v_p(\boldsymbol{\xi})$, we have
\begin{align*}
    \hat{\varphi}_k(\boldsymbol{\xi})
    =&\sum_{i=-v_p(\boldsymbol{\xi})}^{k}
    (\hat{\varphi}_i(\boldsymbol{\xi})
    -\hat{\varphi}_{i-1}(\boldsymbol{\xi}))\\
    =&\hat{\varphi}_{-v_p(\boldsymbol{\xi})}(\boldsymbol{\xi})
    \sum_{i=-v_p(\boldsymbol{\xi})}^{k}
    p^{-s(i+v_p(\boldsymbol{\xi}))}\\
    =&(1-p^{-s})\hat{f}(\boldsymbol{\xi})
    \frac{1-p^{-s(k+v_p(\boldsymbol{\xi})+1)}}{1-p^{-s}}\\
    =&\hat{f}(\boldsymbol{\xi})
    (1-p^{-s(k+1)}\|\boldsymbol{\xi}\|_p^s).
\end{align*}

At $\boldsymbol{\xi}=\boldsymbol{0}$, the finite character sum equals
$p^n$. The two fundamental solutions of the recurrence are $1$ and $p^{-sk}$;
boundedness as $k\to-\infty$ excludes the latter, while the boundary
condition fixes the remaining constant as $\hat f(\boldsymbol{0})$.

Conversely, define $\varphi_k$ by taking the inverse Fourier transform of
the formula in the theorem. Its Fourier support is contained in
$\{\|\boldsymbol{\xi}\|_p\leq p^k\}$, so $\varphi_k$ is constant on every
ball $\boldsymbol{x}+p^k\mathbb{Z}_p^n$. The same three cases above verify
$\Delta_w\varphi=0$ at every level. Choose $N$ such that
$\operatorname{supp}\hat f\subseteq
\{\|\boldsymbol{\xi}\|_p\leq p^N\}$. For $k\geq N$ we have the exact
identity
\begin{align*}
    f-\varphi_k=p^{-s(k+1)}D^sf.
\end{align*}
Fourier inversion and Parseval's identity then imply that
$\varphi_k\to f$ uniformly and in $L^2$ as $k\to+\infty$. On the other
hand,
\begin{align*}
    \|\varphi_k\|_\infty
    \leq\int_{\|\boldsymbol{\xi}\|_p\leq p^k}
    |\hat f(\boldsymbol{\xi})|d\boldsymbol{\xi}
    \leq p^{nk}\|\hat f\|_\infty\longrightarrow0
\end{align*}
as $k\to-\infty$. Continuity at the distinguished end along arbitrary
convergent sequences, as well as uniqueness among bounded functions
continuous on the end compactification, will follow from the Poisson
formula below.
\end{proof}

Using Theorem \ref{thm:Fourier transform of varphi_k}, we compute
\begin{align*}
    \widehat{\partial_{n,s}^{(p)}\varphi}(\boldsymbol{\xi})
    =&\lim_{k\to+\infty}p^{(k+1)s}
    \left(\hat{f}(\boldsymbol{\xi})
    -\hat{\varphi}_k(\boldsymbol{\xi})\right)\\
    =&\lim_{k\to+\infty}p^{(k+1)s}
    \left(\hat{f}(\boldsymbol{\xi})
    -\hat{f}(\boldsymbol{\xi})
    (1-p^{-s(k+1)}\|\boldsymbol{\xi}\|_p^s)\right)\\
    =&\lim_{k\to+\infty}p^{(k+1)s}
    \hat{f}(\boldsymbol{\xi})p^{-s(k+1)}
    \|\boldsymbol{\xi}\|_p^s\\
    =&\|\boldsymbol{\xi}\|_p^s\hat{f}(\boldsymbol{\xi}),
\end{align*}
and hence
\begin{align*}
    \partial_{n,s}^{(p)}\varphi(\boldsymbol{x})=D^sf(\boldsymbol{x}).
\end{align*}
The Fourier multiplier formula and the geometric-series calculation below
verify the finite-energy and boundary hypotheses of Theorem
\ref{thm:Green's first identity for Tpn}. Since $\varphi$ is harmonic, the
bulk term in that identity vanishes. Combining Theorem
\ref{thm:Green's first identity for Tpn} with Parseval's identity therefore
gives
\begin{align*}
    \mathcal{E}[\varphi]
    =&\sum_{k\in\mathbb{Z}}p^{ks}
    \int_{\mathbb{Q}_p^n}
    |\varphi_k(\boldsymbol{x})-\varphi_{k-1}(\boldsymbol{x})|^2
    d\boldsymbol{x}\\
    =&\frac{1}{2}\sum_{u\in \mathcal{T}_{p^n}}
    \sum_{v\sim u}w_{uv}|\varphi(u)-\varphi(v)|^2\\
    =&(1-p^{-s})\int_{\mathbb{Q}_p^n}
    \varphi|_{\partial\mathcal T_{p^n}}(\boldsymbol{x})
    \partial_{n,s}^{(p)}\varphi(\boldsymbol{x})d\boldsymbol{x}\\
    =&(1-p^{-s})\int_{\mathbb{Q}_p^n}
    f(\boldsymbol{x})D^sf(\boldsymbol{x})d\boldsymbol{x}\\
    =&(1-p^{-s})\int_{\mathbb{Q}_p^n}
    \|\boldsymbol{\xi}\|_p^s|\hat{f}(\boldsymbol{\xi})|^2
    d\boldsymbol{\xi}.
\end{align*}
Indeed, the Fourier formula yields
\begin{align*}
\hat\varphi_k(\boldsymbol\xi)-\hat\varphi_{k-1}(\boldsymbol\xi)
=\begin{cases}
0,&k<-v_p(\boldsymbol\xi);\\
(1-p^{-s})p^{-s(k+v_p(\boldsymbol\xi))}\hat f(\boldsymbol\xi),
&k\geq-v_p(\boldsymbol\xi).
\end{cases}
\end{align*}
Summing the resulting geometric series over $k$ gives
$(1-p^{-s})\|\boldsymbol\xi\|_p^s|\hat f(\boldsymbol\xi)|^2$,
and thus verifies the energy constant directly.

\subsection{The Poisson kernel via inverse Fourier transform}
Consider the Poisson equation in the weak boundary sense
\begin{align}
\label{eq:Poisson_equation}
\begin{cases}
    \Delta_wP(v)=0,&v\in \mathcal{T}_{p^n};\\
    P_k(\boldsymbol{y})d\boldsymbol{y}
    \rightharpoonup\delta_{\boldsymbol{0}},&k\to+\infty;\\
    \displaystyle\int_{\mathbb{Q}_p^n}P_k(\boldsymbol y)
    d\boldsymbol y=1,&k\in\mathbb Z;\\
    \|P_k\|_\infty\to0,&k\to-\infty.
\end{cases}
\end{align}
Here $\rightharpoonup$ denotes weak-$*$ convergence of finite Radon measures,
tested against functions in $C_0(\mathbb{Q}_p^n)$.

\begin{theorem}
\label{thm:Poisson_kernel}
    \begin{align*}
        P_k(\boldsymbol{y})
        =\frac{1-p^{-s}}{1-p^{-(s+n)}}\begin{cases}
            p^{nk},&v_p(\boldsymbol{y})\geq k;\\
            p^{-sk}\|\boldsymbol{y}\|_p^{-s-n},
            &v_p(\boldsymbol{y})<k.
        \end{cases}
    \end{align*}
\end{theorem}
\begin{proof}
We define $P_k$ by prescribing its Fourier transform:
\begin{align*}
    \hat{P}_k(\boldsymbol{\xi})=\begin{cases}
        1-p^{-s(k+1)}\|\boldsymbol{\xi}\|_p^s,
        &k\geq-v_p(\boldsymbol{\xi});\\
        0,&k<-v_p(\boldsymbol{\xi}).
    \end{cases}
\end{align*}
At $\boldsymbol\xi=\boldsymbol0$, the first line is understood to equal
$1$. The same calculation in the three frequency regimes considered above
shows directly that these layers are weighted harmonic. Fourier inversion
gives
\begin{align*}
    &P_k(\boldsymbol{y})\\
    =&\int_{\mathbb{Q}_p^n}\hat{P}_k(\boldsymbol{\xi})
    \chi_p(-\boldsymbol{\xi}\cdot\boldsymbol{y})d\boldsymbol{\xi}\\
    =&\int_{p^{-k}\mathbb{Z}_p^n}
    (1-p^{-s(k+1)}\|\boldsymbol{\xi}\|_p^s)
    \chi_p(-\boldsymbol{\xi}\cdot\boldsymbol{y})d\boldsymbol{\xi}\\
    =&p^{nk}\int_{\mathbb{Z}_p^n}
    (1-p^{-s}\|\boldsymbol{\eta}\|_p^s)
    \chi_p(-p^{-k}\boldsymbol{\eta}\cdot\boldsymbol{y})
    d\boldsymbol{\eta}\\
    =&p^{nk}\bigg(
    \int_{\mathbb{Z}_p^n}
    \chi_p(-p^{-k}\boldsymbol{\eta}\cdot\boldsymbol{y})
    d\boldsymbol{\eta}-p^{-s}\sum_{j=0}^{\infty}p^{-js}
    \int_{p^j(\mathbb{Z}_p^n\setminus p\mathbb{Z}_p^n)}
    \chi_p(-p^{-k}\boldsymbol{\eta}\cdot\boldsymbol{y})
    d\boldsymbol{\eta}\bigg)\\
    =&p^{nk}\bigg(
    \int_{\mathbb{Z}_p^n}
    \chi_p(-p^{-k}\boldsymbol{\eta}\cdot\boldsymbol{y})
    d\boldsymbol{\eta}-p^{-s}\sum_{j=0}^{\infty}p^{-j(s+n)}
    \int_{\mathbb{Z}_p^n\setminus p\mathbb{Z}_p^n}
    \chi_p(-p^{-k+j}\boldsymbol{\zeta}\cdot\boldsymbol{y})
    d\boldsymbol{\zeta}\bigg)\\
    =&p^{nk}\left(\begin{cases}
            1,&-k\geq-v_p(\boldsymbol{y});\\
            0,&-k<-v_p(\boldsymbol{y})
        \end{cases}-p^{-s}\sum_{j=0}^{\infty}p^{-j(s+n)}
        \begin{cases}
            1-p^{-n},&-k+j\geq-v_p(\boldsymbol{y});\\
            -p^{-n},&-k+j=-v_p(\boldsymbol{y})-1;\\
             0,&-k+j\leq-v_p(\boldsymbol{y})-2
        \end{cases}\right)\\
    =&p^{nk}\left(\begin{cases}
            1,&v_p(\boldsymbol{y})\geq k;\\
            0,&v_p(\boldsymbol{y})<k
        \end{cases}-p^{-s}\begin{cases}
            (1-p^{-n})\displaystyle\sum_{j=0}^{\infty}p^{-j(s+n)},
            &v_p(\boldsymbol{y})\geq k;\\
            -p^{-n}p^{-(k-v_p(\boldsymbol{y})-1)(s+n)}
            +(1-p^{-n})\displaystyle
            \sum_{j=k-v_p(\boldsymbol{y})}^{\infty}p^{-j(s+n)},
            &v_p(\boldsymbol{y})<k
        \end{cases}\right)\\
    =&p^{nk}\begin{cases}
            1-\dfrac{p^{-s}(1-p^{-n})}{1-p^{-(s+n)}},
            &v_p(\boldsymbol{y})\geq k;\\
            p^{-(k-v_p(\boldsymbol{y}))(s+n)}
            -\dfrac{p^{-s}(1-p^{-n})}{1-p^{-(s+n)}}
            p^{-(k-v_p(\boldsymbol{y}))(s+n)},
            &v_p(\boldsymbol{y})<k
        \end{cases}\\
    =&\frac{1-p^{-s}}{1-p^{-(s+n)}}\begin{cases}
            p^{nk},&v_p(\boldsymbol{y})\geq k;\\
            p^{-sk}\|\boldsymbol{y}\|_p^{-s-n},
            &v_p(\boldsymbol{y})<k.
        \end{cases}
\end{align*}

The formula shows that $P_k$ is nonnegative and integrable. Using
$|\{\|\boldsymbol y\|_p=p^j\}|=(1-p^{-n})p^{nj}$, we obtain
\begin{align*}
    \int_{\mathbb{Q}_p^n}P_k(\boldsymbol y)d\boldsymbol y
    =&\frac{1-p^{-s}}{1-p^{-(s+n)}}
    \left(1+(1-p^{-n})p^{-sk}
    \sum_{j=-k+1}^{\infty}p^{-sj}\right)\\
    =&\frac{1-p^{-s}}{1-p^{-(s+n)}}
    \left(1+\frac{(1-p^{-n})p^{-s}}{1-p^{-s}}\right)=1.
\end{align*}
For every compact-open neighborhood $U$ of the origin, the second branch of
the formula gives, for all sufficiently large $k$,
\begin{align*}
    \int_{U^c}P_k(\boldsymbol y)d\boldsymbol y\longrightarrow0
    \qquad(k\to+\infty).
\end{align*}
Hence
$P_k(\boldsymbol y)d\boldsymbol y\rightharpoonup\delta_{\boldsymbol0}$.
Also,
\begin{align*}
    \|P_k\|_\infty
    =\frac{1-p^{-s}}{1-p^{-(s+n)}}p^{nk}\longrightarrow0
    \qquad(k\to-\infty).
\end{align*}

This Poisson kernel is unique among nonnegative harmonic layer densities of
mass one with the stated weak trace. Indeed, the Fourier transform of any
such layer vanishes below the same support threshold. Since these measures
all have mass one and converge against $C_0(\mathbb{Q}_p^n)$ to the
probability measure $\delta_{\boldsymbol 0}$, the tail of the family is
tight as $k\to+\infty$. Consequently, the convergence also holds against
bounded continuous characters, and the corresponding Fourier transforms
converge to $1$. For each fixed nonzero frequency, the support threshold,
the recurrence above, and this boundary limit uniquely determine the
displayed multiplier. The mass normalization fixes its value at zero
frequency.
\end{proof}

\begin{remark}
    When $s=n=1$,
    \begin{align*}
    P_k(y)=&\frac{1-p^{-1}}{1-p^{-2}}\begin{cases}
            p^k,&v_p(y)\geq k;\\
            p^{-k}|y|_p^{-2},&v_p(y)<k
        \end{cases}\\
        =&\frac{p}{p+1}\min\{p^k,p^{-k}|y|_p^{-2}\},
    \end{align*}
    which agrees with the formula in \cite{MR1003429} after matching
    conventions.
\end{remark}

The normalization
$\int_{\mathbb{Q}_p^n}P_k(\boldsymbol{y})d\boldsymbol{y}=1$ gives
\begin{align*}
    f(\boldsymbol{x})
    =\int_{\mathbb{Q}_p^n}P_k(\boldsymbol{y})
    f(\boldsymbol{x})d\boldsymbol{y},
\end{align*}
and
\begin{align*}
    \varphi_k(\boldsymbol{x})
    =(P_k*f)(\boldsymbol{x})
    =\int_{\mathbb{Q}_p^n}P_k(\boldsymbol{y})
    f(\boldsymbol{x}-\boldsymbol{y})d\boldsymbol{y}.
\end{align*}
Because $f$ is uniformly locally constant and $\{P_k\}$ is an approximate
identity, $\varphi_k\to f$ uniformly as $k\to+\infty$. Also,
\begin{align*}
    \|\varphi_k\|_\infty
    \leq\|P_k\|_\infty\|f\|_1\longrightarrow0
\end{align*}
as $k\to-\infty$. If
$\operatorname{supp}f\subseteq
\{\|\boldsymbol y\|_p\leq p^M\}$ and
$\|\boldsymbol x\|_p>p^M$, then ultrametricity and the explicit kernel give
\begin{align*}
    |\varphi_k(\boldsymbol x)|
    \leq\frac{1-p^{-s}}{1-p^{-(s+n)}}\|f\|_1
    \min\left\{p^{nk},
    p^{-sk}\|\boldsymbol x\|_p^{-n-s}\right\}\leq C\|f\|_1\|\boldsymbol x\|_p^{-n}.
\end{align*}
These estimates establish continuity at the distinguished end, including
along sequences in which the height and the horizontal coordinate vary
simultaneously. Finally, the
difference of two bounded end-continuous solutions with the same boundary
values is a bounded harmonic function vanishing at every end. The maximum
principle on the compactified tree, applied to the real and imaginary parts
when necessary, shows that this difference is identically zero.

We now identify the Dirichlet-to-Neumann map in real space. We have
\begin{align*}
    f(\boldsymbol{x})-\varphi_k(\boldsymbol{x})
    =\frac{1-p^{-s}}{1-p^{-(s+n)}}\bigg(p^{nk}
    \int_{\|\boldsymbol{y}\|_p\leq p^{-k}}
    \bigl(f(\boldsymbol{x})-f(\boldsymbol{x}-\boldsymbol{y})\bigr)
    d\boldsymbol{y}+p^{-sk}\int_{\|\boldsymbol{y}\|_p>p^{-k}}
    \frac{f(\boldsymbol{x})-f(\boldsymbol{x}-\boldsymbol{y})}
    {\|\boldsymbol{y}\|_p^{s+n}}d\boldsymbol{y}\bigg).
\end{align*}
Since $f\in \mathcal{S}(\mathbb{Q}_p^n)$, there exists $k_0$ such that
$f(\boldsymbol{x})=f(\boldsymbol{x}-\boldsymbol{y})$ for all
$\boldsymbol{x}$ whenever $\|\boldsymbol{y}\|_p\leq p^{-k_0}$.
For $k\geq k_0$, the first integral therefore vanishes, and the second
integral can be extended to all of $\mathbb{Q}_p^n$. Hence
\begin{align*}
    \partial_{n,s}^{(p)}\varphi(\boldsymbol{x})
    =&\lim_{k\to+\infty}p^{(k+1)s}
    \left(f(\boldsymbol{x})-\varphi_k(\boldsymbol{x})\right)\\
    =&\frac{p^s-1}{1-p^{-(s+n)}}
    \int_{\mathbb{Q}_p^n}
    \frac{f(\boldsymbol{x})-f(\boldsymbol{x}-\boldsymbol{y})}
    {\|\boldsymbol{y}\|_p^{s+n}}d\boldsymbol{y}\\
    =&D^sf(\boldsymbol{x}).
\end{align*}

\subsection{A direct derivation of the Poisson kernel}
We solve the Poisson equation \eqref{eq:Poisson_equation} directly and derive
the formula in Theorem \ref{thm:Poisson_kernel}. Relative to the path
$\infty\to\boldsymbol{0}$, the vertices of $\mathcal{T}_{p^n}$ fall into
the following two classes:
\begin{enumerate}
    \item $v=(\boldsymbol{x},k)\in\{\infty\to\boldsymbol{0}\}$.
    \item $v=(\boldsymbol{x},j)\notin
    \{\infty\to\boldsymbol{0}\}$. Such a vertex branches off from the path
    $\infty\to\boldsymbol{0}$ at $(\boldsymbol{y},k)$ with $k<j$ and lies
    at distance $r=j-k$ from the branching point.
\end{enumerate}
By symmetry with respect to the path
$\infty\to\boldsymbol{0}$, the value of the Poisson kernel at such a vertex
depends only on the branching level $k$ and the distance $r$ from the path.
We therefore use separation of variables from the outset and seek a
solution of the form
\begin{align*}
    P(v)=A(k)B(r).
\end{align*}
We shall verify that this ansatz yields a function satisfying the harmonic
equation, the normalization, and all the boundary conditions in
\eqref{eq:Poisson_equation}. The uniqueness established above then
identifies this function as the Poisson kernel.

\begin{enumerate}
    \item[]\textbf{Case 1.}
    $v\notin\{\infty\to\boldsymbol{0}\}$ (determine $B(r)$).
    As shown in Figure \ref{fig:case1}, the height of $v$ is $k+r$.
    The edge joining $v$ to the vertex at distance $r-1$ has weight
    $p^{(k+r)(s-n)}$, whereas each of the $p^n$ edges joining $v$ to
    vertices at distance $r+1$ has weight
    $p^{(k+r+1)(s-n)}$. The harmonic equation in
    \eqref{eq:Poisson_equation} therefore becomes
    \begin{align*}
        p^{(k+r)(s-n)}
        \bigl(A(k)B(r-1)-A(k)B(r)\bigr)+p^{(k+r+1)(s-n)}p^n
        \bigl(A(k)B(r+1)-A(k)B(r)\bigr)=0.
    \end{align*}
    Dividing both sides by $p^{(k+r)(s-n)}A(k)$ and simplifying, we obtain
    \begin{align*}
        p^sB(r+1)-(1+p^s)B(r)+B(r-1)=0.
    \end{align*}
    This second-order difference equation has characteristic equation
    \begin{align*}
        p^s\lambda^2-(1+p^s)\lambda+1=0,
    \end{align*}
    whose roots are
    \begin{align*}
        \lambda=1\qquad\text{or}\qquad\lambda=p^{-s}.
    \end{align*}
    Thus
    \begin{align*}
        B(r)=C_1+C_2p^{-sr}.
    \end{align*}
    The weak boundary condition
    $P_k(\boldsymbol{y})d\boldsymbol{y}
    \rightharpoonup\delta_{\boldsymbol{0}}$ requires the mass of every fixed
    compact-open boundary ball disjoint from $\boldsymbol{0}$ to tend to
    zero. If $C_1\neq0$, the mass of such an off-axis ball would instead
    converge to a nonzero limit. Hence $C_1=0$. Absorbing $C_2$ into $A(k)$
    and taking $B(0)=1$, we obtain
    \begin{align*}
        B(r)=p^{-sr}.
    \end{align*}
    Consequently,
    \begin{align*}
        P(v)=A(k)B(r)=A(k)p^{-sr}.
    \end{align*}
    \begin{figure}[H]
        \centering
        \includegraphics[width=0.6\textwidth]{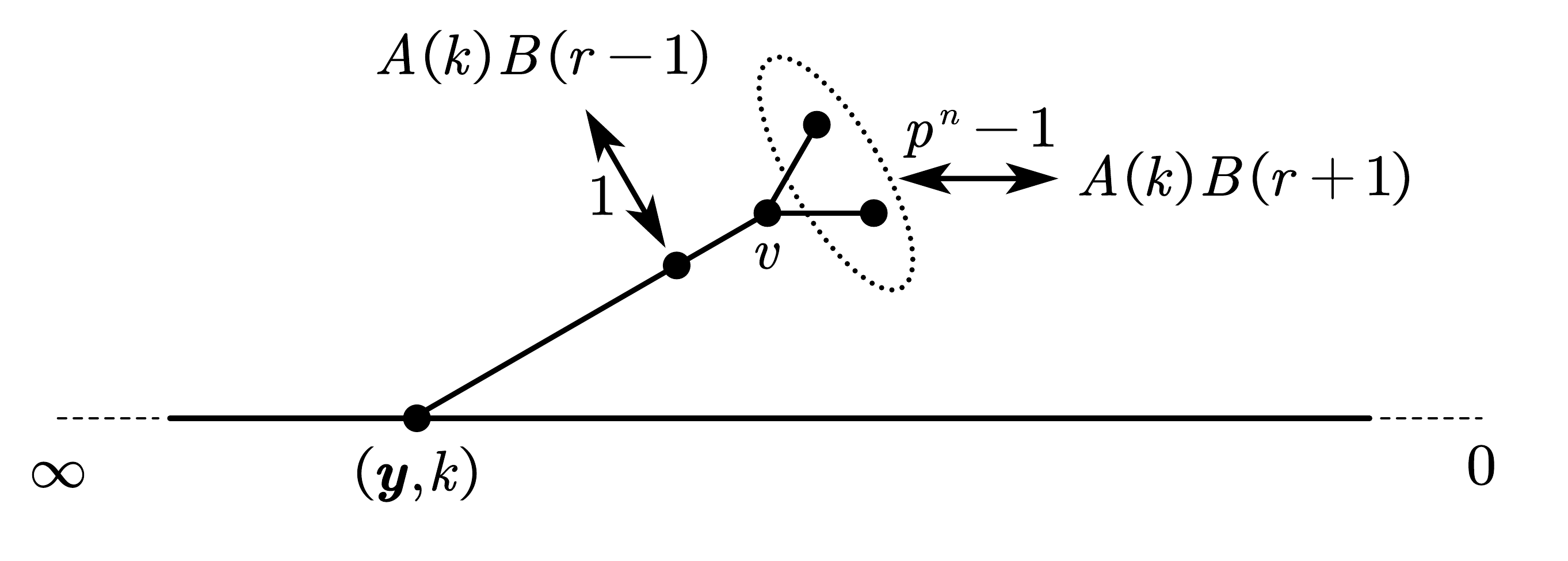}
        \caption{An off-axis vertex
        $v\notin\{\infty\to\boldsymbol{0}\}$}
        \label{fig:case1}
    \end{figure}

    \item[]\textbf{Case 2.}
    $v\in\{\infty\to\boldsymbol{0}\}$ (determine $A(k)$). As shown in Figure
    \ref{fig:case2}, we have $r=0$, and equation
    \eqref{eq:Poisson_equation} becomes
    \begin{align*}
        &p^{k(s-n)}(A(k-1)B(0)-A(k)B(0))\\
        &+p^{(k+1)(s-n)}
        \left(A(k+1)B(0)+(p^n-1)A(k)B(1)-p^nA(k)B(0)\right)=0.
    \end{align*}
    Dividing both sides by $p^{k(s-n)}$ and substituting
    $B(0)=1$, $B(1)=p^{-s}$, we obtain
    \begin{align*}
        p^{s-n}A(k+1)-(p^s+p^{-n})A(k)+A(k-1)=0.
    \end{align*}
    The corresponding characteristic equation is
    \begin{align*}
        p^{s-n}\mu^2-(p^s+p^{-n})\mu+1=0
        \quad\Longrightarrow\quad\mu=p^n\text{ or }p^{-s}.
    \end{align*}
    Hence
    \begin{align*}
        A(k)=C_1p^{nk}+C_2p^{-sk}.
    \end{align*}
    Since $\|P_k\|_\infty\to0$ as $k\to-\infty$, we have
    $A(k)\to0$, and therefore $C_2=0$. Hence
    \begin{align*}
        A(k)=C_1p^{nk}.
    \end{align*}
    \begin{figure}[H]
        \centering
        \includegraphics[width=0.6\textwidth]{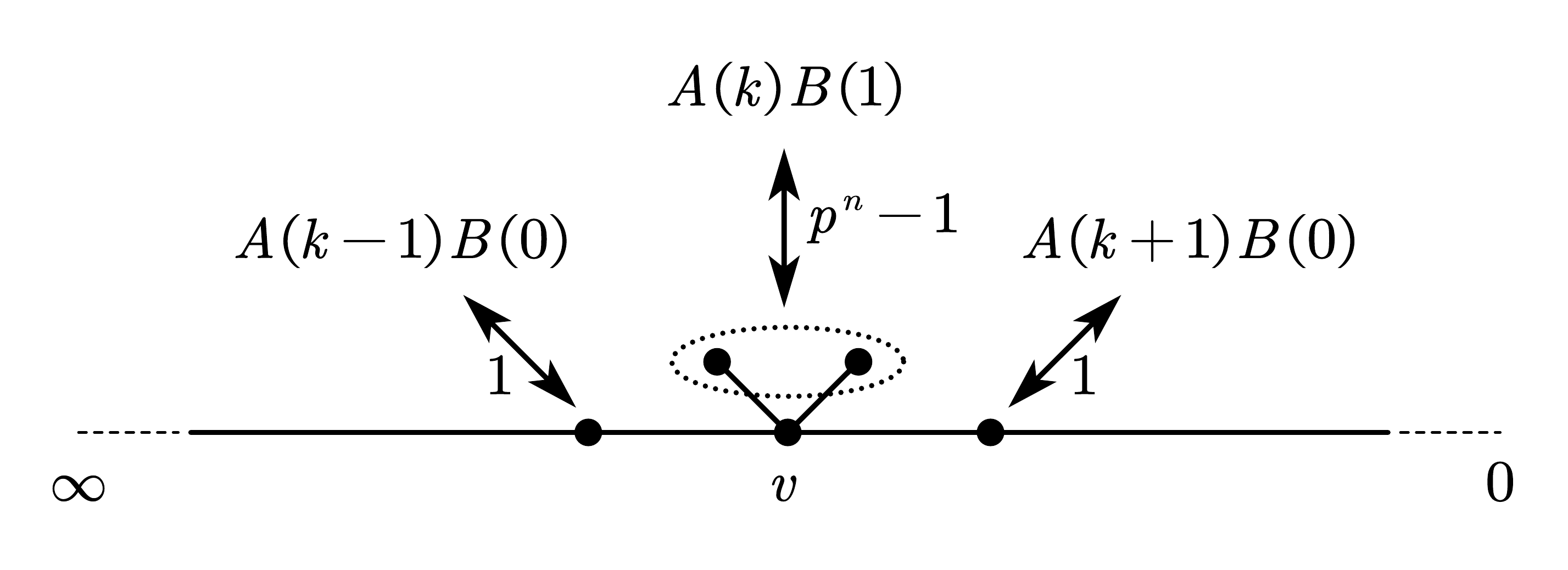}
        \caption{A vertex $v\in\{\infty\to\boldsymbol{0}\}$ on the
        distinguished geodesic}
        \label{fig:case2}
    \end{figure}
\end{enumerate}

At each height $l$, the preceding solution gives
\begin{align*}
    \int_{\mathbb{Q}_p^n}P(\boldsymbol{y},l)d\boldsymbol{y}
    =&p^{-nl}\sum_{\boldsymbol{y}\in
    \mathbb{Q}_p^n/p^l\mathbb{Z}_p^n}P(\boldsymbol{y},l)\\
    =&p^{-nl}\left(A(l)B(0)+
    \sum_{k=-\infty}^{l-1}(p^n-1)p^{n(l-1-k)}A(k)B(l-k)\right)\\
    =&p^{-nl}\left(C_1p^{nl}+
    \sum_{k=-\infty}^{l-1}(p^n-1)p^{n(l-1-k)}
    C_1p^{nk-s(l-k)}\right)\\
    =&C_1\left(1+(1-p^{-n})
    \sum_{k=-\infty}^{l-1}p^{-s(l-k)}\right)\\
    =&C_1\left(1+(1-p^{-n})
    \sum_{m=1}^{+\infty}p^{-sm}\right)\\
    =&C_1\frac{1-p^{-(s+n)}}{1-p^{-s}}.
\end{align*}
This expression is independent of $l$. By the normalization in
\eqref{eq:Poisson_equation}, the displayed integral equals $1$ for every
$l$, and hence
\begin{align*}
    C_1=\frac{1-p^{-s}}{1-p^{-(s+n)}}.
\end{align*}
Thus,
\begin{align*}
    P(v)=\frac{1-p^{-s}}{1-p^{-(s+n)}}p^{nk-sr}.
\end{align*}
We now distinguish the two classes of vertices:
\begin{enumerate}
    \item $v=(\boldsymbol{x},k)\in\{\infty\to\boldsymbol{0}\}$. This
    is equivalent to the inclusion
    $\boldsymbol{0}\in\boldsymbol{x}+p^k\mathbb{Z}_p^n$, and therefore
    $\|\boldsymbol{x}\|_p\leq p^{-k}$. Then
    \begin{align*}
        P(v)=A(k)B(0)
        =\frac{1-p^{-s}}{1-p^{-(s+n)}}p^{nk}.
    \end{align*}
    \item $v=(\boldsymbol{x},k)\notin
    \{\infty\to\boldsymbol{0}\}$. Then it branches off from the path
    $\infty\to\boldsymbol{0}$ at $(\boldsymbol{y},j)$ with $j<k$. This
    means that $\boldsymbol{x}+p^k\mathbb{Z}_p^n$ does not contain the
    origin, whereas $\boldsymbol{x}+p^j\mathbb{Z}_p^n$ is the smallest ball
    on the chain that does. Therefore
    $\|\boldsymbol{x}\|_p=p^{-j}>p^{-k}$, and
    \begin{align*}
        P(v)=A(j)B(k-j)
        =&\frac{1-p^{-s}}{1-p^{-(s+n)}}p^{nj-s(k-j)}\\
        =&\frac{1-p^{-s}}{1-p^{-(s+n)}}
        p^{-sk}\|\boldsymbol{x}\|_p^{-s-n}.
    \end{align*}
\end{enumerate}
Thus,
\begin{align*}
    P_k(\boldsymbol{x})
    =\frac{1-p^{-s}}{1-p^{-(s+n)}}\begin{cases}
        p^{nk},&v_p(\boldsymbol{x})\geq k;\\
        p^{-sk}\|\boldsymbol{x}\|_p^{-s-n},
        &v_p(\boldsymbol{x})<k.
    \end{cases}
\end{align*}
\section{An alternative extension method}
Motivated by the holographic construction in \cite{MR3631398}, we
present a second extension method. As in Section 4, the result can be
proved either by Fourier analysis or by the Poisson integral formula.
Here we prove Theorem
\ref{thm:another_Extension_method_related_to_the_Vladimirov-Taibleson_operator}
by Fourier analysis.

Let
\begin{align*}
    \alpha>\frac{n}{2},\qquad
    s=2\alpha-n>0,\qquad
    m_p^2=p^{\alpha}+p^{n-\alpha}-p^n-1.
\end{align*}
Following the holographic literature, we write this parameter as $m_p^2$; the notation does not imply that $m_p^2\geq0$. Indeed,
\begin{align*}
    m_p^2=(p^\alpha-1)(1-p^{n-\alpha}),
\end{align*}
which is negative when $\frac n2<\alpha<n$.

We begin with the finite-cutoff construction. In the
horocyclic coordinates introduced above, set
\begin{align*}
    \mathcal T_{p^n}^{R}
    =\{(\boldsymbol x,k)\in\mathcal T_{p^n}:k\leq R\},
    \qquad
    \partial\mathcal T_{p^n}^{R}
    =\{(\boldsymbol x,R):
    \boldsymbol x\in\mathbb Q_p^n/p^R\mathbb Z_p^n\}.
\end{align*}
This is the horocyclic truncation associated with the nested-ball
coordinates.
Given layer-$R$ data $\varphi_R$, we first solve
\begin{align*}
    (\Delta-m_p^2)\varphi(v)=0,
    \qquad v\in\mathcal T_{p^n}^{R-1}.
\end{align*}
For the finite-cutoff boundary calculation below, we then attach $p^n$
artificial children to each height-$R$ vertex and assign them the common value
\begin{align*}
    \varphi(v_{R+1})=p^{\alpha-n}\varphi(v_R).
\end{align*}
This multiplier serves only to define the cutoff residual; the artificial
layer need not coincide with the height-$R+1$ layer of the full-tree
solution. The boundary datum is recovered from the renormalized
trace
\begin{align*}
    f(\boldsymbol x)
    =\lim_{R\to+\infty}p^{(n-\alpha)R}\varphi_R(\boldsymbol x).
\end{align*}
The horocycle $\partial\mathcal T_{p^n}^{R}$ is countably infinite, and the
boundary sums below are interpreted as $L^2$ pairings of the corresponding
layer step functions. This cutoff construction leads to the following full-tree
formulation, in which the massive equation holds at every level.

Consider the following extension problem:
\begin{align}
\label{eq:mass_square_term_extension}
    \begin{cases}
        (\Delta-m_p^2)\varphi(v)=0,&v\in\mathcal{T}_{p^n};\\
        \displaystyle\lim_{k\to+\infty}p^{(n-\alpha)k}\varphi_k=f,
        &\text{in }L^2(\mathbb{Q}_p^n).
    \end{cases}
\end{align}
Here the second condition is a scale-corrected, or renormalized, boundary
limit rather than the ordinary trace of $\varphi$. A solution of
\eqref{eq:mass_square_term_extension} is a sequence of layer functions
$\varphi_k\in L^2(\mathbb{Q}_p^n)$. Each $\varphi_k$ is constant on every
coset of $p^k\mathbb{Z}_p^n$; the massive equation holds as an identity in
$L^2(\mathbb{Q}_p^n)$ at every level; and the displayed renormalized
boundary limit holds in $L^2(\mathbb{Q}_p^n)$.

\begin{theorem}
\label{thm:Fourier_transform_mass_square_term_extension}
For every real-valued $f\in \mathcal{S}(\mathbb{Q}_p^n)$, problem \eqref{eq:mass_square_term_extension} has a unique solution in this class. For $\boldsymbol{\xi}\neq\boldsymbol{0}$, its Fourier transform is
\begin{align*}
    \hat{\varphi}_{k}(\boldsymbol{\xi})=\begin{cases}
        \hat{f}(\boldsymbol{\xi})\left(p^{(\alpha-n)k}-p^{n-2\alpha}\|\boldsymbol{\xi}\|_p^{2\alpha-n}p^{-\alpha k}\right),&k\geq-v_p(\boldsymbol{\xi});\\
        0,&k<-v_p(\boldsymbol{\xi}).
    \end{cases}
\end{align*}
The value at the zero frequency $\boldsymbol{\xi}=\boldsymbol{0}$ may be
taken from the same multiplier formula and does not affect the $L^2$
solution. Moreover, the renormalized boundary limit holds uniformly for
the canonical representatives of the layer functions.
\end{theorem}
\begin{proof}
Taking the Fourier transform of the bulk equation in
\eqref{eq:mass_square_term_extension} yields
\begin{align*}
    &(\hat{\varphi}_{k-1}(\boldsymbol{\xi})-\hat{\varphi}_{k}(\boldsymbol{\xi}))
    +\sum_{\boldsymbol{a}\in\mathbb{F}_p^n}
    \left(\hat{\varphi}_{k+1}(\boldsymbol{\xi})
    \chi_p(-p^k\boldsymbol{a}\cdot\boldsymbol{\xi})
    -\hat{\varphi}_{k}(\boldsymbol{\xi})\right)
    -m_p^2\hat{\varphi}_{k}(\boldsymbol{\xi})=0.
\end{align*}
Fix $\boldsymbol{\xi}\neq\boldsymbol{0}$. By Lemma \ref{lem:hat_varphi=0_k<-vp_xi}, all three layer transforms in this equation vanish when $k\leq-v_p(\boldsymbol{\xi})-2$. When $k=-v_p(\boldsymbol{\xi})-1$, the transforms $\hat\varphi_{k-1}(\boldsymbol{\xi})$ and $\hat\varphi_k(\boldsymbol{\xi})$ vanish, while
\begin{align*}
    \sum_{\boldsymbol{a}\in\mathbb{F}_p^n}
    \chi_p(-p^k\boldsymbol{a}\cdot\boldsymbol{\xi})=0.
\end{align*}
Hence the transformed equation also holds at this transition level. When $v_p(\boldsymbol{\xi})+k\geq0$, the same character sum equals $p^n$, and the equation becomes
\begin{align*}
    p^n\hat{\varphi}_{k+1}(\boldsymbol{\xi})
    -(p^{n-\alpha}+p^{\alpha})\hat{\varphi}_{k}(\boldsymbol{\xi})
    +\hat{\varphi}_{k-1}(\boldsymbol{\xi})=0.
\end{align*}
The associated characteristic equation is
\begin{align*}
    p^n\lambda^2-(p^{n-\alpha}+p^{\alpha})\lambda+1=0
    \quad\Longrightarrow\quad
    \lambda=p^{-\alpha}\ \text{or}\ p^{\alpha-n}.
\end{align*}
The general solution is therefore
\begin{align*}
    \hat{\varphi}_{k}(\boldsymbol{\xi})
    =C_1(\boldsymbol{\xi})p^{-\alpha k}
    +C_2(\boldsymbol{\xi})p^{(\alpha-n)k}.
\end{align*}
By Lemma \ref{lem:hat_varphi=0_k<-vp_xi}, $\hat{\varphi}_k(\boldsymbol{\xi})=0$ for $k<-v_p(\boldsymbol{\xi})$. At the first active level, $k=-v_p(\boldsymbol{\xi})$, we thus have
\begin{align*}
    p^n\hat{\varphi}_{-v_p(\boldsymbol{\xi})+1}(\boldsymbol{\xi})
    -(p^{n-\alpha}+p^{\alpha})
    \hat{\varphi}_{-v_p(\boldsymbol{\xi})}(\boldsymbol{\xi})=0.
\end{align*}
Substituting the general solution yields
\begin{align*}
p^n\left(
C_1(\boldsymbol{\xi})p^{\alpha(v_p(\boldsymbol{\xi})-1)}
+C_2(\boldsymbol{\xi})p^{-(\alpha-n)(v_p(\boldsymbol{\xi})-1)}
\right)=(p^{n-\alpha}+p^{\alpha})\left(
C_1(\boldsymbol{\xi})p^{\alpha v_p(\boldsymbol{\xi})}
+C_2(\boldsymbol{\xi})p^{-(\alpha-n)v_p(\boldsymbol{\xi})}
\right).
\end{align*}
After simplification, we obtain
\begin{align*}
    p^{\alpha}p^{\alpha v_p(\boldsymbol{\xi})}C_1(\boldsymbol{\xi})
    +p^{n-\alpha}p^{(n-\alpha)v_p(\boldsymbol{\xi})}C_2(\boldsymbol{\xi})=0,
\end{align*}
and hence
\begin{align*}
    C_1(\boldsymbol{\xi})
    &=-p^{n-2\alpha}p^{(n-2\alpha)v_p(\boldsymbol{\xi})}
      C_2(\boldsymbol{\xi})\\
    &=-p^{n-2\alpha}\|\boldsymbol{\xi}\|_p^{2\alpha-n}
      C_2(\boldsymbol{\xi}).
\end{align*}
Therefore,
\begin{align*}
    \hat{\varphi}_{k}(\boldsymbol{\xi})
    =C_2(\boldsymbol{\xi})\left(
    p^{(\alpha-n)k}
    -p^{n-2\alpha}\|\boldsymbol{\xi}\|_p^{2\alpha-n}p^{-\alpha k}
    \right),
\end{align*}
for $k\geq-v_p(\boldsymbol{\xi})$. On any fixed nonzero frequency shell, multiplication by $p^{(n-\alpha)k}$ gives
\begin{align*}
    p^{(n-\alpha)k}\hat{\varphi}_{k}(\boldsymbol{\xi})
    =\left(1-p^{-(k+1)(2\alpha-n)}
    \|\boldsymbol{\xi}\|_p^{2\alpha-n}\right)C_2(\boldsymbol{\xi}),
\end{align*}
for all sufficiently large $k$. Since the factor in parentheses tends to $1$,
restricting the prescribed $L^2$ boundary limit to this shell and letting
$k\to+\infty$ gives
\begin{align*}
    C_2(\boldsymbol{\xi})=\hat f(\boldsymbol{\xi}),
\end{align*}
for almost every $\boldsymbol{\xi}$ on the shell. This proves the formula on
every nonzero frequency shell.

Conversely, the formula yields
\begin{align*}
    p^{(n-\alpha)k}\hat{\varphi}_{k}(\boldsymbol{\xi})
    =\mathbf{1}_{\{\|\boldsymbol{\xi}\|_p\leq p^k\}}
    \left(1-p^{-(k+1)(2\alpha-n)}
    \|\boldsymbol{\xi}\|_p^{2\alpha-n}\right)
    \hat f(\boldsymbol{\xi}).
\end{align*}
The multiplier on the right lies between $0$ and $1$ and converges to $1$ almost everywhere. Plancherel's theorem and the dominated convergence theorem show that this formula defines an $L^2$ layer solution satisfying
\begin{align*}
    p^{(n-\alpha)k}\varphi_k\longrightarrow f
    \quad\text{in }L^2(\mathbb{Q}_p^n).
\end{align*}
Because $f\in \mathcal{S}(\mathbb{Q}_p^n)$, we also have
$\hat f\in L^1(\mathbb{Q}_p^n)$, and the multiplier expression converges
to $\hat f$ in $L^1$. Fourier inversion then gives uniform convergence
of the canonical representatives. The boundary limit determines $C_2$
on every nonzero frequency shell, while the vanishing of the preceding
layer determines $C_1$. The solution is therefore unique in the stated
class; the value at the singleton $\{\boldsymbol{0}\}$ is immaterial to
$L^2$ uniqueness.
\end{proof}

\begin{definition}
The scale-corrected normal derivative is defined by
\begin{align*}
    \partial_{n,\alpha}\varphi(\boldsymbol{x})
    =\lim_{k\to+\infty}
    \left(f(\boldsymbol{x})-p^{(n-\alpha)k}
    \varphi_k(\boldsymbol{x})\right)p^{(k+1)(2\alpha-n)},
\end{align*}
where the limit is taken in $L^2(\mathbb{Q}_p^n)$.
\end{definition}

Since $\hat f$ has compact support, the formula in Theorem \ref{thm:Fourier_transform_mass_square_term_extension} is valid on the support of $\hat f$ for all sufficiently large $k$. Hence
\begin{align*}
    \widehat{\partial_{n,\alpha}\varphi}(\boldsymbol{\xi})
    =&\lim_{k\to+\infty}
    \left(\hat f(\boldsymbol{\xi})
    -p^{(n-\alpha)k}\hat\varphi_k(\boldsymbol{\xi})\right)
    p^{(k+1)(2\alpha-n)}\\
    =&\lim_{k\to+\infty}
    p^{n-2\alpha}\|\boldsymbol{\xi}\|_p^{2\alpha-n}
    p^{(n-2\alpha)k}\hat f(\boldsymbol{\xi})
    p^{(k+1)(2\alpha-n)}\\
    =&\|\boldsymbol{\xi}\|_p^{2\alpha-n}
    \hat f(\boldsymbol{\xi}).
\end{align*}
Consequently,
\begin{align*}
    \partial_{n,\alpha}\varphi(\boldsymbol{x})
    =D^{2\alpha-n}f(\boldsymbol{x}).
\end{align*}

\begin{remark}
The extension methods of Sections 4 and 5 are equivalent after rescaling
each layer by $p^{(n-\alpha)k}$. Indeed, the Fourier formula for
$p^{(n-\alpha)k}\varphi_k$ coincides with the formula obtained in Section 4
for the weighted extension with $s=2\alpha-n$. Thus, modifying the edge
weights is equivalent to keeping the tree unweighted, introducing the
parameter $m_p^2$, and replacing the ordinary boundary trace by the
scale-corrected boundary limit.
\end{remark}

We conclude this section with the finite-cutoff boundary calculation. For
this purpose, we retain the exact solution through level $R$ and introduce
a single artificial ghost layer. Here
\begin{align*}
    \partial\mathcal{T}_{p^n}^R
    =\{v_R=(\boldsymbol{x},R):
    \boldsymbol{x}\in\mathbb{Q}_p^n/p^R\mathbb{Z}_p^n\}
\end{align*}
denotes the level-$R$ horocycle. It is countable and is not the boundary of a finite subtree. Accordingly, the vertex sums below are understood as $L^2$ pairings of level-$R$ step functions; their absolute convergence follows from the Cauchy--Schwarz inequality. For every $v_R\in\partial\mathcal{T}_{p^n}^R$, assign to each of its $p^n$ ghost children the value
\begin{align*}
    \varphi(v_{R+1})=p^{\alpha-n}\varphi(v_R).
\end{align*}
Denote the residual determined by this ghost layer at $v_R$ by
\begin{align*}
    \mathcal{R}_R\varphi(v_R)
    :=(\Delta-m_p^2)\varphi(v_R),
\end{align*}
where the right-hand side is evaluated using the artificial ghost children. We use the same notation for the corresponding level-$R$ step function on $\mathbb{Q}_p^n$. A direct calculation gives
\begin{align*}
    \mathcal{R}_R\varphi(v_R)
    =&p^{n}\varphi(v_{R+1})+\varphi(v_{R-1})
    -(p^n+1+m_p^2)\varphi(v_R)\\
    =&p^{n}p^{\alpha-n}\varphi(v_R)+\varphi(v_{R-1})
    -(p^{n-\alpha}+p^{\alpha})\varphi(v_R)\\
    =&\varphi(v_{R-1})-p^{n-\alpha}\varphi(v_R)\\
    =&\left(p^{(n-\alpha)(R-1)}\varphi(v_{R-1})
    -p^{(n-\alpha)R}\varphi(v_R)\right)
    p^{(\alpha-n)(R-1)}\\
    =&\left(\left(f(\boldsymbol{x})
    -p^{(n-\alpha)R}\varphi(v_R)\right)
    -\left(f(\boldsymbol{x})
    -p^{(n-\alpha)(R-1)}\varphi(v_{R-1})\right)\right)
    p^{(\alpha-n)(R-1)}\\
    =&p^{-\alpha R}\Bigg(
    p^{2n-3\alpha}
    \left(f(\boldsymbol{x})-p^{(n-\alpha)R}\varphi(v_R)\right)
    p^{(R+1)(2\alpha-n)}\\
    &\hspace{36mm}
    -p^{n-\alpha}
    \left(f(\boldsymbol{x})
    -p^{(n-\alpha)(R-1)}\varphi(v_{R-1})\right)
    p^{R(2\alpha-n)}\Bigg).
\end{align*}
This identity holds for every finite $R$.

Every level-$R$ function is constant on balls of Haar measure $p^{-nR}$. The preceding identity therefore gives
\begin{align*}
    &\sum_{v_R\in\partial\mathcal{T}_{p^n}^{R}}
    \varphi(v_R)\mathcal{R}_R\varphi(v_R)\\
    =&p^{nR}\int_{\mathbb{Q}_p^n}
    \varphi_R(\boldsymbol{x})\mathcal{R}_R\varphi_R(\boldsymbol{x})
    d\boldsymbol{x}\\
    =&\int_{\mathbb{Q}_p^n}p^{(n-\alpha)R}
    \varphi_R(\boldsymbol{x})\Bigg[
    p^{2n-3\alpha}
    \left(f(\boldsymbol{x})-p^{(n-\alpha)R}
    \varphi_R(\boldsymbol{x})\right)p^{(R+1)(2\alpha-n)}\\
    &\hspace{39mm}
    -p^{n-\alpha}
    \left(f(\boldsymbol{x})-p^{(n-\alpha)(R-1)}
    \varphi_{R-1}(\boldsymbol{x})\right)p^{R(2\alpha-n)}
    \Bigg]d\boldsymbol{x}.
\end{align*}
Theorem \ref{thm:Fourier_transform_mass_square_term_extension} implies
\begin{align*}
    &p^{(n-\alpha)R}\widehat{\varphi_R}(\boldsymbol{\xi})
    =\mathbf{1}_{\{\|\boldsymbol{\xi}\|_p\leq p^R\}}
    \left(1-p^{-(R+1)(2\alpha-n)}
    \|\boldsymbol{\xi}\|_p^{2\alpha-n}\right)
    \hat f(\boldsymbol{\xi}),\\
    &p^{(R+1)(2\alpha-n)}\left(f-p^{(n-\alpha)R}\varphi_R\right)^{\widehat{}}(\boldsymbol{\xi})
    =\min\left\{\|\boldsymbol{\xi}\|_p^{2\alpha-n},
    p^{(R+1)(2\alpha-n)}\right\}\hat f(\boldsymbol{\xi}).
\end{align*}
The analogous formula at $R-1$ has multiplier
\begin{align*}
    \min\left\{\|\boldsymbol{\xi}\|_p^{2\alpha-n},
    p^{R(2\alpha-n)}\right\}.
\end{align*}
The vertex series is absolutely convergent by the Cauchy--Schwarz
inequality. The Fourier transform of
$p^{(n-\alpha)R}\varphi_R$ is supported in
$\{\|\boldsymbol{\xi}\|_p\leq p^R\}$. Since the $p$-adic norm takes
values in $p^{\mathbb{Z}}\cup\{0\}$, both normal-derivative multipliers
equal $\|\boldsymbol{\xi}\|_p^{2\alpha-n}$ on this support.
Plancherel's theorem therefore yields the exact identity
\begin{align*}
    \sum_{v_R\in\partial\mathcal{T}_{p^n}^{R}}
    \varphi(v_R)\mathcal{R}_R\varphi(v_R)=p^{n-2\alpha}(p^{n-\alpha}-p^{\alpha})
    \int_{\|\boldsymbol{\xi}\|_p\leq p^R}
    \left(1-p^{-(R+1)(2\alpha-n)}
    \|\boldsymbol{\xi}\|_p^{2\alpha-n}\right)
    \|\boldsymbol{\xi}\|_p^{2\alpha-n}
    |\hat f(\boldsymbol{\xi})|^2d\boldsymbol{\xi}.
\end{align*}
Since $f\in \mathcal{S}(\mathbb{Q}_p^n)$, the dominated convergence theorem gives
\begin{align*}
    \lim_{R\to+\infty}
    \sum_{v_R\in\partial\mathcal{T}_{p^n}^{R}}
    \varphi(v_R)\mathcal{R}_R\varphi(v_R)=&p^{n-2\alpha}(p^{n-\alpha}-p^{\alpha})
    \int_{\mathbb{Q}_p^n}
    \|\boldsymbol{\xi}\|_p^{2\alpha-n}
    |\hat f(\boldsymbol{\xi})|^2d\boldsymbol{\xi}\\
    =&p^{n-2\alpha}(p^{n-\alpha}-p^{\alpha})
    \int_{\mathbb{Q}_p^n}f(\boldsymbol{x})
    D^{2\alpha-n}f(\boldsymbol{x})d\boldsymbol{x}.
\end{align*}
This expression is the renormalized boundary functional associated with
the ghost-layer residual. Its identification with a full on-shell bulk
action would additionally require a one-sided Green identity for the
massive equation, an explicit counterterm, and convergence of the bulk and boundary series.

Table \ref{tab:Caffarelli–Silvestre extension and non-Archimedean extension} summarizes the structural parallels between the classical Caffarelli--Silvestre extension and the non-Archimedean tree extensions considered here.

\begin{table}[h!]
    \centering
    \begin{tabular}{|c|c|c|}
        \hline
        &Caffarelli--Silvestre extension& Non-Archimedean tree extensions \\ \hline
        Boundary space & $\mathbb{R}^n$ & $\mathbb{Q}_p^n$\\ \hline
        Bulk space & $\mathbb{R}_{+}^{n+1}$ & $\mathcal{T}_{p^n}$\\ \hline
        Bulk coordinates & $(\boldsymbol{x},y),\ y>0$ & $(\boldsymbol{x},k),\ k\in\mathbb{Z}$\\ \hline
        Extension equation & $\nabla\cdot(y^{1-2s}\nabla\varphi)=0$, $0<s<1$ & $\Delta_w\varphi=0$ or $(\Delta-m_p^2)\varphi=0$\\ \hline
        Boundary limit & $\displaystyle\lim_{y\to0^+}\varphi$ & $\displaystyle\lim_{k\to+\infty}\varphi_k$ or $\displaystyle\lim_{k\to+\infty}p^{k(n-\alpha)}\varphi_k$\\ \hline
        Normal derivative & $-C_{n,s}\displaystyle\lim_{y\to0^+}y^{1-2s}\varphi_y$ & $\partial_{n,s}^{(p)}\varphi$ or $\partial_{n,\alpha}\varphi$\\ \hline
        Nonlocal operator & $(-\Delta)^s$ & $D^s$ or $D^{2\alpha-n}$\\ \hline
    \end{tabular}
    \caption{Comparison of the Caffarelli--Silvestre extension with the
    non-Archimedean tree extensions}
    \label{tab:Caffarelli–Silvestre extension and non-Archimedean extension}
\end{table}
\section{Extension problems for the hierarchical Laplacian}
In this section, we prove Theorem
\ref{thm:Extension_problems_related_to_the_hierarchical_Laplacian}. We retain
the notation introduced in the preliminaries. In particular, $B'$ denotes
the unique immediate super-ball of $B$, and
$\text{sub}(B)$ denotes the finite collection of maximal proper
sub-balls of $B$.

Consider
\begin{align}
\label{eq:Extension_problem_TX}
    \begin{cases}
        \Delta_W\varphi(v_B)=0,&v_B\in\mathcal{T}_X;\\
        \varphi|_{\partial\mathcal{T}_X}(x)=f(x),&x\in X;\\
        \varphi(\infty)=0.&
    \end{cases}
\end{align}
Here $W(A,B)=W(B,A)>0$ for adjacent balls $A$ and $B$, and
\begin{align}
\label{eq:weighted_Laplacian_TX}
    \Delta_W\varphi(v_B)
    =W(B,B')\bigl(\varphi(v_{B'})-\varphi(v_B)\bigr)+\sum_{A\in\text{sub}(B)}
    W(A,B)\bigl(\varphi(v_A)-\varphi(v_B)\bigr).
\end{align}
Let
\begin{align}
\label{eq:flux_at_B}
    \Phi(B)=W(B,B')\bigl(\varphi(v_B)-\varphi(v_{B'})\bigr).
\end{align}
Then \eqref{eq:weighted_Laplacian_TX} can be rewritten as
\begin{align*}
    \Phi(B)=\sum_{A\in\text{sub}(B)}\Phi(A).
\end{align*}

This identity yields finite additivity on the algebra of ball cylinders.
Finite additivity alone, however, does not imply countable additivity,
finite variation, or absolute continuity with respect to $m$. We therefore
work with the following natural class. An admissible solution of
\eqref{eq:Extension_problem_TX} is a harmonic function $\varphi$ such that
$\varphi(\infty)=0$, its boundary trace exists in $L^2(X,m)$, and there is a
function $g_\varphi\in \mathcal{S}_0(X)$ satisfying
\begin{align}
\label{eq:the_measure_recovery_formula}
    \Phi(B)=\int_B g_\varphi(y)\,dm(y),
    \qquad B\in\mathcal B.
\end{align}
The density $g_\varphi$ is unique. Indeed, if two locally constant
$L^2$-functions have the same integral over every ball, then every ball
average of their difference vanishes. Taking balls sufficiently fine that
the difference is constant on each of them shows that the difference is
zero.

Conversely, let $g\in \mathcal{S}_0(X)$ and suppose that the positive edge weight $W$
satisfies, along every boundary chain,
\begin{align*}
    \sum_{j=k}^{+\infty}
    \frac{m(B_j(x))}{W(B_j(x),B_{j-1}(x))}<+\infty,
\end{align*}
with the sum locally bounded in $x$ for each fixed $k$. Define
\begin{align*}
    \varphi_g(v_B)=\sum_{D\supseteq B}
    \frac{1}{W(D,D')}\int_Dg(y)\,dm(y).
\end{align*}
For each fixed $B$, only finitely many terms in the ancestor sum are
nonzero, because the integral of $g$ vanishes on every sufficiently large
super-ball. The summability condition above controls the additional descendant terms as
$B$ tends to a boundary point. Moreover,
\begin{align*}
    W(B,B')\bigl(\varphi_g(v_B)-\varphi_g(v_{B'})\bigr)
    =\int_Bg(y)\,dm(y).
\end{align*}
Since the sub-balls of $B$ form a finite measurable partition of $B$, these
fluxes satisfy the conservation law; hence
$\Delta_W\varphi_g(v_B)=0$. The same formula also shows that
$\varphi_g(\infty)=0$.

For an admissible solution, \eqref{eq:the_measure_recovery_formula} gives
\begin{align}
\label{eq:gx}
    \frac{\Phi(B_k(x))}{m(B_k(x))}
    =&\frac{W(B_k,B_{k-1})}{m(B_k)}
    \bigl(\varphi(v_{B_k})-\varphi(v_{B_{k-1}})\bigr)\notag\\
    =&\frac{1}{m(B_k(x))}\int_{B_k(x)}g_\varphi(y)\,dm(y)
    \longrightarrow g_\varphi(x).
\end{align}
For $g_\varphi\in \mathcal{S}_0(X)$, the final convergence is especially simple.
Since $g_\varphi$ is locally constant and compactly supported, there is
$k_0$ such that it is constant on every ball of height at least $k_0$.
Thus every ball average in \eqref{eq:gx} equals $g_\varphi(x)$ for
sufficiently large $k$.

\subsection{Green's identity and the deformed normal derivative at the boundary}
\begin{definition}[Deformed normal derivative on $\partial\mathcal T_X\setminus\{\infty\}$]
\label{def:The_deformed_normal_derivative_TX}
    Suppose that
    $\cdots\supseteq B_{k-1}(x)\supseteq B_k(x)
    \supseteq B_{k+1}(x)\supseteq\cdots\ni x$ and that
    $\varphi(v_{B_k(\,\cdot\,)})\to f$ in $L^2(X,m)$. The deformed normal
    derivative at $x\in\partial\mathcal T_X\setminus\{\infty\}\cong X$ is
    defined, whenever the limit exists in $L^2(X,m)$, by
    \begin{align*}
        \partial_n^{(W)}\varphi(x)
        =\lim_{k\to+\infty}
        \frac{W(B_k,B_{k-1})}{m(B_k)}
        \bigl(f(x)-\varphi(v_{B_{k-1}})\bigr).
    \end{align*}
\end{definition}

\begin{lemma}
\label{lem:relation_partialn_and_g}
    Suppose that
    \begin{align*}
        \frac{W(B_k,B_{k-1})}{m(B_k)}
        \sum_{j=k}^{+\infty}
        \frac{m(B_j)}{W(B_j,B_{j-1})}
        \longrightarrow\kappa_W\in(0,+\infty)
    \end{align*}
    locally uniformly in $x$, where the limit $\kappa_W$ is independent of
    $x$. Then every admissible solution has an $L^2$ normal derivative, and
    \begin{align*}
         \partial_n^{(W)}\varphi=\kappa_Wg_\varphi.
    \end{align*}
\end{lemma}
\begin{proof}
    By telescoping the edge increments and using
    \eqref{eq:the_measure_recovery_formula}, we obtain
    \begin{align*}
        \frac{W(B_k,B_{k-1})}{m(B_k)}
        \bigl(f(x)-\varphi(v_{B_{k-1}})\bigr)
        =&\frac{W(B_k,B_{k-1})}{m(B_k)}
        \sum_{j=k}^{+\infty}
        \bigl(\varphi(v_{B_j})-\varphi(v_{B_{j-1}})\bigr)\\
        =&\frac{W(B_k,B_{k-1})}{m(B_k)}
        \sum_{j=k}^{+\infty}
        \frac{\Phi(B_j)}{W(B_j,B_{j-1})}\\
        =&\frac{W(B_k,B_{k-1})}{m(B_k)}
        \sum_{j=k}^{+\infty}
        \frac{m(B_j)}{W(B_j,B_{j-1})}
        \left(\frac{1}{m(B_j)}\int_{B_j}g_\varphi(y)\,dm(y)\right).
    \end{align*}
    Choose a ball $P$ containing $\operatorname{supp}g_\varphi$. Since
    $g_\varphi$ is locally constant and $P$ is compact, a finite refinement
    by balls yields a height $k_0$ such that $g_\varphi$ is constant on every
    ball of height $j\geq k_0$ that meets $P$.
    Thus, for $k\geq k_0$, the last expression is exactly
    \begin{align*}
        \frac{W(B_k,B_{k-1})}{m(B_k)}
        \sum_{j=k}^{+\infty}
        \frac{m(B_j)}{W(B_j,B_{j-1})}g_\varphi(x).
    \end{align*}
    For the same values of $k$, this expression vanishes outside $P$, since
    every ball in the fine-scale tail is disjoint from
    $\operatorname{supp}g_\varphi$. The locally uniform convergence of the
    coefficient on $P$ therefore gives convergence to
    $\kappa_Wg_\varphi$ in $L^2(X,m)$.
\end{proof}

We next derive Green's identity in the hierarchical coordinate system used
to define the ultrametric tree. For $R\in\mathbb Z$, consider the truncated
tree
\begin{align*}
    \mathcal{T}_{X}^R
    =&\{v_B\in\mathcal{T}_{X}:h(B)\leq R\},\\
    \partial\mathcal{T}_{X}^R
    =&\{v_B\in\mathcal{T}_{X}:h(B)=R\}.
\end{align*}
Here $h(B)$ is the Busemann height introduced above, so $\mathcal T_X^R$ is
the horocyclic truncation determined by the inclusion levels of ultrametric
balls. For each $x\in X$, let $B_R(x)$ denote
the unique ball of height $R$ containing $x$. The balls of height $R$ form a
disjoint partition of $X$. Although $\mathcal T_X^R$ need not be finite,
Green's identity remains valid under the following absolute convergence
assumption.

\begin{lemma}[Green's first identity for the truncated tree $\mathcal{T}_{X}^R$]
\label{lem:Green's first identity for the truncated tree TXR}
Let $g,h$ be functions on the vertices of $\mathcal T_X$, and assume that
\begin{align*}
    \sum_{v_A\in\mathcal T_X^R}|g(v_A)|
    \sum_{v_B\sim v_A}W(A,B)|h(v_A)-h(v_B)|<+\infty.
\end{align*}
Then
    \begin{align*}
        &\frac{1}{2}
        \sum_{\substack{v_A,v_B\in \mathcal{T}_{X}^R\\v_A\sim v_B}}
        W(A,B)\bigl(g(v_A)-g(v_B)\bigr)
        \bigl(h(v_A)-h(v_B)\bigr)\\
        =&-\sum_{v_A\in \mathcal{T}_{X}^R}
        g(v_A)\Delta_Wh(v_A)-\sum_{v_A\in\partial \mathcal{T}_{X}^R}
        g(v_A)\sum_{\substack{v_B\notin\mathcal T_X^R\\v_B\sim v_A}}
        \bigl(h(v_A)-h(v_B)\bigr)W(A,B).
    \end{align*}
\end{lemma}
\begin{proof}
    Expanding the full weighted Laplacian gives
    \begin{align*}
        -\sum_{v_A\in \mathcal{T}_{X}^R}g(v_A)\Delta_Wh(v_A)
        =\sum_{v_A\in \mathcal{T}_{X}^R}g(v_A)\sum_{v_B\sim v_A}W(A,B)\bigl(h(v_A)-h(v_B)\bigr).
    \end{align*}
    Absolute convergence allows us to group the terms edge by edge. If
    $v_A,v_B\in\mathcal T_X^R$, the two oriented contributions
    corresponding to the interior edge $\{v_A,v_B\}$ add to
    \begin{align*}
        W(A,B)\bigl(g(v_A)-g(v_B)\bigr)
        \bigl(h(v_A)-h(v_B)\bigr).
    \end{align*}
    Since the displayed double sum over adjacent pairs counts each interior
    edge twice, the contribution from the interior edges is the first term
    in the statement.
    If $v_A\in\mathcal T_X^R$ and $v_B\notin\mathcal T_X^R$, then the
    Busemann heights of $A$ and $B$ are $R$ and $R+1$, respectively, and
    $B\in\text{sub}(A)$. Such an edge occurs only once and contributes
    \begin{align*}
        g(v_A)W(A,B)\bigl(h(v_A)-h(v_B)\bigr).
    \end{align*}
    These are precisely the boundary terms in the statement. Moving them
    to the other side proves the identity.
\end{proof}

\begin{theorem}[Green's first identity for $\mathcal{T}_{X}$]
\label{thm:Green's_first_identity_TX}
    Let $h$ be an admissible solution with flux density $g_h$, and assume
    the hypotheses of Lemma \ref{lem:relation_partialn_and_g}. Let $g$ be a
    vertex function and define its height-$R$ boundary step function by
    \begin{align*}
        g_R(x)=g(v_{B_R(x)}),\qquad x\in X.
    \end{align*}
    Suppose that, for every compact open ball $P$, $g_R$ converges in
    $L^2(P,m)$ to a boundary function, again denoted by $g(x)$, and that the
    weighted edge pairing below is absolutely convergent. Then
    \begin{align*}
        &\frac{1}{2}
        \sum_{\substack{v_A,v_B\in \mathcal{T}_{X}\\v_A\sim v_B}}
        W(A,B)\bigl(g(v_A)-g(v_B)\bigr)
        \bigl(h(v_A)-h(v_B)\bigr)\\
        =&-\sum_{v_A\in \mathcal{T}_{X}}
        g(v_A)\Delta_Wh(v_A)
        +\frac{1}{\kappa_W}\int_Xg(x)\partial_n^{(W)}h(x)\,dm(x).
    \end{align*}
\end{theorem}
\begin{proof}
    Choose a ball $P$ containing $\operatorname{supp}g_h$. If a ball $B$
    is not contained in $P$, then either $B\cap P=\varnothing$ or
    $B\supseteq P$. In the first case $\int_Bg_h\,dm=0$, while in the
    second case
    \begin{align*}
        \int_Bg_h\,dm=\int_Pg_h\,dm=\int_Xg_h\,dm=0.
    \end{align*}
    Write
    \begin{align*}
        \Phi_h(B):=W(B,B')\bigl(h(v_B)-h(v_{B'})\bigr).
    \end{align*}
    It follows that $\Phi_h(B)=0$ whenever
    $B\not\subseteq P$. For a fixed $R$, the balls $B\subseteq P$ with
    $h(B)\leq R+1$ occupy only finitely many generations below $P$. By
    finite branching, only finitely many edges incident to
    $\mathcal T_X^R$ have nonzero increments, including the child edges
    outside the truncation. Thus Lemma
    \ref{lem:Green's first identity for the truncated tree TXR} applies to
    the global height truncation.

    Harmonicity and \eqref{eq:the_measure_recovery_formula} give, for every
    $v_A\in\partial\mathcal T_X^R$,
    \begin{align*}
        \sum_{D\in\text{sub}(A)}\Phi_h(D)
        =\Phi_h(A)=\int_Ag_h(x)\,dm(x).
    \end{align*}
    The boundary term in the truncated Green identity is therefore
    \begin{align*}
        -\sum_{v_A\in\partial\mathcal T_X^R}g(v_A)
        \sum_{\substack{v_B\notin\mathcal T_X^R\\v_B\sim v_A}}
        \bigl(h(v_A)-h(v_B)\bigr)W(A,B)=&\sum_{v_A\in\partial\mathcal T_X^R}g(v_A)\sum_{D\in\text{sub}(A)}\Phi_h(D)\\
        =&\sum_{v_A\in\partial\mathcal T_X^R}
        g(v_A)\int_Ag_h(x)\,dm(x)\\
        =&\int_Xg_R(x)g_h(x)\,dm(x).
    \end{align*}
    Since $g_h$ is supported in $P$ and $g_R\to g$ in $L^2(P,m)$, the
    Cauchy--Schwarz inequality yields
    \begin{align*}
        \int_Xg_R(x)g_h(x)\,dm(x)
        \longrightarrow\int_Xg(x)g_h(x)\,dm(x)
        =\frac{1}{\kappa_W}\int_Xg(x)\partial_n^{(W)}h(x)\,dm(x),
    \end{align*}
    where the last equality follows from Lemma
    \ref{lem:relation_partialn_and_g}. Absolute convergence permits passage
    to the limit in the interior-edge term as $R\to+\infty$. The bulk sum
    vanishes term by term because $h$ is harmonic; it is retained in the
    statement only to display the usual form of Green's first identity.
    Letting $R\to+\infty$ proves the result.
\end{proof}

\subsection{\texorpdfstring{Isotropic operators and bijectivity on $\mathcal{S}_0(X)$}{Isotropic operators and bijectivity on S0(X)}}
\begin{definition}[Isotropic operator]
    Let $T$ be a linear operator defined on $\mathcal{S}(X)$. If $T$ commutes with
    every isometry $\tau$ of the ultrametric space $X$, that is,
    \begin{align*}
        T(f\circ\tau)=(Tf)\circ\tau,
    \end{align*}
    for all $f\in \mathcal{S}(X)$ and all isometries $\tau$, then $T$ is called a
    globally isotropic operator on $\mathcal{S}(X)$.
\end{definition}

\begin{theorem}
\label{thm:isotropic_operators}
    Let $\tau$ be a measure-preserving isometry such that
    \begin{align*}
        C(\tau(B))=C(B),
        \qquad
        W(\tau(A),\tau(B))=W(A,B)
    \end{align*}
    for every ball $B$ and every adjacent pair $A,B$. Then $L_C$ is
    equivariant under $\tau$, and the normal derivative of every admissible
    extension transforms equivariantly. If the admissible Dirichlet
    extension is unique, the corresponding Dirichlet-to-Neumann operator is
    equivariant under $\tau$. In particular, these operators are globally
    isotropic if the stated conditions hold for every isometry.
\end{theorem}
\begin{proof}
\begin{enumerate}
    \item For $L_C$, recall that
    \begin{align*}
        L_Cf(x)=\int_XK(x,y)(f(x)-f(y))\,dm(y).
    \end{align*}
    Since the assumptions imply $K(\tau(x),\tau(y))=K(x,y)$, we obtain
    \begin{align*}
        L_C(f\circ\tau)(x)
        =&\int_XK(x,y)(f(\tau(x))-f(\tau(y)))\,dm(y)\\
        =&\int_XK(x,\tau^{-1}(z))(f(\tau(x))-f(z))
        \,dm(\tau^{-1}(z))\\
        =&\int_XK(\tau(x),z)(f(\tau(x))-f(z))\,dm(z)\\
        =&(L_Cf)\circ\tau(x).
    \end{align*}

    \item For $\partial_n^{(W)}$, lift the boundary isometry to the tree
    $\mathcal T_X$. Since $\tau$ is an isometry, it maps each ultrametric
    ball $B$ to $\tau(B)$ and thereby induces a graph automorphism of
    $\mathcal T_X$. The assumed invariance of the edge weight gives
    \begin{align*}
        W(\tau(A),\tau(B))=W(A,B).
    \end{align*}
    Given $f\in \mathcal{S}(X)$ and an admissible extension $\varphi$ solving
    \eqref{eq:Extension_problem_TX}, define
    \begin{align*}
        \psi(v_B):=\varphi(v_{\tau(B)}).
    \end{align*}
    Then
    \begin{align*}
        \lim_{v_B\to x}\psi(v_B)
        =\lim_{v_B\to x}\varphi(v_{\tau(B)})
        =\lim_{v_{\tau(B)}\to\tau(x)}\varphi(v_{\tau(B)})=f(\tau(x))=(f\circ\tau)(x),
    \end{align*}
    and
    \begin{align*}
        \sum_{v_A\sim v_B}W(A,B)
        \bigl(\psi(v_A)-\psi(v_B)\bigr)=\sum_{v_A\sim v_B}W(\tau(A),\tau(B))
        \bigl(\varphi(v_{\tau(A)})-\varphi(v_{\tau(B)})\bigr)=0,
    \end{align*}
    since $\tau(A)$ runs over all neighbors of $\tau(B)$ as $A$ runs over
    all neighbors of $B$. Moreover,
    \begin{align*}
        \Phi_\psi(B)=\Phi_\varphi(\tau(B))=\int_{\tau(B)}g_\varphi(z)\,dm(z)=\int_Bg_\varphi(\tau(y))\,dm(y).
    \end{align*}
    Thus $\psi$ is admissible, $g_\psi=g_\varphi\circ\tau$, and
    $\psi(\infty)=0$. If the admissible solution is unique, $\psi$ is
    precisely the extension of $f\circ\tau$.

    Finally, for the chain $B_k=B_k(x)$, we have
    \begin{align*}
        \partial_n^{(W)}\psi(x)
        =&\lim_{k\to+\infty}
        \frac{W(B_k,B_{k-1})}{m(B_k)}
        \Bigl((f\circ\tau)(x)-\psi(v_{B_{k-1}})\Bigr)\\
        =&\lim_{k\to+\infty}
        \frac{W(B_k,B_{k-1})}{m(B_k)}
        \Bigl(f(\tau(x))-\varphi(v_{\tau(B_{k-1})})\Bigr)\\
        =&\lim_{k\to+\infty}
        \frac{W(\tau(B_k),\tau(B_{k-1}))}{m(\tau(B_k))}
        \Bigl(f(\tau(x))-\varphi(v_{\tau(B_{k-1})})\Bigr)\\
        =&\bigl(\partial_n^{(W)}\varphi\bigr)(\tau(x))\\
        =&\bigl(\partial_n^{(W)}\varphi\bigr)\circ\tau(x).
    \end{align*}
\end{enumerate}
\end{proof}

For the next two results, we retain the hypotheses of Lemma
\ref{lem:relation_partialn_and_g} and impose two additional assumptions.
First, the admissible Dirichlet problem is uniquely solvable for data in
$\mathcal{S}_0(X)$, so that the extension depends linearly on the boundary datum.
Second, for every immediate inclusion $B\subset B'$, the maximal sub-balls
of $B'$ are permuted transitively by tree
automorphisms induced by measure-preserving isometries of $X$. These
automorphisms preserve both $C$ and $W$, fix $\infty$, fix $v_{B'}$ and
every ancestor $v_D$ with $D\supseteq B'$ pointwise, and fix the part of the
tree outside the descendant branches of $B'$. We further assume that the
stabilizer of $B$ and $B'$ acts transitively on both $B$ and
$B'\setminus B$. No explicit formula for $W$ is assumed here. These
additional assumptions are used only in Lemma
\ref{lem:eigvalue_deformed_derivative_TX} and Theorem
\ref{thm:bijective_LC_partialn}.

\begin{lemma}
\label{lem:eigvalue_deformed_derivative_TX}
    Let $\varphi_{f_B}$ be the admissible harmonic extension of $f_B$
    solving \eqref{eq:Extension_problem_TX}. Then there exists $\mu_B$ such
    that
    \begin{align*}
        \partial_n^{(W)}\varphi_{f_B}=\mu_Bf_B.
    \end{align*}
\end{lemma}
\begin{proof}
    Let $\text{Aut}(\mathcal T_X)$ denote the group of tree
    automorphisms induced by measure-preserving isometries of $X$ that fix
    $\infty$ and preserve $C$ and $W$, and consider the stabilizer subgroup
    $G_{B,B'}<\text{Aut}(\mathcal T_X)$ that fixes the vertices
    corresponding to $B$ and $B'$. By the local homogeneity assumed above,
    it acts transitively on $B$ and on $B'\setminus B$. The function $f_B$
    is invariant under $G_{B,B'}$. Theorem
    \ref{thm:isotropic_operators} and uniqueness of the admissible extension
    therefore imply that both $\varphi_{f_B}$ and its deformed normal
    derivative are invariant under $G_{B,B'}$. Consequently, the normal
    derivative is constant on $B$ and on $B'\setminus B$; denote these
    constants by $k_1$ and $k_2$, respectively.

    We next prove that the extension vanishes on the exterior subtree. For
    each $A\in\text{sub}(B')$, choose one of the local automorphisms above
    that maps $A$ onto $B$, and denote its boundary action by $\tau$. Since it
    preserves $m$ and fixes $B'$, we have $f_A=f_B\circ\tau$. It also
    fixes every $D\supseteq B'$. Therefore, by Theorem
    \ref{thm:isotropic_operators},
    \begin{align*}
        \varphi_{f_A}(v_D)
        =\varphi_{f_B}(v_{\tau(D)})
        =\varphi_{f_B}(v_D),
        \qquad D\supseteq B'.
    \end{align*}
    On the other hand, the definition of the functions $f_A$ gives
    \begin{align*}
        \sum_{A\in\text{sub}(B')}m(A)f_A=0.
    \end{align*}
    Linearity and uniqueness of the extension thus imply, for every
    $D\supseteq B'$,
    \begin{align*}
        0=\varphi_{\sum\limits_{A\in\text{sub}(B')}m(A)f_A}(v_D)
        =\sum_{A\in\text{sub}(B')}m(A)\varphi_{f_A}(v_D)
        =m(B')\varphi_{f_B}(v_D).
    \end{align*}
    Hence $\varphi_{f_B}(v_D)=0$ along the ancestor chain of $B'$. To obtain
    the corresponding conclusion on the exterior subtree within the
    admissible class, define
    \begin{align*}
        \widetilde\varphi(v_A)=
        \begin{cases}
            \varphi_{f_B}(v_A),&A\subseteq B',\\
            0,&A\not\subseteq B'.
        \end{cases}
    \end{align*}
    Since $\varphi_{f_B}$ vanishes at $v_{B'}$ and at the vertex
    corresponding to the immediate super-ball of $B'$, the flux through the
    upper edge of $B'$ is zero. Thus
    \begin{align*}
        \int_{B'}g_{\varphi_{f_B}}\,dm=\Phi(B')=0,
    \end{align*}
    so $\boldsymbol1_{B'}g_{\varphi_{f_B}}\in \mathcal{S}_0(X)$. The measure recovery
    formula, applied separately to balls contained in $B'$, balls disjoint
    from $B'$, and balls containing $B'$, shows that $\widetilde\varphi$ is
    an admissible harmonic function with this flux density. Its boundary trace is
    $f_B$, because $f_B=0$ on $X\setminus B'$, and it vanishes at $\infty$.
    Uniqueness therefore gives
    \begin{align*}
        \varphi_{f_B}=0\quad\text{on }\mathcal T_{X\setminus B'}.
    \end{align*}
    Hence $\varphi_{f_B}$, and therefore its normal derivative, vanishes on
    $X\setminus B'$, as illustrated in Figure
    \ref{fig:partialTXbackslashB'}. Writing this third value as $k_3$, we
    have $k_3=0$, and the three-region decomposition reads
    \begin{align*}
        \partial_n^{(W)}\varphi_{f_B}(x)
        =k_1\boldsymbol1_B(x)
        +k_2\boldsymbol1_{B'\setminus B}(x)
        +k_3\boldsymbol1_{X\setminus B'}(x).
    \end{align*}
    Taking $g=1$ and $h=\varphi_{f_B}$ in Theorem
    \ref{thm:Green's_first_identity_TX}, we obtain
    \begin{align*}
        0
        =&\frac{1}{\kappa_W}
        \int_X\partial_n^{(W)}\varphi_{f_B}(x)\,dm(x)\\
        =&\frac{1}{\kappa_W}
        \int_X\left(k_1\boldsymbol1_B(x)
        +k_2\boldsymbol1_{B'\setminus B}(x)\right)\,dm(x)\\
        =&\frac{1}{\kappa_W}
        \left(k_1m(B)+k_2\bigl(m(B')-m(B)\bigr)\right).
    \end{align*}
    Thus
    \begin{align*}
        \partial_n^{(W)}\varphi_{f_B}(x)
        =&-k_2\frac{m(B')-m(B)}{m(B)}\boldsymbol1_B(x)
        +k_2\boldsymbol1_{B'\setminus B}(x)\\
        =&-k_2m(B')\left(
        \frac{m(B')-m(B)}{m(B)m(B')}\boldsymbol1_B(x)
        -\frac{1}{m(B')}\boldsymbol1_{B'\setminus B}(x)\right)\\
        =&-k_2m(B')\left(\left(\frac{1}{m(B)}
        -\frac{1}{m(B')}\right)\boldsymbol1_B(x)
        -\frac{1}{m(B')}\boldsymbol1_{B'\setminus B}(x)\right)\\
        =&-k_2m(B')f_B(x).
    \end{align*}
    Consequently,
    \begin{align*}
        \mu_B=-k_2m(B').
    \end{align*}
\end{proof}

\begin{figure}[H]
    \centering
    \includegraphics[width=0.7\textwidth]{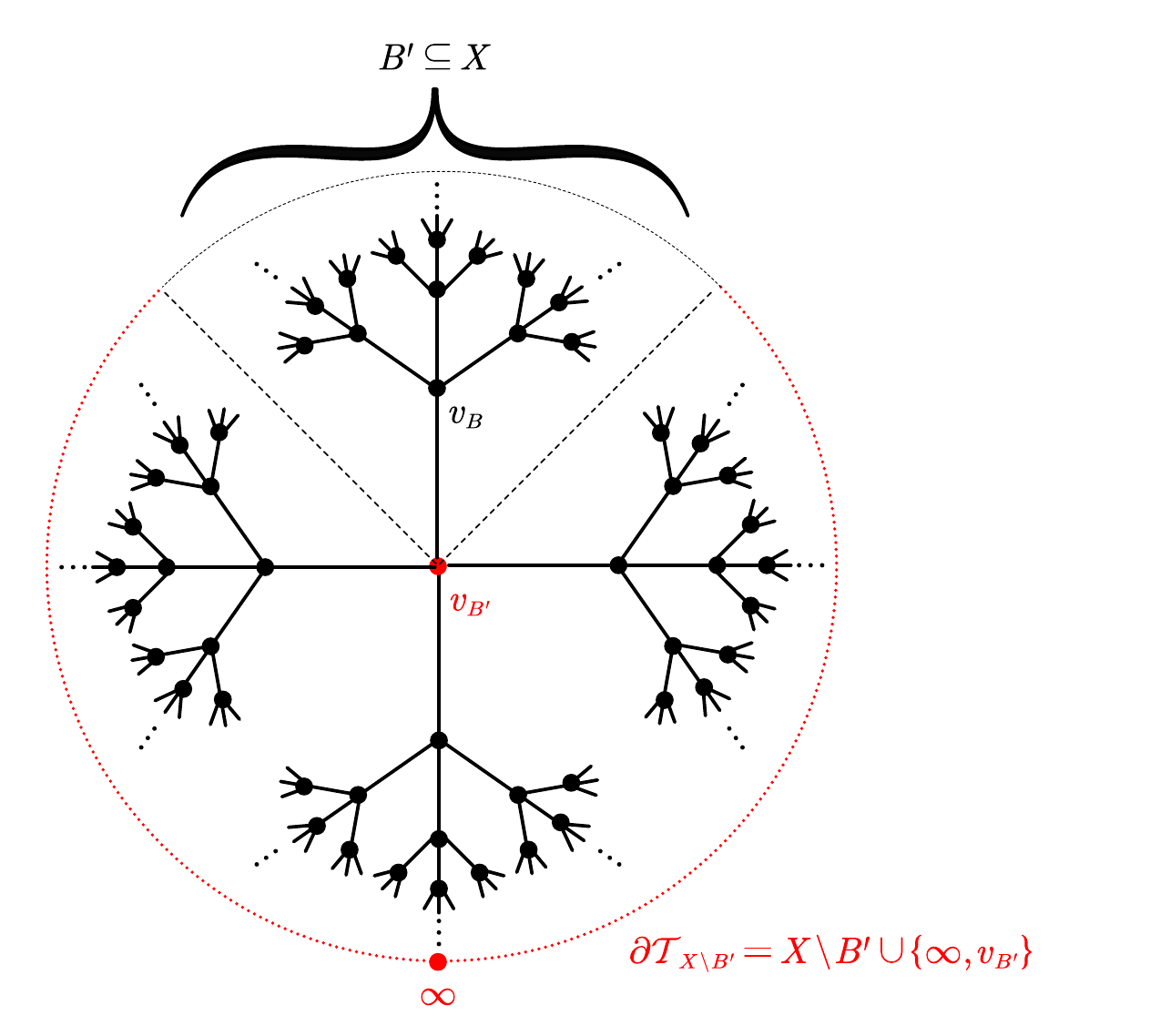}
    \caption{The exterior subtree of $\mathcal T_X$ relative to $B'$}
    \label{fig:partialTXbackslashB'}
\end{figure}

\begin{theorem}
\label{thm:bijective_LC_partialn}
    Under the assumptions on unique solvability, the normal limit, and local
    homogeneity stated above, both $L_C$ and $\partial_n^{(W)}$ are bijections
    from $\mathcal{S}_0(X)$ onto itself.
\end{theorem}
\begin{proof}
    As in the remark following Lemma \ref{lem:Eigenfunctions_of_LC}, we
    delete one element from each family $W_{B'}$ and enumerate the resulting
    algebraic eigenbasis of $\mathcal{S}_0(X)$ as $\{f_{B_i}\}$. All expansions below
    are taken with respect to this
    fixed linearly independent family. For brevity,
    $\partial_n^{(W)}f$ denotes
    $\partial_n^{(W)}\varphi_f$, where $\varphi_f$ is the unique admissible
    extension of $f$.
    \begin{enumerate}
        \item For $L_C$:
        \begin{enumerate}
            \item $L_C:\mathcal{S}_0(X)\to \mathcal{S}_0(X)$: For any
            $f=\sum\limits_{i=1}^Nc_if_{B_i}\in \mathcal{S}_0(X)$, we have
            \begin{align*}
                L_Cf
                =L_C\left(\sum_{i=1}^Nc_if_{B_i}\right)
                =\sum_{i=1}^Nc_i\lambda(B_i')f_{B_i}\in \mathcal{S}_0(X).
            \end{align*}

            \item Injectivity: Suppose that $f\in \mathcal{S}_0(X)$ satisfies
            $L_Cf=0$. Expanding $f$ as a finite sum over the fixed
            linearly independent Haar family,
            $f=\sum\limits_{i=1}^Nc_if_{B_i}$, we have
            \begin{align*}
                \sum_{i=1}^Nc_i\lambda(B_i')f_{B_i}=0.
            \end{align*}
            It follows that $c_i\lambda(B_i')=0$ for every $i$. Since
            $\lambda(B_i')>0$, we conclude that $c_i=0$ for every $i$.
            Consequently, $f=0$.

            \item Surjectivity: For every
            $g=\sum\limits_{i=1}^Md_if_{B_i}\in \mathcal{S}_0(X)$, let
            \begin{align*}
                f:=\sum_{i=1}^M\frac{d_i}{\lambda(B_i')}f_{B_i}
                \in \mathcal{S}_0(X).
            \end{align*}
            Then $L_Cf=g$.
        \end{enumerate}

        \item For $\partial_n^{(W)}$:
        \begin{enumerate}
            \item $\partial_n^{(W)}:\mathcal{S}_0(X)\to \mathcal{S}_0(X)$: For $f_{B_i}$ and
            its extension $\varphi_{f_{B_i}}$, Lemma
            \ref{lem:eigvalue_deformed_derivative_TX} gives
            \begin{align*}
                \partial_n^{(W)}\varphi_{f_{B_i}}
                =\mu_{B_i}f_{B_i}.
            \end{align*}
            Hence, for any
            $f=\sum\limits_{i=1}^Nc_if_{B_i}\in \mathcal{S}_0(X)$,
            linearity of the extension and of the normal derivative gives
            \begin{align*}
                \partial_n^{(W)}f
                =\partial_n^{(W)}\left(\sum_{i=1}^Nc_if_{B_i}\right)
                =\sum_{i=1}^Nc_i\mu_{B_i}f_{B_i}\in \mathcal{S}_0(X).
            \end{align*}

            \item Injectivity: Suppose that $f\in \mathcal{S}_0(X)$ satisfies
            $\partial_n^{(W)}f=0$. By Lemma
            \ref{lem:relation_partialn_and_g} and $\kappa_W>0$, the
            flux density of $\varphi_f$ satisfies
            \begin{align*}
                g_{\varphi_f}=0.
            \end{align*}
            For each $R$, take $g=h=\varphi_f$ in the truncated Green
            identity and use the harmonicity of $\varphi_f$. The boundary
            calculation in the proof of Theorem
            \ref{thm:Green's_first_identity_TX} gives
            \begin{align*}
                \frac{1}{2}\sum_{\substack{v_A,v_B\in\mathcal T_X^R\\v_A\sim v_B}}W(A,B)\left|\varphi_f(v_A)-\varphi_f(v_B)\right|^2=\int_X\varphi_f(v_{B_R(x)})g_{\varphi_f}(x)\,dm(x)=0.
            \end{align*}
            For every edge, both endpoints lie in $\mathcal T_X^R$ once $R$
            is sufficiently large. The nonnegative truncated energies
            therefore increase to the full energy as $R\to+\infty$:
            \begin{align*}
                0
                =&\frac{1}{2}
                \sum_{\substack{v_A,v_B\in\mathcal T_X\\v_A\sim v_B}}
                W(A,B)\left|\varphi_f(v_A)-\varphi_f(v_B)\right|^2.
            \end{align*}
            Since $W(A,B)>0$, the function $\varphi_f$ is constant on the
            connected tree $\mathcal T_X$. Since $\varphi_f(\infty)=0$, this
            constant is zero, and its boundary trace satisfies $f=0$.

            \item Surjectivity: Lemma
            \ref{lem:eigvalue_deformed_derivative_TX} gives
            \begin{align*}
                \partial_n^{(W)}f_{B_i}=\mu_{B_i}f_{B_i}.
            \end{align*}
            If $\mu_{B_i}=0$, then $f_{B_i}$ belongs to the kernel of
            $\partial_n^{(W)}$, contrary to injectivity. Therefore
            \begin{align*}
                \mu_{B_i}\ne0
            \end{align*}
            for every $i$. For every
            $g=\sum\limits_{i=1}^Md_if_{B_i}\in \mathcal{S}_0(X)$, let
            \begin{align*}
                f:=\sum_{i=1}^M\frac{d_i}{\mu_{B_i}}f_{B_i}
                \in \mathcal{S}_0(X).
            \end{align*}
            Then $\partial_n^{(W)}f=g$.
        \end{enumerate}
    \end{enumerate}
\end{proof}

\subsection{Determination of the edge weight}
\begin{lemma}
\label{lem:Green's_formula}
    Let $\varphi$ be an admissible solution of
    \eqref{eq:Extension_problem_TX} with boundary trace $f$, and suppose that
    the hypotheses of Lemma \ref{lem:relation_partialn_and_g} are satisfied.
    Then the following
    cancellation form of Green's formula holds:
    \begin{align*}
        f(x)=G_W\bigl(\partial_n^{(W)}\varphi\bigr)(x),
    \end{align*}
    where, for $q\in \mathcal{S}_0(X)$,
    \begin{align*}
        G_Wq(x)=\frac{1}{\kappa_W}
        \sum_{\substack{B\in\mathcal B\\B\ni x}}
        \frac{1}{W(B,B')}\int_Bq(y)\,dm(y).
    \end{align*}
    If, in addition,
    \begin{align*}
        \sum_{D\supseteq B_{x,y}}\frac{1}{W(D,D')}<+\infty,
        \qquad x\ne y,
    \end{align*}
    then the pointwise kernel representation is valid:
    \begin{align*}
        f(x)=\int_XG_W(x,y)\partial_n^{(W)}\varphi(y)\,dm(y),
    \end{align*}
    where
    \begin{align*}
        G_W(x,y)=\frac{1}{\kappa_W}
        \sum_{\substack{B\in\mathcal B\\B\supseteq B_{x,y}}}
        \frac{1}{W(B,B')}.
    \end{align*}
\end{lemma}
\begin{proof}
    Since
    $\displaystyle\lim_{k\to+\infty}\varphi(v_{B_k})=f(x)$ and
    $\displaystyle\lim_{k\to-\infty}\varphi(v_{B_k})
    =\varphi(\infty)=0$, it follows from \eqref{eq:flux_at_B},
    \eqref{eq:the_measure_recovery_formula}, and Lemma
    \ref{lem:relation_partialn_and_g} that
    \begin{align*}
        f(x)
        =&\sum_{k=-\infty}^{+\infty}
        \bigl(\varphi(v_{B_k})-\varphi(v_{B_{k-1}})\bigr)\\
        =&\sum_{k=-\infty}^{+\infty}
        \frac{\Phi(B_k)}{W(B_k,B_{k-1})}\\
        =&\sum_{\substack{B\in\mathcal B\\B\ni x}}
        \frac{\Phi(B)}{W(B,B')}\\
        =&\sum_{\substack{B\in\mathcal B\\B\ni x}}
        \frac{\int_Bg_\varphi(y)\,dm(y)}{W(B,B')}\\
        =&\frac{1}{\kappa_W}
        \sum_{\substack{B\in\mathcal B\\B\ni x}}
        \frac{1}{W(B,B')}
        \int_B\partial_n^{(W)}\varphi(y)\,dm(y)\\
        =&G_W\bigl(\partial_n^{(W)}\varphi\bigr)(x).
    \end{align*}

    It remains to justify the passage from the cancellation sum to the
    pointwise two-variable kernel. Let $q\in \mathcal{S}_0(X)$. Choose a ball
    $B_{k_0}(x)$ in the chain through $x$ such that
    \begin{align*}
        \{x\}\cup\operatorname{supp}q\subseteq B_{k_0}(x).
    \end{align*}
    For the strict descendants of $B_{k_0}(x)$, the resistance summability
    gives
    \begin{align*}
        \sum_{\substack{D\in\mathcal B\\D\ni x\\
        D\subsetneq B_{k_0}(x)}}
        \frac{1}{W(D,D')}\int_D|q(y)|\,dm(y)\leq\|q\|_\infty
        \sum_{\substack{D\in\mathcal B\\D\ni x\\
        D\subsetneq B_{k_0}(x)}}
        \frac{m(D)}{W(D,D')}<+\infty.
    \end{align*}
    Choose $y_0\in B_{k_0}(x)$ in a maximal sub-ball different from the one
    containing $x$. Then $B_{x,y_0}=B_{k_0}(x)$. Hence, for the ancestors of
    $B_{k_0}(x)$, the additional hypothesis gives
    \begin{align*}
        \sum_{D\supseteq B_{k_0}(x)}\frac{1}{W(D,D')}
        \int_D|q(y)|\,dm(y)\leq\|q\|_1
        \sum_{D\supseteq B_{k_0}(x)}\frac{1}{W(D,D')}<+\infty.
    \end{align*}
    Hence Fubini's theorem applies, and
    \begin{align*}
        G_Wq(x)
        =&\frac{1}{\kappa_W}
        \sum_{\substack{B\in\mathcal B\\B\ni x}}
        \frac{1}{W(B,B')}
        \int_X\boldsymbol1_B(y)q(y)\,dm(y)\\
        =&\frac{1}{\kappa_W}\int_X
        \left(\sum_{\substack{B\in\mathcal B\\B\ni x}}
        \frac{\boldsymbol1_B(y)}{W(B,B')}\right)q(y)\,dm(y)\\
        =&\int_XG_W(x,y)q(y)\,dm(y).
    \end{align*}
    Therefore,
    \begin{align*}
        G_W(x,y)
        =\frac{1}{\kappa_W}
        \sum_{\substack{B\in\mathcal B\\B\ni x}}
        \frac{\boldsymbol1_B(y)}{W(B,B')}
        =\frac{1}{\kappa_W}
        \sum_{\substack{B\in\mathcal B\\B\ni x\\B\ni y}}
        \frac{1}{W(B,B')}
        =\frac{1}{\kappa_W}
        \sum_{\substack{B\in\mathcal B\\B\supseteq B_{x,y}}}
        \frac{1}{W(B,B')}.
    \end{align*}
    Because $m$ is non-atomic, the value assigned to $G_W(x,x)$ does not
    affect the integral. Without the additional summability over ancestors,
    the cancellation formula remains valid, but the positive pointwise
    kernel need not converge; in that case, the interchange above is not
    justified.
\end{proof}

We first reduce the Dirichlet-to-Neumann identity to an identity for the
Green operator. Whenever the
admissible Dirichlet problem is uniquely solvable on $\mathcal{S}_0(X)$ under the
hypotheses of Lemma \ref{lem:relation_partialn_and_g}, Lemma
\ref{lem:Green's_formula} gives
\begin{align*}
    G_W\partial_n^{(W)}=I\quad\text{on }\mathcal{S}_0(X).
\end{align*}
The algebraic Haar-basis decomposition described in the remark following
Lemma \ref{lem:Eigenfunctions_of_LC} shows directly that $L_C$ is bijective
on $\mathcal{S}_0(X)$: its inverse is obtained by dividing each Haar coefficient by
the corresponding positive eigenvalue $\lambda(B')$. Hence
\begin{align*}
    L_C=\partial_n^{(W)}\text{ on }\mathcal{S}_0(X)
    \quad\Longleftrightarrow\quad
    G_W=L_C^{-1}\text{ on }\mathcal{S}_0(X).
\end{align*}
For the forward implication, compose on the right with $L_C^{-1}$. For the
reverse implication, $G_W=L_C^{-1}$ is injective; hence
$G_W\partial_n^{(W)}=G_WL_C$ implies
$\partial_n^{(W)}=L_C$. This equivalence uses only unique solvability and
the preceding Green-operator identities, not the local homogeneity assumed
in the previous subsection.

We now determine the edge weight for which $G_W=L_C^{-1}$ on $\mathcal{S}_0(X)$.

\begin{theorem}
\label{thm:find_out_the_weight}
Assume that
\begin{align*}
    \kappa:=\lim_{k\to+\infty}
    \frac{\lambda(B_k(x))}{C(B_k(x))}\in(0,+\infty)
\end{align*}
is independent of $x$ and that the convergence is locally uniform in $x$.
Assume further that, for each fixed $k$, the series
\begin{align*}
    \sum_{j=k}^{+\infty}
    \frac{m(B_j(x))}{W(B_j(x),B_{j-1}(x))}
\end{align*}
converges and is locally bounded in $x$, and suppose that the coefficient in Lemma
\ref{lem:relation_partialn_and_g} converges locally
uniformly to a constant $\kappa_W\in(0,+\infty)$ independent of $x$. Then
\begin{align*}
    G_W=L_C^{-1}\text{ on }\mathcal{S}_0(X)
    \quad\Longleftrightarrow\quad
    W(B,B')=\frac{1}{\kappa}
    \frac{m(B)\lambda(B)\lambda(B')}{C(B)}.
\end{align*}
For the weight on the right-hand side, $\kappa_W=\kappa$.
\end{theorem}
\begin{proof}
    Suppose first that
    \begin{align*}
        G_W=L_C^{-1}\quad\text{on }\mathcal{S}_0(X).
    \end{align*}
    Fix $B\in\mathcal B$ and denote its immediate super-ball by $B'$. By
    Lemma \ref{lem:Eigenfunctions_of_LC},
    \begin{align*}
        L_Cf_B=\lambda(B')f_B,
    \end{align*}
    and hence
    \begin{align*}
        G_Wf_B(x)=\frac{1}{\lambda(B')}f_B(x).
    \end{align*}
    On the other hand, the cancellation representation in Lemma
    \ref{lem:Green's_formula} gives
    \begin{align*}
        G_Wf_B(x)=\frac{1}{\kappa_W}
        \sum_{\substack{D\in\mathcal B\\D\ni x}}
        \frac{1}{W(D,D')}\int_Df_B(y)\,dm(y).
    \end{align*}
    We evaluate this expression on the three geometric regions appearing in
    the definition of $f_B$.
    \begin{enumerate}
        \item[]\textbf{Case 1.} If $x\notin B'$, every ball containing $x$
        is either disjoint from $B'$ or contains $B'$. In the first case
        $f_B$ vanishes on that ball, while in the second case
        \begin{align*}
            \int_Df_B\,dm=\int_{B'}f_B\,dm=0.
        \end{align*}
        Hence every term in the defining sum is zero, and
        \begin{align*}
            G_Wf_B(x)=0.
        \end{align*}

        \item[]\textbf{Case 2.} Suppose that $x\in B$. Every term with
        $D\supseteq B'$ vanishes because
        \begin{align*}
            \int_Df_B\,dm=\int_{B'}f_B\,dm=0.
        \end{align*}
        Every $D\ni x$ with $D\subsetneq B'$ lies in $B$, because $B$ is
        the maximal proper sub-ball of $B'$ containing $x$. On every such
        $D$, the function $f_B$ is constant and
        \begin{align*}
            \int_Df_B(y)\,dm(y)=f_B(x)m(D).
        \end{align*}
        Therefore,
        \begin{align*}
            G_Wf_B(x)
            =\frac{1}{\kappa_W}
            \sum_{\substack{D\ni x\\D\subsetneq B'}}
            \frac{1}{W(D,D')}\int_Df_B(y)\,dm(y)=\frac{f_B(x)}{\kappa_W}
            \sum_{\substack{D\ni x\\D\subsetneq B'}}
            \frac{m(D)}{W(D,D')}.
        \end{align*}

        \item[]\textbf{Case 3.} Suppose that $x\in B'\setminus B$. Then
        $x$ belongs to a unique maximal sub-ball $A$ of $B'$ different from
        $B$. Every strict descendant $D$ of $B'$ that contains $x$ lies in
        $A$, where $f_B=-1/m(B')$ is constant. Again the terms with
        $D\supseteq B'$ vanish. Therefore
        \begin{align*}
            G_Wf_B(x)
            =\frac{1}{\kappa_W}
            \sum_{\substack{D\ni x\\D\subsetneq B'}}
            \frac{1}{W(D,D')}\int_Df_B(y)\,dm(y)=\frac{f_B(x)}{\kappa_W}
            \sum_{\substack{D\ni x\\D\subsetneq B'}}
            \frac{m(D)}{W(D,D')}.
        \end{align*}
    \end{enumerate}
    Since $f_B(x)\ne0$ for every $x\in B'$, comparing this expression with
    $G_Wf_B=\lambda(B')^{-1}f_B$ gives, for every ball $B'$ and every
    $x\in B'$,
    \begin{align*}
        \frac{1}{\kappa_W}
        \sum_{\substack{D\ni x\\D\subsetneq B'}}
        \frac{m(D)}{W(D,D')}=\frac{1}{\lambda(B')}.
    \end{align*}

    Let $B$ be an arbitrary maximal proper sub-ball of $B'$ and fix $x\in B$.
    Apply the preceding identity first with upper ball $B'$ and then with
    upper ball $B$. The first sum is the second sum together with the single
    term $D=B$, which corresponds to the edge $(B,B')$. Subtracting the two
    identities therefore gives
    \begin{align*}
        \frac{1}{\kappa_W}\frac{m(B)}{W(B,B')}
        =\frac{1}{\lambda(B')}-\frac{1}{\lambda(B)}
        =\frac{\lambda(B)-\lambda(B')}
        {\lambda(B)\lambda(B')}
        =\frac{C(B)}{\lambda(B)\lambda(B')}.
    \end{align*}
    Hence
    \begin{align*}
        W(B,B')=\frac{1}{\kappa_W}
        \frac{m(B)\lambda(B)\lambda(B')}{C(B)}.
    \end{align*}

    It remains to identify $\kappa_W$. Substituting this expression into the
    coefficient in Lemma
    \ref{lem:relation_partialn_and_g} gives
    \begin{align*}
        \kappa_W=&\lim_{k\to+\infty}\frac{W(B_k,B_{k-1})}{m(B_k)}\sum_{j=k}^{+\infty}\frac{m(B_j)}{W(B_j,B_{k-j})}\\
        =&\lim_{k\to+\infty}\frac{1}{\kappa_W}\frac{\lambda(B_k)\lambda(B_{k-1})}{C(B_k)}\sum_{j=k}^{+\infty}\kappa_W\frac{C(B_j)}{\lambda(B_j)\lambda(B_{j-1})}\\
        =&\lim_{k\to+\infty}\frac{\lambda(B_k)\lambda(B_{k-1})}{C(B_k)}\sum_{j=k}^{+\infty}\left(\frac{1}{\lambda(B_{j-1})}-\frac{1}{\lambda(B_{j})}\right)\\
        =&\lim_{k\to+\infty}\frac{\lambda(B_k)}{C(B_k)}\\
        =&\kappa.
    \end{align*}
    where we have used
    $C(B_j)=\lambda(B_j)-\lambda(B_{j-1})$. The same coefficient converges
    to $\kappa_W$ by hypothesis. Therefore $\kappa_W=\kappa$, and the
    asserted necessary form of the weight follows.

    Conversely, suppose that
    \begin{align*}
        W(B,B')=\frac{1}{\kappa}
        \frac{m(B)\lambda(B)\lambda(B')}{C(B)}.
    \end{align*}
    Then
    \begin{align*}
        \frac{m(D)}{W(D,D')}
        =\kappa\frac{C(D)}{\lambda(D)\lambda(D')}=\kappa\left(
        \frac{1}{\lambda(D')}-\frac{1}{\lambda(D)}\right).
    \end{align*}
    Along a chain
    $B'=B_0(x)\supset B_1(x)\supset B_2(x)\supset\cdots\ni x$, the condition
    $\lambda(B_j(x))\to+\infty$ yields
    \begin{align*}
        \frac{1}{\kappa}\sum_{\substack{D\ni x\\D\subsetneq B'}}\frac{m(D)}{W(D,D')}=\sum_{j=1}^{+\infty}\left(\frac{1}{\lambda(B_{j-1})}-\frac{1}{\lambda(B_j)}\right)=\frac{1}{\lambda(B')}.
    \end{align*}
    This also proves the required resistance summability. Moreover,
    \begin{align*}
        \kappa_W=&\lim_{k\to+\infty}\frac{W(B_k,B_{k-1})}{m(B_k)}\sum_{j=k}^{+\infty}\frac{m(B_j)}{W(B_j,B_{k-j})}\\
        =&\lim_{k\to+\infty}\frac{1}{\kappa}\frac{\lambda(B_k)\lambda(B_{k-1})}{C(B_k)}\sum_{j=k}^{+\infty}\kappa\frac{C(B_j)}{\lambda(B_j)\lambda(B_{j-1})}\\
        =&\lim_{k\to+\infty}\frac{\lambda(B_k)\lambda(B_{k-1})}{C(B_k)}\sum_{j=k}^{+\infty}\left(\frac{1}{\lambda(B_{j-1})}-\frac{1}{\lambda(B_{j})}\right)\\
        =&\lim_{k\to+\infty}\frac{\lambda(B_k)}{C(B_k)}\\
        =&\kappa.
    \end{align*}
    Hence $\kappa_W=\kappa$. Applying the three-region calculation above
    now gives
    \begin{align*}
        G_Wf_B(x)=\frac{1}{\lambda(B')}f_B(x)
        =L_C^{-1}f_B(x),
        \qquad x\in X.
    \end{align*}
    The functions $f_B$ span each local Haar space, and the algebraic direct
    sum of these local spaces is $\mathcal{S}_0(X)$, by the remark following Lemma
    \ref{lem:Eigenfunctions_of_LC}. Therefore
    \begin{align*}
        G_W=L_C^{-1}\quad\text{on }\mathcal{S}_0(X).
    \end{align*}
    This establishes sufficiency and completes the proof.
\end{proof}

\begin{theorem}
    For the weight specified in Theorem \ref{thm:find_out_the_weight}, every
    $f\in \mathcal{S}_0(X)$ has a unique admissible solution of
    \eqref{eq:Extension_problem_TX}, and it satisfies
    \begin{align*}
        \partial_n^{(W)}\varphi=L_Cf.
    \end{align*}
\end{theorem}
\begin{proof}
    For $f\in \mathcal{S}_0(X)$, set
    \begin{align*}
        g_\varphi=\frac{1}{\kappa}L_Cf,
    \end{align*}
    and define
    \begin{align*}
        \varphi(v_B)=\sum_{D\supseteq B}
        \frac{1}{W(D,D')}\int_Dg_\varphi(y)\,dm(y).
    \end{align*}
    The algebraic Haar expansion and Lemma
    \ref{lem:Eigenfunctions_of_LC} show that $L_Cf\in \mathcal{S}_0(X)$. Hence the
    flux construction following
    \eqref{eq:the_measure_recovery_formula} gives
    \begin{align*}
        \Delta_W\varphi=0,
        \qquad
        \varphi(\infty)=0,
        \qquad
        \Phi(B)=\int_Bg_\varphi\,dm.
    \end{align*}
    By Theorem \ref{thm:find_out_the_weight},
    $G_W=L_C^{-1}$ and $\kappa_W=\kappa$. Thus the boundary trace is
    \begin{align*}
        \varphi|_{\partial\mathcal T_X}
        =\kappa G_Wg_\varphi
        =G_WL_Cf=f.
    \end{align*}
    To verify explicitly that this is the required $L^2$ trace, choose a
    ball $P$ containing $\operatorname{supp}f\cup
    \operatorname{supp}g_\varphi$. Because $g_\varphi$ is locally constant
    and compactly supported, there is $k_0$ such that, for every
    $k\geq k_0$, it is constant on every ball $B_k(x)$ that meets $P$. Using
    \begin{align*}
        \frac{m(B_j(x))}{W(B_j(x),B_{j-1}(x))}
        =\kappa\left(
        \frac{1}{\lambda(B_{j-1}(x))}
        -\frac{1}{\lambda(B_j(x))}\right),
    \end{align*}
    telescoping the flux increments gives
    \begin{align*}
        f(x)-\varphi(v_{B_k(x)})
        =g_\varphi(x)\sum_{j=k+1}^{+\infty}
        \frac{m(B_j(x))}{W(B_j(x),B_{j-1}(x))}
        =\frac{\kappa g_\varphi(x)}{\lambda(B_k(x))}.
    \end{align*}
    The difference vanishes outside $P$ for such $k$. On $P$, the functions
    $x\mapsto\lambda(B_k(x))^{-1}$ are continuous (indeed, locally constant)
    and decrease pointwise to zero. Dini's theorem therefore gives uniform
    convergence on the compact set $P$. Hence
    \begin{align*}
        \left\|f-\varphi(v_{B_k(\cdot)})\right\|_{L^2(X,m)}
        \leq\kappa
        \sup_{x\in P}\frac{1}{\lambda(B_k(x))}
        \|g_\varphi\|_{L^2(X,m)}\longrightarrow0.
    \end{align*}
    Lemma \ref{lem:relation_partialn_and_g} then yields
    \begin{align*}
        \partial_n^{(W)}\varphi
        =\kappa g_\varphi=L_Cf.
    \end{align*}
    In particular, for $f=f_B$ this gives
    \begin{align*}
        \partial_n^{(W)}\varphi_{f_B}=\lambda(B')f_B.
    \end{align*}
    Thus, for the weight determined in Theorem
    \ref{thm:find_out_the_weight}, the coefficient in Lemma
    \ref{lem:eigvalue_deformed_derivative_TX} is
    $\mu_B=\lambda(B')$.

    For uniqueness, suppose that two admissible solutions have the same
    trace and the same normalization at $\infty$. Their difference has flux
    density $g\in \mathcal{S}_0(X)$ and zero trace. The telescoping identity in Lemma
    \ref{lem:Green's_formula} gives
    \begin{align*}
        0=\kappa G_Wg.
    \end{align*}
    Since $G_W=L_C^{-1}$ is injective on $\mathcal{S}_0(X)$, it follows that $g=0$.
    Thus every edge increment of the difference vanishes. The tree is
    connected, and the value at $\infty$ is zero; hence the two solutions
    coincide.
\end{proof}

\begin{remark}
    If $(X,d,m)=(\mathbb{Q}_p^n,\|\cdot\|_p,d\boldsymbol{x})$, let
    $B_k=\boldsymbol{x}+p^k\mathbb Z_p^n$. Then
    \begin{align*}
        m(B_k)=p^{-nk},
        \qquad
        \lambda(B_k)=p^{(k+1)s},
        \qquad
        C(B_k)=p^{(k+1)s}(1-p^{-s}).
    \end{align*}
    Hence $\kappa=(1-p^{-s})^{-1}$ and
    \begin{align*}
        W(B_k,B_{k-1})
        =\frac{1}{\kappa}
        \frac{m(B_k)\lambda(B_k)\lambda(B_{k-1})}{C(B_k)}
        =(1-p^{-s})
        \frac{p^{-nk}p^{(k+1)s}p^{ks}}
        {p^{(k+1)s}(1-p^{-s})}
        =p^{k(s-n)},
    \end{align*}
    recovering the weighted extension for the Vladimirov--Taibleson operator.

    The pointwise Green representation takes different forms in three regimes. If $B_{\boldsymbol{x},\boldsymbol{y}}=B_k$ and $0<s<n$, then
    \begin{align*}
        \sum_{D\supseteq B_{\boldsymbol{x},\boldsymbol{y}}}\frac{1}{W(D,D')}
        =\sum_{j=-\infty}^{k}p^{j(n-s)}
        =\frac{\|\boldsymbol{x}-\boldsymbol{y}\|_p^{s-n}}{1-p^{s-n}},
    \end{align*}
    and hence, for $q\in \mathcal{S}_0(\mathbb{Q}_p^n)$,
    \begin{align*}
        G_Wq(\boldsymbol{x})=\frac{1-p^{-s}}{1-p^{s-n}}
        \int_{\mathbb Q_p^n}\|\boldsymbol{x}-\boldsymbol{y}\|_p^{s-n}
        q(\boldsymbol{y})\,d\boldsymbol{y}.
    \end{align*}
    When $s=n$, the positive series over ancestors diverges. Removing its additive cutoff constant, which pairs to zero with every $q\in \mathcal{S}_0(\mathbb{Q}_p^n)$, gives the logarithmic representative
    \begin{align*}
        G_Wq(\boldsymbol{x})=-(1-p^{-n})\int_{\mathbb Q_p^n}
        \log_p\|\boldsymbol{x}-\boldsymbol{y}\|_p\,q(\boldsymbol{y})\,d\boldsymbol y.
    \end{align*}
    When $s>n$, the same cancellation prescription gives the homogeneous
    finite-part representative
    \begin{align*}
        G_Wq(\boldsymbol{x})=\frac{1-p^{-s}}{1-p^{s-n}}
        \int_{\mathbb Q_p^n}\|\boldsymbol{x}-\boldsymbol{y}\|_p^{s-n}
        q(\boldsymbol{y})\,d\boldsymbol y,
    \end{align*}
    again modulo an additive constant in the kernel. Thus
    $G_W=(D^s)^{-1}$ on $\mathcal{S}_0(\mathbb{Q}_p^n)$ for every $s>0$, whereas an
    unrenormalized pointwise positive Green kernel exists only when
    $0<s<n$.
\end{remark}

Finally, we prove the energy identity by monotone exhaustion. Let
$f\in \mathcal{S}_0(X)$, and let $\varphi$ be its canonical extension, with
\begin{align*}
    g_\varphi=\frac{1}{\kappa}L_Cf.
\end{align*}
Choose a ball $P$ containing $\operatorname{supp}g_\varphi$. As in the proof
of Theorem \ref{thm:Green's_first_identity_TX}, the edge increment of
$\varphi$ vanishes unless the lower ball of that edge is contained in $P$.
Consequently, only finitely many terms are nonzero in each global
horocyclic truncation $\mathcal T_X^R=\{v_B:h(B)\leq R\}$.

Apply Lemma
\ref{lem:Green's first identity for the truncated tree TXR} to
$g=h=\varphi$. Harmonicity removes the bulk term. For $x\in X$, put
$\varphi_R(x)=\varphi(v_{B_R(x)})$. Flux conservation and
\eqref{eq:the_measure_recovery_formula} give
\begin{align*}
    \frac{1}{2}
    \sum_{\substack{v_A,v_B\in \mathcal{T}_{X}^R\\v_A\sim v_B}}
    W(A,B)|\varphi(v_A)-\varphi(v_B)|^2
    =&\sum_{v_A\in\partial\mathcal T_X^R}
    \varphi(v_A)\sum_{D\in\text{sub}(A)}\Phi(D)\\
    =&\sum_{v_A\in\partial\mathcal T_X^R}
    \varphi(v_A)\Phi(A)\\
    =&\sum_{v_A\in\partial\mathcal T_X^R}
    \varphi(v_A)\int_Ag_\varphi(x)\,dm(x)\\
    =&\int_X\varphi_R(x)g_\varphi(x)\,dm(x).
\end{align*}
Since $\varphi_R\to f$ in $L^2(X,m)$, the right-hand side converges to
\begin{align*}
    \int_Xf(x)g_\varphi(x)\,dm(x)=\frac{1}{\kappa}\int_Xf(x)L_Cf(x)\,dm(x).
\end{align*}
The truncated energies on the left are nonnegative. The sets of interior edges in these truncations increase to the set of all edges as $R\to+\infty$, so the truncated energies increase to the full-tree energy. Consequently,
\begin{align*}
    \frac{1}{2}
    \sum_{\substack{v_A,v_B\in \mathcal{T}_{X}\\v_A\sim v_B}}
    W(A,B)|\varphi(v_A)-\varphi(v_B)|^2
    =\frac{1}{\kappa}\int_X
    f(x)L_Cf(x)\,dm(x).
\end{align*}

\section*{Acknowledgments}
I am grateful to Professor An Huang for his insightful ideas and valuable suggestions, and to Professor Bobo Hua for his continued support throughout this work.
\section*{Data availability}
Data sharing is not applicable to this article as no datasets were generated or analysed during the current study.
\section*{Competing interests} 
The authors have no relevant interests to disclose.

\bibliographystyle{plain}
\bibliography{references}
\noindent\emph{Yaojia Sun, School of Mathematical Sciences, Fudan University,
Shanghai 200433, China}\\
\noindent\emph{Email:}
\href{mailto:26110180043@m.fudan.edu.cn}{\texttt{26110180043@m.fudan.edu.cn}}
\end{document}